\documentclass[12pt]{amsart}
\usepackage{amsmath,amssymb,amsthm,amscd}

\usepackage{graphicx}

\newtheorem{formula}{}[section]
\newtheorem{proposition}[formula]{Proposition}
\newtheorem{corollary}[formula]{Corollary}
\newtheorem{lemma}[formula]{Lemma}
\newtheorem{theorem}[formula]{Theorem}
\newtheorem{conjecture}[formula]{Conjecture}
\newtheorem{question}{Question}
\newtheorem{problem}{Problem}

\theoremstyle{definition}
\newtheorem{definition}[formula]{Definition}
\newtheorem{construction}[formula]{Construction}
\newtheorem{example}[formula]{Example}

\newtheorem{agreement}[formula]{Agreement}

\theoremstyle{remark}
\newtheorem{remark}[formula]{Remark}

\begin{document}

\title[Manifolds defined by non-free actions]{Manifolds realized as orbit spaces \\
of non-free  $\mathbb Z_2^k$-actions on real moment-angle manifolds}
\author[N.Yu.~Erokhovets]{Nikolai~Erokhovets}
\address{Steklov Mathematical Institute of Russian Academy of Sciences, Moscow, Russia \&Department of Mechanics and Mathematics, Lomonosov Moscow State University}
\email{erochovetsn@hotmail.com}

\def\sgn{\mathrm{sgn}\,}
\def\bideg{\mathrm{bideg}\,}
\def\tdeg{\mathrm{tdeg}\,}
\def\sdeg{\mathrm{sdeg}\,}
\def\grad{\mathrm{grad}\,}
\def\ch{\mathrm{ch}\,}
\def\sh{\mathrm{sh}\,}
\def\th{\mathrm{th}\,}

\def\mod{\mathrm{mod}\,}
\def\In{\mathrm{In}\,}
\def\Im{\mathrm{Im}\,}
\def\Ker{\mathrm{Ker}\,}
\def\Hom{\mathrm{Hom}\,}
\def\Tor{\mathrm{Tor}\,}
\def\rk{\mathrm{rk}\,}
\def\codim{\mathrm{codim}\,}

\def\ko{{\mathbf k}}
\def\sk{\mathrm{sk}\,}
\def\RC{\mathrm{RC}\,}
\def\gr{\mathrm{gr}\,}

\def\R{{\mathbb R}}
\def\C{{\mathbb C}}
\def\Z{{\mathbb Z}}
\def\A{{\mathcal A}}
\def\B{{\mathcal B}}
\def\K{{\mathcal K}}
\def\M{{\mathcal M}}
\def\N{{\mathcal N}}
\def\E{{\mathcal E}}
\def\G{{\mathcal G}}
\def\D{{\mathcal D}}
\def\F{{\mathcal F}}
\def\L{{\mathcal L}}
\def\V{{\mathcal V}}
\def\H{{\mathcal H}}

\thanks{This work was supported by the Russian Science Foundation under grant no. 23-11-00143, https://rscf.ru/en/project/23-11-00143/ }



\subjclass[2010]{
57S12, 
57S17, 
57S25, 
52B05, 
52B10, 
52B70, 
57R18, 
57R91
}

\keywords{Non-free action of a finite group, convex polytope, real moment-angle manifold, hyperelliptic manifold,
rational homology sphere, Hamiltonian cycle}

\begin{abstract}
We consider (non-necessarily free) actions of subgroups $H\subset \mathbb Z_2^m$ on the~real
moment-angle manifold $\mathbb R\mathcal{Z}_P$ corresponding to a simple convex $n$ polytope $P$
with $m$ facets.  
The criterion when the orbit space $\mathbb R\mathcal{Z}_P/H$ 
is a~topological manifold (perhaps with a boundary) can be extracted from  
results by M.A.~Mikhailova and C.~Lange.
For any dimension $n$ we construct series of manifolds $\mathbb R\mathcal{Z}_P/H$ 
homeomorphic to $S^n$ and series of manifolds $M^n=\mathbb R\mathcal{Z}_P/H$ 
admitting a hyperelliptic involution $\tau\in\mathbb Z_2^m/H$, that is~an~involution $\tau$ 
such that $M^n/\langle\tau\rangle$ is~homeomorphic to~$S^n$. 
For~any~simple~$3$-polytope $P$ we classify all subgroups $H\subset\mathbb Z_2^m$ such that $\mathbb R\mathcal{Z}_P/H$
is homeomorphic to $S^3$. For~any~simple~$3$-polytope $P$ 
and any~subgroup $H\subset\mathbb Z_2^m$ 
we classify all hyperelliptic involutions $\tau\in\mathbb Z_2^m/H$ acting on~$\mathbb R\mathcal{Z}_P/H$.  
As a~corollary we obtain that a~$3$-dimensional small cover has $3$ hyperelliptic involutions in $\mathbb Z_2^3$
if and only if it is a~rational homology $3$-sphere and  if and only if it correspond to a triple of
Hamiltonian cycles such that each edge of the polytope belongs to exactly two of them.
\end{abstract}
\maketitle
\tableofcontents
\setcounter{section}{0}

\section*{Introduction}
Toric topology (see \cite{BP15,DJ91}) assigns to each $n$-dimensional simple convex polytope $P$ with $m$ facets  
$F_1,\dots,F_m$ an~$n$-dimensional {\it real moment-angle manifold} $\mathbb{R} \mathcal{Z}_P$ with 
an action of~a~finite group $\mathbb Z_2^m$ and an $(n+m)$-dimensional {\it moment-angle manifold} $\mathcal{Z}_P$
with an action of~a~compact torus $\mathbb T^m$ such that 
$\mathbb{R} \mathcal{Z}_P/\mathbb Z_2^m=\mathcal{Z}_P/\mathbb T^m=P$ and the equivariant topology of these spaces
depends only on combinatorics of $P$. This construction allows one to build large families of~manifolds for 
which deep mathematical results can be proved in a more efficient and explicit form. For example, the problem 
of classification of $3$-dimensional manifolds and  $6$-dimensional  simply-connected manifold by their algebraic topology
invariants can be explicitly solved for~the~large families of small covers and quasitoric manifolds over $3$-dimensional 
right-angled hyperbolic polytopes \cite{BEMPP17}. The Thurston's problem of existence of~a~geometric decomposition 
of~any~orientable $3$-manifold was finally solved by G.~Perelman. For all $3$-dimensional manifolds obtained as orbit
spaces of free actions of subgroups in~$\mathbb Z_2^m$ on~$\mathbb{R} \mathcal{Z}_P$ this decomposition 
can be~described explicitly and constructively \cite{E22M}.

In this paper we consider the specification of the following general question to the case of~real moment-angle manifolds
and subgroups $H\subset\mathbb Z_2^m$:  
\begin{question}\label{Q1}
When is the orbit space $M/G$ of~a~smooth~action of~a~finite group $G$ on~a~smooth manifold $M$ a~topological manifold 
(perhaps with a boundary)?
\end{question} 
For manifolds $\mathbb{R} \mathcal{Z}_P/H$ we consider the following questions.
\begin{question}\label{Q2}
When is  $\mathbb{R} \mathcal{Z}_P/H$ homeomorphic to $S^n$?
\end{question}
\begin{question}\label{Q3}
To classify all {\it hyperelliptic involutions} in the group $\mathbb Z_2^m/H$ acting on the~manifold $\mathbb{R} \mathcal{Z}_P/H$,
that is involutions with the orbit space homeomorphic to $S^n$.
\end{question}
\begin{question}\label{Q4}
When is $\mathbb{R} \mathcal{Z}_P/H$ a manifold with the same rational homology as $S^n$?
\end{question}

The exhaustive answer to {\bf Question \ref{Q1}} was obtained in the works by~M.A.~Mikhailova and C.~Lange \cite{M85,LM16,L19}.
For~a~finite~abelian group $G$ the space $M/G$ is a topological manifold if and only if {\it for any point $x\in M$ the subgroup in 
$O(n)$ corresponding to the action of~the~stabilizer $G_x$ on~the~tangent space $T_xM$ with the invariant
scalar product is generated by~reflexions and rotations, where 
the presence of a reflexion indicates the presence of~a~boundary in~the~manifold}.
In our particular case in Theorem \ref{th:ZHM} we give an effective explicit answer in~terms of~the~polytope and the matrix
defining a subgroup $H\subset \mathbb Z_2^m$ 
and its short proof not based on results by Mikhailova and Lange. Namely, a~subgroup $H$ of
rank $m-r$ is defined by {\it a vector-coloring}
of rank $r$, that is a~mapping $\Lambda\colon\{F_1,\dots,F_m\}\to\mathbb Z_2^r$ such that 
$\langle\Lambda_1,\dots,\Lambda_r\rangle=\mathbb Z_2^r$. Usually in toric topology one considers
freely acting subgroups. This is equivalent to the fact that the~coloring is {\it linearly independent}, 
that is~the~vectors $\Lambda_{i_1}$, $\dots$, $\Lambda_{i_k}$ are linearly independent if 
$F_{i_1}\cap\dots\cap F_{i_k}\ne \varnothing$. In this case the~orbit space $N(P,\Lambda)=\mathbb R\mathcal{Z}_P/H$ 
is automatically a~(smooth) manifold.

In the general case $N(P,\Lambda)$ is a pseudomanifold, possibly with a boundary, where the boundary
is glued of facets $F_i$ with $\Lambda_i=\boldsymbol{0}$. We prove that {\it $N(P,\Lambda)$
is a topological manifold if and only if for any collection of facets $F_{i_1}\cap\dots\cap F_{i_k}\ne \varnothing$ such that
$F_{i_1}\cap\dots\cap F_{i_k}\ne \varnothing$ different nonzero vectors among $\Lambda_{i_1}$, $\dots$, $\Lambda_{i_k}$ 
are linearly independent}.

We prove (Corollary \ref{cororc}) that the pseudomanifold $N(P,\Lambda)$ is closed and orientable if and only if
all the vectors $\Lambda_1$, $\dots$, $\Lambda_m$ in $\mathbb Z_2^r$ lie in an~affine hyperplane 
$\boldsymbol{c}\boldsymbol{x}=1$ not containing $\boldsymbol{0}$ 
(this generalizes  the sufficient condition of orientability of small covers over right-angled $3$-polytopes \cite[Lemma 2]{V87}, 
the criterion of orientability of small covers of any dimension  \cite[Theorem 1.7]{NN05}  and 
manifolds defined by linearly independent colorings of right-angled polytopes \cite[Lemma 2.4]{KMT15}). 
We call such colorings {\it affine colorings} of rank $r-1$ and denote them $\lambda$.
In some coordinate system $\Lambda_i=(1,\lambda_i)$.

A {\it coloring} $c\colon\{F_1,\dots, F_m\}\to\{1,\dots,r\}$ defines a complex $\mathcal{C}(P,c)$ with facets $G_j$
the~connected components of unions $\bigcup_{c(F_i)=const}F_i$ corresponding to the same color 
and faces the connected components of~intersections of~facets $G_j$. 
The complexes $\mathcal{C}(P,c_P)$ and $\mathcal{C}(Q,c_Q)$ are {\it equivalent} 
($\mathcal{C}(P,c_P)\simeq \mathcal{C}(Q,c_Q)$) if there is a homeomorphism
$P\to Q$ mapping bijectively facets of the first complex to facets of the second. 
In Corollary \ref{cor:dnc} we prove that any two colorings of~the~simplex 
$\Delta^n$ in $r$ colors produce equivalent complexes. We denote this equivalence class 
$\mathcal{C}(n,r)$. It turns out that any affine coloring $\lambda$ of rank $r$ of a polytope $P$ with 
$\mathcal{C}(P,\lambda)\simeq \mathcal{C}(n,r+1)$ produces a~sphere $N(P,\lambda)\simeq S^n$ 
(see Construction \ref{con:spn}).
Our main result concerning {\bf Question \ref{Q2}} is that in dimension $n=3$ this construction  exhausts all $3$-spheres
among $N(P,\lambda)$ (Theorem \ref{th:NP3sp}). The $1$-skeleton $\mathcal{C}^1(3,1)$ is empty,
$\mathcal{C}^1(3,2)$ is a~circle without vertices, $\mathcal{C}^1(3,3)$ is a theta-graph -- a graph with two vertices
connected by three multiple edges, and  $\mathcal{C}^1(3,4)$ is the complete graph $K_4$. Thus, for a $3$-polytope $P$
subgroups in $\mathbb Z_2^m$ producing spheres $\mathbb{R}\mathcal{Z}_P/H$ bijectively correspond 
to the empty set, simple cycles, theta-subgraphs and $K_4$-subgraphs in the $1$-skeleton of $P$. 

{\bf Question \ref{Q3}} is motivated by papers \cite{M90, VM99M, VM99S2} by 
A.D.~Mednykh and A.Yu.~Vesnin who constructed examples of hyperelliptic $3$-manifolds
with geometric structures modelled on five of eight 
Thurston's geometries: $\mathbb R^3$, $\mathbb H^3$, $\mathbb S^3$, $\mathbb H^2\times\mathbb R$, 
and $\mathbb S^2\times\mathbb R$. Each example was built using a right-angled $3$-polytope $P$ equipped with 
a~Hamiltonian cycle, a~Hamiltonian theta-subgraph, or a~Hamiltonian $K_4$-subgraph, where a~subgraph 
is~Hamiltonian if it contains all vertices of $P$.  
We call an involution $\tau\in \mathbb Z_2^m/H$ acting on~the manifold 
$N(P,\lambda)$ defined by an~affine coloring of rank $r$
{\it special} if the complex $\mathcal{C}(P,\lambda_\tau)$ corresponding
to~the~orbit space $N(P,\lambda)/\langle\tau\rangle$ is equivalent to~$\mathcal{C}(n,r)$.
By~Construction \ref{con:spn} any special involution is hyperelliptic. We introduce Construction \ref{con:shm} 
producing any special hyperelliptic manifold from a~coloring $c\colon \{F_1,\dots,F_m\}\to \{1,\dots, r\}$
such that $\mathcal{C}(P,c)\simeq \mathcal{C}(n,r)$ and a $0/1$-coloring $\chi\colon \{F_1,\dots, F_m\}\to\{0,1\}$.
In Theorem \ref{th:shmclass} we classify all special hyperelliptic involutions $\tau\in\mathbb Z_2^m/H$.
For $n=3$ Theorem \ref{th:NP3sp} implies that any hyperelliptic 
involution in $\mathbb Z_2^m/H$ is special. 
Our main result concerning {\bf Question \ref{Q3}} is~the~classification
of all hyperelliptic involutions in $\mathbb Z_2^m/H$ for $n=3$. In particular,
any Hamiltonian empty set ($r=1$), cycle ($r=2$), theta-subgraph $(r=3)$ or $K_4$-subgraph  ($r=4$) 
$\Gamma$ in $\mathcal{C}^1(P,c)$ induces 
an affine coloring $\lambda_{\Gamma}$ of~rank $r$ by~the~following rule. The 
facets of $P$ lying in the same facet $G_i$ of $\Gamma$ can be colored in two colors
in~such a~way that adjacent facets have different colors. 
Assign to one color the point $\boldsymbol{a}_i$ and to~the~other color $\boldsymbol{b}_i$, where
the points $\boldsymbol{a}_1$, $\boldsymbol{a}_2$, $\dots$, $\boldsymbol{a}_r$, $\boldsymbol{b}_1$ are affinely 
independent and the vector $\tau=\boldsymbol{a}_i+\boldsymbol{b}_i$ does not depend on $i$. We 
obtain an affine coloring $\lambda_\Gamma$ and the hyperelliptic involution $\tau$ on~$N(P,\lambda_\Gamma)$
{\it induced} by $\Gamma$. In Theorem \ref{th:hypHam} we prove that {\it for $n=3$ 
hyperelliptic involutions in $\mathbb Z_2^m/H(\lambda)$
bijectively correspond to~Hamiltonian subgraphs of the above type inducing $\lambda$}. 
Also in Theorem \ref{th:hyperell} for $n=3$ we classify all pairs $(P,\lambda)$ admitting more than one hyperelliptic involution.
In particular, $3$-dimensional small covers $N(P,\Lambda)$ with three hyperelliptic involutions correspond 
to~triples of~Hamiltonian cycles on~a~simple $3$-polytope $P$ such that any edge of~$P$ belongs
to exactly two cycles. 

To study {\bf Question \ref{Q4}} we use the description of the cohomology $H^*(N(P,\Lambda),\mathbb Q)$
obtained  by A.~Suciu and A.~Trevisan \cite{ST12,T12}, and S.~Choi and H.~Park \cite{CP17}. On~the~base
of~this description in~Proposition \ref{pr:rhs} we describe all $3$-dimensional rational homology $3$-spheres among 
manifolds $N(P,\lambda)$. Namely 
for $n=3$ {\it the~manifold $N(P,\lambda)$ corresponding to an affine coloring of rank $r$ is~a~rational
homology sphere if and only if for any affine hyperplane $\pi$ in $\mathbb Z_2^r$ passing through a fixed 
point $\boldsymbol{p}\in\mathbb Z_2^r$ the~union $\bigcup_{\lambda_i\in \pi}F_i$ is a disk}. 
In particular, a $3$-dimensional small cover is~a~rational homology $3$-sphere if and only if the group $\mathbb Z_2^3$
canonically acting on it contains three hyperelliptic involutions. 
In~Example~\ref{ex:RHSGeom} we build rational homology $3$-spheres $N(P,\lambda)$ 
with geometric structures 
modelled on $\mathbb S^3$, $\mathbb S^2\times\mathbb R$, $\mathbb R^3$, 
$\mathbb H^2\times \mathbb R$, and $\mathbb H^3$.
Proposition \ref{pr:rhs} is a refinement of a description of rational homology $3$-spheres over right-angled polytopes
in $\mathbb S^3$, $\mathbb R^3$ and $\mathbb H^3$ used in \cite[Corollary 7.9]{FKR23} 
to build an infinite family of arithmetic hyperbolic rational homology $3$-spheres 
that are totally geodesic boundaries of compact hyperbolic $4$-manifolds, and in \cite[Proposition 3.1]{FKS21} 
to detect the  Hantzsche-Wendt manifold among  manifolds defined by linearly independent colorings of the $3$-cube.
(It is equivalent to the connectivity of the full subcomplex $K_{\omega}$ of the boundary $K=\partial P^*$ 
of the dual polytope $P^*$ for each subset $\omega=\{i\colon \lambda_i\in \pi\}$ corresponding to an affine hyperplane $\pi$.)

The paper is organized as follows.

In Section \ref{Sec:RMAM} we give main definitions and basic facts about real moment-angle manifolds
$\mathbb{R}\mathcal{Z}_P$ and their factor spaces $N(P,\Lambda)$. In particular, in Proposition \ref{orprop} and Corollary 
\ref{cororc} we give the criterion when the pseudomanifold $N(P,\Lambda)$ is closed and orientable. 

In Section \ref{sec:CPc} we describe complexes $\mathcal{C}(P,c)$ corresponding to colorings of facets of $P$
and their properties. In particular, in Proposition \ref{prop:dnr} and Corollary \ref{cor:dnc} we prove 
that all colorings of~facets of~the~simplex $\Delta^n$ in $r$ colors produce equivalent complexes.

In Sections \ref{sec:wecL} and \ref{sec:acol} we describe the weakly equivariant classification
of spaces $N(P,\Lambda)$ defined by vector-colorings and $N(P,\lambda)$ defined by affine colorings.

In Section \ref{sec:NPman} we give the~criterion when $N(P,\Lambda)$ is a~topological manifold
(Theorem \ref{th:ZHM}) and give a~Construction \ref{con:spn} of~spheres  $N(P,\Lambda)$. In particular,
in Example \ref{ex:spnG} for any face $G\subset P$ of~codimension $k$ we build a subgroup $H_G\subset\mathbb Z_2^m$
of codimension $k+1$ such that $\mathbb{R}\mathcal{Z}_P/H_G\simeq S^n$. For a vertex of the product 
$\Delta^{n_1}\times\dots\times\Delta^{n_k}$ this gives an action of $\mathbb Z_2^{k-1}$ on $S^{n_1}\times\dots\times S^{n_k}$
with the orbit space $S^{n_1+\dots+n_k}$ build by Dmitry Gugnin in \cite{G19}.

In Section \ref{sec:mantorus} we give a sufficient condition for the space $\mathcal{Z}_P/H$
to be a closed topological manifold (Proposition \ref{prop:ta}). This condition is similar to Theorem \ref{th:ZHM}
and can be also extracted from the general theory developed in \cite{S09, AGo24}. 
Namely, {\it if a subgroup $H\subset T^m$ is defined by an integer vector-coloring 
$\Lambda\colon \{F_1,\dots,F_m\}\to \mathbb Z^r\setminus\{\boldsymbol{0}\}$ 
such that $\langle\Lambda_1,\dots\Lambda_m\rangle=\mathbb Z^r$ and 
for any vertex  $v=F_{i_1}\cap\dots\cap F_{i_n}$ all the different vectors among $\{\Lambda_{i_1},\dots,\Lambda_{i_n}\}$ 
form a~part of~a~basis in~$\mathbb Z^r$, then $\mathcal{Z}_P/H$ is a~closed topological $(n+r)$-manifold}.
In Proposition \ref{prop:spn+r} we give a~sufficient condition for $\mathcal{Z}_P/H$ to be homeomorphic to a sphere.
As an application in~Example \ref{ex:AGU} we build an action of $\mathbb T^{k-1}$ on $S^{n_1+1}\times\dots\times S^{n_k+1}$
with the orbit space $S^{n_1+\dots+n_k+1}$ constructed in~\cite{AGu23}.

In Section \ref{sec:bospx} we describe combinatorial properties of boolean simplicial prisms
important for~a~construction of hyperelliptic manifolds.

In Section \ref{sec:shm} we give Construction \ref{con:shm} of special hyperelliptic manifolds $N(P,\lambda)$
with a hyperelliptic involution $\tau\in \mathbb Z_2^m/H(\lambda)$ such that $\mathcal{C}(P,\lambda_\tau)\simeq \mathcal{C}(n,r)$. 
In~Theorem \ref{th:shmclass} for these manifolds we classify all special hyperelliptic involutions  
$\tau\in \mathbb Z_2^m/H(\lambda)$.

In Section \ref{sec:CPc3} we give basic facts from the graph theory and theory of $3$-polytopes
and in~Theorem \ref{th:CPc3} we prove that {\it complexes $\mathcal{C}(P,c)$ corresponding to $3$-polytopes $P$ are 
exactly subdivisions of~the~$2$-sphere arising from disjoint unions (perhaps empty) of simple curves and 
connected $3$-valent graphs without bridges}.  

In Section \ref{Sec:3sph} we prove that {\it for an~affine coloring $\lambda$ of rank $r$ of a~simple $3$-polytope $P$ 
the space $N(P,\lambda)$ is  homeomorphic to $S^3$ if and only if  $\mathcal{C}(P,\lambda)$ is~equivalent 
to~$\mathcal{C}(3,r+1)$} (Theorem \ref{th:NP3sp}).

In Section \ref{sec:3hm} for~an~affine coloring $\lambda$  of~a~simple $3$-polytope $P$ 
we classify all hyperelliptic involutions in $\mathbb Z_2^m/H$ acting on $N(P,\lambda)$ (Theorems
\ref{th:hypHam} and \ref{th:hyperell}).

In Section \ref{Sec:RHS} we give a criterion when the space $N(P,\lambda)$
is a~rational homology $3$-sphere (Proposition \ref{pr:rhs}) and consider examples of such spaces.

In Section \ref{sec:3Ham} we gather known information on simple $3$-polytopes 
admitting three consistent Hamiltonian cycles and build examples of such polytopes
and also of polytopes that do not have such a property.

\section{Real moment-angle manifolds and their factor  spaces}\label{Sec:RMAM}
For an introduction to the polytope theory we recommend the books \cite{Z95} and \cite{Gb03}.  
In~this paper by a {\it polytope} we call an $n$-dimensional combinatorial convex polytope. 
Sometimes we implicitly use  its geometric realization in $\mathbb R^n$ and sometimes we use it explicitly. 
In the~latter case we call the polytope {\it geometric}.
A polytope is {\it simple}, if any its vertex is contained in exactly $n=\dim P$ facets.
Let $\{F_1,\dots,F_m\}$ be the set of all the facets, and 
$\mathbb Z_2=\Z/2\Z$.
\begin{definition} 
For each geometric simple $n$-polytope $P$ one can associate an $n$-dimensional
{\it real moment-angle manifold}:
$$
\mathbb R\mathcal{Z}_P=P\times \mathbb Z_2^m/\sim, 
\text{ where }(p,a)\sim(q,b)\text{ if and only if }p=q \text{ and }a-b\in\langle \boldsymbol{e}_i\colon p\in F_i\rangle,
$$ 
and $\boldsymbol{e}_1,\dots, \boldsymbol{e}_m$ is the standard basis in $\mathbb Z_2^m$.
\end{definition}
There is a natural action of $\mathbb Z_2^m$
on $\mathbb R\mathcal{Z}_P$ induced from the action on the second factor. 
We~have $\mathbb R\mathcal{Z}_P/\mathbb Z_2^m=P$. The space $\mathcal{Z}_P$ was introduced in \cite{DJ91}. 
It can be showed that it has a structure of a~smooth manifold such that the action of $\mathbb Z_2^m$ is smooth (see \cite{BP15}).

It is convenient to imagine  $\mathbb R\mathcal{Z}_P$ as a space glued from copies of the polytope $P$
 along facets. If we fix an orientation on $P\times 0$, then define on the polytope $P\times a$ the same orientation, 
 if $a$ has an even number of unit coordinates, and the opposite orientation, in the other case.  
 A polytope $P\times a$ is glued to the polytope $P\times (a+e_i)$ along the facet $F_i$. At each vertex 
 the polytopes are arranged as~coordinate orthants in $\mathbb R^n$, at~each edge -- as~the~orthants at~a~coordinate axis,
 and at~face of~dimension $i$ -- as~the~orthants at~an~$i$-dimensional coordinate subspace. Therefore,
 $\mathbb R\mathcal{Z}_P$ has a natural structure of an oriented piecewise linear manifold.  
The actions of basis vectors $e_i$ can be viewed as
reflections in facets of the polytope. In particular, it changes the orientation. The following fact
is straightforward from the definition.
\begin{lemma}
The element $x=(x_1,\dots, x_m)\in \mathbb Z_2^m$ preserves the orientation of $\mathbb R\mathcal{Z}_P$ 
if and only if it has an even number of nonzero coordinates. In other words, if $x_1+\dots+x_m=0$.
\end{lemma}
\begin{definition}
We will denote by $H_0$ the subgroup of $\mathbb Z_2^m$ consisting of all the orientation preserving elements. 
\end{definition}
We consider manifolds obtained as orbit spaces of (not necessarily free) actions of subgroups $H\subset\mathbb Z_2^m$ on 
$\mathbb R\mathcal{Z}_P$. Each subgroup of $\mathbb Z_2^m$ is isomorphic to $\mathbb Z_2^{m-r}$ for some $r$ and may 
be described as a kernel $H(\Lambda)={\rm Ker}\,\Lambda$ 
of a an epimorphism $\Lambda\colon\mathbb Z_2^m\to\mathbb Z_2^r$. 
Such a mapping is uniquely defined by the images $\Lambda_i\in\mathbb Z_2^r$ of all the vectors 
$e_i\in\mathbb Z_2^m$ corresponding to facets $F_i$, $i=1$,\dots, $m$. 
It can be shown  (see \cite{DJ91, BP15}) that the action of the subgroup $H(\Lambda)\subset \mathbb Z_2^m$ 
on $\mathbb R\mathcal{Z}_P$ is free
if and only if 
$$
(*)\quad\text{for any face $F_{i_1}\cap\dots\cap F_{i_k}\ne\varnothing$ of $P$ the vectors $\Lambda_{i_1}$, $\dots$,
$\Lambda_{i_k}$ are linearly independent.} 
$$
Since any face of $P$ contains a vertex, it is sufficient to check this condition only for vertices.
\begin{definition}
We call a mapping $\Lambda\colon \{F_1,\dots,F_m\}\to \mathbb Z_2^r$  such that the images $\Lambda_j$
of the facets $F_j$ span $\mathbb Z_2^r$  {\it a (general) vector-coloring} of rank $r$. If, additionally,
the condition (*) holds we call such a vector-coloring {\it linearly independent}.
\end{definition}
\begin{remark}
In \cite{E22M} by definition any vector-coloring is assumed to be linearly  independent. 
\end{remark}
\begin{remark}
Sometimes we call by a vector-coloring of rank $r$ a mapping $\Lambda\colon \{F_1,\dots,F_m\}\to \mathbb Z_2^s$
such that $\dim\,\langle\Lambda_1,\dots,\Lambda_s\rangle=r$. 
\end{remark}
Denote by $N(P,\Lambda)$ the orbit space $\mathbb R\mathcal{Z}_P/H(\Lambda)$ of the action of the subgroup $H(\Lambda)$
corresponding to a vector-coloring $\Lambda$ of rank $r$.  If we identify $\mathbb Z_2^m/{\rm Ker}\,\Lambda$ with $\mathbb Z_2^r$ via the mapping $\Lambda$, then 
$$
N(P,\Lambda)=P\times \mathbb Z_2^r/\sim, \text{ where }(p,a)\sim(q,b)\text{ if and only if }p=q \text{ and }a-b\in\langle \Lambda_i\colon p\in F_i\rangle.
$$ 
In particular, the space $N(P,\Lambda)$ is glued from $2^r$ copies of $P$. It has
a canonical action of $\mathbb Z_2^r$ such that the orbit space is $P$.
\begin{definition}
We call $N(P,\Lambda)$ {\it a space defined by a vector-coloring} $\Lambda$.
\end{definition}
\begin{example}
For $r=m$ and the~mapping $E(F_i)=\boldsymbol{e}_i$,
where $\boldsymbol{e}_1$, $\dots$, $\boldsymbol{e}_m$ is the~standard basis in~$\mathbb Z_2^m$, 
the space $N(P,E)$ is $\mathbb R\mathcal{Z}_P$. 

For $r=n$ a linearly independent vector-coloring is called a {\it characteristic mapping}, and 
the space $N(P,\Lambda)$ is called a {\it small cover} over the polytope $P$.

For $r=1$  and the constant mapping $\Lambda_i=1$ the subgroup $H(\Lambda)$ is the subgroup
$H_0$ consisting  of all the elements preserving the orientation of $\mathbb R\mathcal{Z}_P$. The space  $N(P,\Lambda)$
is glued of two copies of $P$ along the common boundary. It is homeomorphic to $S^n$.
\end{example}

\begin{proposition}\label{propH12}
For vector-colorings $\Lambda_1$ and $\Lambda_2$ of ranks $r_1$ and $r_2$ of a polytope $P$ we have $H(\Lambda_1)\subset H(\Lambda_2)$ if and only if there is an epimorphism $\Pi\colon\mathbb Z_2^{r_1}\to \mathbb Z_2^{r_2}$ 
such that $\Pi \circ \Lambda_1=\Lambda_2$. In this case $N(P,\Lambda_2)=N(P,\Lambda_1)/{\rm Ker\,}\Pi$,
where ${\rm Ker\,}\Pi\simeq H(\Lambda_2)/H(\Lambda_1)$. In particular, if the action of ${\rm Ker\,}\Pi$ 
is free, then there is a covering  $N(P,\Lambda_1)\to N(P,\Lambda_2)$ with the fiber $H(\Lambda_2)/H(\Lambda_1)$.
\end{proposition}
\begin{remark}
For $r_1=r_2+1$ in \cite[Section 7.2]{FKR23} the vector-coloring $\Lambda_1$ is called 
an~{\it extension}~of~$\Lambda_2$.
\end{remark}
\begin{proof}
We have $H(\Lambda_1)\subset H(\Lambda_2)$ if and only if each row of the matrix $\Lambda_2$ with columns $\Lambda_{2,i}$
is a linear combination of rows of $\Lambda_1$. This is equivalent to the existence of a surjection 
$\Pi\colon\mathbb Z_2^{r_1}\to \mathbb Z_2^{r_2}$ such that 
$\Pi(\Lambda_{1,i})=\Lambda_{2,i}$ for all $i=1$, $\dots$, $m$.
\end{proof}

\begin{corollary}
We have $H(\Lambda_1)=H(\Lambda_2)$ if and only if there is an~isomorphism $\Pi\colon \mathbb Z_2^{r_1}\to\mathbb Z_2^{r_2}$ such that $\Lambda_2=\Pi\circ\Lambda_1$.
\end{corollary}
\begin{corollary}\label{corh'}
Let $\Lambda$ be a vector-coloring of rank $r$ of a simple polytope $P$. Then there is a~bijection between 
the~subgroups $H'\subset \mathbb Z_2^r$ and the~subgroups in $\mathbb Z_2^m$ containing $H(\Lambda)$  given by 
the~correspondance $H'={\rm Ker\,}\Pi\to {\rm Ker\,}\Pi\circ\Lambda$ (or by the isomorphism
$\mathbb Z_2^r\simeq \mathbb Z_2^m/{\rm Ker}\,\Lambda$). 
Moreover, $N(P,\Lambda)/H'\simeq N(P,\Pi\circ\Lambda)$.
\end{corollary}

\begin{corollary}\label{l120}
We have $H(\Lambda_1)\subset H(\Lambda_2)$ if and only if there is a change of coordinates in $\mathbb R^{r_1}$
such that $\mathbb R^{r_2}$ corresponds to the first $r_2$ coordinates, and $\Lambda_{1,i}=(\Lambda_{2,i},\beta_i)$
for each $i=1,\dots, m$ and  some $\beta_i\in\mathbb R^{r_1-r_2}$.
\end{corollary}
\begin{proof}
Indeed, we can choose a basis $e_1$, $\dots$, $e_{r_1}$ in $\mathbb Z_2^{r_1}$ such that
$\Pi(e_1)$, $\dots$, $\Pi(e_{r_2})$ is the standard basis in $\mathbb R^{r_2}$, and $e_{r_2+1}$, $\dots$, $e_{r_1}$ 
is a basis in ${\rm Ker}\, \Pi$. We have $\mathbb Z_2^{r_1}=\langle e_1,\dots,e_{r_2}\rangle\oplus 
\langle e_{r_2+1},\dots,e_{r_1}\rangle$, and in this basis $\Pi(a,b)=a$.
\end{proof}

The space $N(P,\Lambda)$ is a {\it pseudomanifold}, perhaps with a boundary. 
It is glued from $2^r$ copies of $P$, any facet of each copy belongs to at 
most two copies of $P$, and for any two copies $P\times a$ and $P\times b$ there is a sequence of polytopes $P\times a_i$, 
$i=0$, $\dots$, $l$, such that $a_{i_0}=a$, $a_l=b$, and $P \times a_i\cap P\times a_{i+1}$ 
contains a facet of both polytopes. 
After several barycentric subdivisions this condition translates to a standard definition of a~pseudomanifold
as a~simplicial complex. 
In~particular, the~notion of an~orientation of the space $N(P,\Lambda)$ is well-defined. 
The boundary of~$N(P,\Lambda)$ is glued of copies of facets $F_i$
of $P$ with $\Lambda_i=0$. 
The following result is a generalization of \cite[Lemma 2]{V87}, which gives the sufficient condition for 
orientability of $3$-dimensional small covers,  \cite[Theorem 1.7]{NN05}, which gives the criterion of 
orientability of small covers in any dimension,  and \cite[Lemma 2.4]{KMT15}, which gives the criterion of 
orientability of manifolds defined by linearly independent colorings of right-angled polytopes in any dimension 
(see also \cite[Proposition 1.12]{E22M}).

\begin{proposition}\label{orprop}
Let the vectors $\Lambda_{j_1}, \dots, \Lambda_{j_r}$ form a basis in $\mathbb Z_2^r$. 
Then the  pseudomanifold $N(P,\Lambda)$
is orientable if and only if any nonzero $\Lambda_i$ is a sum of an odd number of these vectors.
Moreover, if $N(P,\Lambda)$ is orientable, then the action of 
an~element $\boldsymbol{x}\in\mathbb Z_2^r$ preserves its orientation 
if and only if $\boldsymbol{x}$ is a sum of an even number of the vectors $\Lambda_{j_1}, \dots, \Lambda_{j_r}$.
\end{proposition}
\begin{proof}
For $N(P,\Lambda)=P\times \mathbb Z_2^r/\sim$ to be orientable it is necessary and sufficient that for any facet $F_i$ of an oriented polytope $P$ such that $\Lambda_i\ne 0$
the polytope $P\times (\boldsymbol{a}+\Lambda_i)$, which is glued to $P\times\boldsymbol{a}$ along this facet, 
has an opposite orientation. Starting from $P\times a$ and using only facets $F_{j_1}$, $\dots$, $F_{j_r}$ 
we can come from $P\times \boldsymbol{a}$ to any 
$P\times \boldsymbol{b}$, $\boldsymbol{b}\in \mathbb Z_2^r$, 
which defines uniquely the orientation of any polytope $P\times \boldsymbol{b}$. 
For these orientations to be consistent it is necessary and sufficient that for any facet $F_i$ with $\Lambda_i\ne 0$ the polytope
$P\times (\boldsymbol{a}+\Lambda_i)$ is achieved in an odd number of steps, which is equivalent to the fact that 
$\Lambda_i$ is a sum of an odd number of vectors $\Lambda_{j_l}$. 
The element  $\boldsymbol{x}\in\mathbb Z_2^r$  moves the polytope $P\times \boldsymbol{a}$ to 
$P\times (\boldsymbol{a}+\boldsymbol{x})$, so 
it preserves the orientation if and only if $\boldsymbol{x}$ is a sum of an even number of the vectors 
$\Lambda_{j_1}, \dots, \Lambda_{j_r}$.
\end{proof}
This condition can be reformulated in a more invariant form.
\begin{corollary}\label{cororc}
The  pseudomanifold $N(P,\Lambda)$ is orientable if and only if there is a linear function  
$\boldsymbol{c}\in (\mathbb Z_2^r)^*$ such that $\boldsymbol{c}\Lambda_i=1$ for all $i$ with $\Lambda_i\ne 0$. Moreover, if 
$N(P,\Lambda)$ is orientable, then the~action of an element $\boldsymbol{x}\in\mathbb Z_2^r$ preserves its orientation 
if and only if $\boldsymbol{c}\boldsymbol{x}=0$. 
\end{corollary}
\begin{proof}
Indeed, if there is such a function $\boldsymbol{c}\in (\mathbb Z_2^r)^*$, then for a basis $\Lambda_{j_1}, \dots, \Lambda_{j_r}$
$\boldsymbol{c}\Lambda_{j_s}=1$ for all $s$, hence if $\Lambda_j=u_1\Lambda_{j_1}+\dots+u_r\Lambda_{j_r}$,
then $\boldsymbol{c}\Lambda_{j}=u_1+\dots+u_r=1$, and the number of nonzero elements $u_s$ is odd. On the other hand,
if any vector  $\Lambda_j$ is a sum of an odd number of basis vectors, then
the sum of all the coordinates is the desired linear function.
\end{proof}
\begin{remark}\label{remor}
We can consider the function $\boldsymbol{c}\in (\mathbb Z_2^r)^*$ from Corollary \ref{cororc} 
as the~first coordinate in $\mathbb Z_2^r$. Then $\Lambda_i=(1,\lambda_i)$ if $\Lambda_i\ne 0$. 
More on this correspondence see in~Section \ref{sec:acol}.
\end{remark}

\begin{corollary}\label{Hor}
The  pseudomanifold $N(P,\Lambda)$ is closed and orientable if and only $H(\Lambda)\subset H_0$, that is 
$H(\Lambda)$ consists of orientation preserving involutions.  Moreover, if $N(P,\Lambda)$ is closed and orientable, 
then the subgroup of the orientation-preserving involutions $H_0'\subset \mathbb Z_2^r$ corresponds to the subgroup 
$H_0/{\rm Ker}\,\Lambda$ under the isomorphism $\mathbb Z_2^r\simeq \mathbb Z_2^m/{\rm Ker}\,\Lambda$. 
\end{corollary}
\begin{proof}
The subgroup $H_0$ corresponds to the mapping $\Lambda_0(F_i)=1$ for all $i$. Thus, this is the direct corollary
of Proposition \ref{propH12} and Corollary \ref{corh'}.
\end{proof}
\begin{corollary}\label{H'or}
The  pseudomanifold $N(P,\Lambda)/H'$,  where $H'\subset\mathbb Z_2^r$, is closed and orientable if and only 
$N(P,\Lambda)$ is closed and orientable and $H'\subset H_0'$, that is $H'$ consists of orientation-preserving involutions.
\end{corollary}
\begin{proof}
Let $H'={\rm Ker}\,\Pi$ for a surjection $\pi\colon \mathbb Z_2^r\to\mathbb Z_2^k$. Then 
$N(P,\Lambda)/H'=N(P,\Pi\circ\Lambda)$ is closed and orientable if and only if ${\rm Ker}\,\Pi\circ\Lambda\subset H_0$.
This holds if and only if ${\rm Ker}\,\Lambda\subset H_0$ and $H'\subset H_0'$.
\end{proof}

\begin{remark}
Corollaries \ref{Hor} and \ref{H'or} can be explained in another way. The pseudomanifold 
$N(P,\Lambda)=\mathbb{R}\mathcal{Z}_P/H(\Lambda)$ of dimension $n$ 
is closed and orientable if and only if $H_n(N(P,\Lambda),\mathbb Q)=\mathbb Q$. There is the following result  connected with 
the~notion of a {\it transfer}.
\begin{theorem}\label{th:transfer}(See \cite[Theorem 2.4]{B72})
Let $G$ be a finite group acting on a simplicial complex $K$ by simplicial homeomorphisms. Then 
for any field $\mathbb F$ of characteristic $0$ or prime to $|G|$ the~mapping 
$\pi_*\colon H_*(|K|,\mathbb F)\to H_*(|K|/G,\mathbb F)$ induces the isomorphism
$$
H_*(|K|,\mathbb F)^G\simeq H_*(|K|/G,\mathbb F),
$$ 
where the subgroup $H_*(|K|,\mathbb F)^G\subset H_*(|K|,\mathbb F)$ 
consists of homology classes invariant under the~action of any  $g_*$, $g\in G$.
\end{theorem}  
The action of $\mathbb Z_2^m$ on $\mathbb R\mathcal{Z}_P$ as well as $\mathbb Z_2^r$ on $N(P,\Lambda)$
is simplicial with respect to the structure of a simplicial complex arising from the barycentric subdivision of $P$,
hence for $H_n(N(P,\Lambda)/H',\mathbb Q)$ to be isomorphic to $\mathbb Q$ it is necessary and sufficient that 
$H_n(N(P,\Lambda),\mathbb Q)\simeq \mathbb Q$ (that is, $N(P,\Lambda)$ is closed and orientable)
and $H_n(N(P,\Lambda),\mathbb Q)^G=H_n(N(P,\Lambda),\mathbb Q)$ (that is, any element of $G$ preserves the orientation). 
\end{remark}

\section{A complex $\mathcal{C}(P,c)$ defined by a coloring $c$}\label{sec:CPc}
\begin{construction} 
Let us call a surjective mapping $c$ of the set of facets $\{F_1,\dots, F_m\}$ of a~polytope $P$
to a finite set consisting of $l$ elements a {\it coloring} of the polytope $P$ in $l$ colors. For 
convenience we identify the set with $[l]=\{1,\dots,l\}$, but 
in what follows it will be often a subset of $\mathbb Z_2^r$. 
For any coloring $c$ define a complex $\mathcal{C}(P,c)\subset \partial P$ as follows. Its ``facets'' are connected 
components of unions of all the facets of $P$ of the same color, 
``$k$-faces'' are connected components of intersections of $(n-k)$ different facets. By definition each $k$-face
is a union of $k$-faces of $P$. Choose a linear order of all the facets $G_1$, $\dots$, $G_M$.

By an {\it equivalence} of two complexes $\mathcal{C}(P,c)$ and $\mathcal{C}(Q,c')$ we mean 
a homeomorphism $P\to Q$ sending facets of $\mathcal{C}(P,c)$ to facets of $\mathcal{C}(Q,c')$. 
If there is such an equivalence, we call $\mathcal{C}(P,c)$ and $\mathcal{C}(Q,c')$ {\it equivalent}.
\end{construction}

Denote $\mathbb R^k_{\geqslant}=\{(y_1,\dots,y_k)\in \mathbb R^k\colon y_i\geqslant 0\text{ for all } i\}$.
For a subset $\omega\subset [m]$ denote $P_{\omega}=\bigcup_{i\in\omega}F_i$.
\begin{lemma}\label{lem:Mc}
Let a point $p\in\partial P$ belong to exactly $l\geqslant 0$ facets $G_{i_1}$, $\dots$, $G_{i_l}$ of $\mathcal{C}(P,c)$. 
Then there is a piecewise linear homeomorphism  $\varphi$ of a neighbourhood $U\subset P$ of $p$ to a neighbourhood 
$V\subset \mathbb R^l_{\geqslant}\times\mathbb R^{n-l}$ such that $\varphi(G_{j_s}\cap U)=V\cap\{y_s=0\}$, $s=1,\dots,l$.  
\end{lemma}
\begin{proof}

Take the face $G(p)=\bigcap_{F_i\ni p}F_i=F_{j_1}\cap\dots\cap F_{j_k}$. 
Since the distance from $p$  to any facet $F_j$, $p\notin F_j$, is positive, 
there is a neighbourhood $U(p)\subset \mathbb R^n$ such that $U(p)\cap P=U(p)\cap S(p)$, where 
$$
S(p)=\{x\in \mathbb R^n\colon \boldsymbol{a}_{j_1}\boldsymbol{x}+b_{j_1}\geqslant 0,\dots, 
\boldsymbol{a}_{j_k}\boldsymbol{x}+b_{j_k}\geqslant 0\},
$$ 
and $\boldsymbol{a}_j\boldsymbol{x}+b_j\geqslant 0$ is the halfspace defined by a facet $F_j$. 

For any vertex $v\in G(p)$ there is an affine change of coordinates $y_j= \boldsymbol{a}_j\boldsymbol{x}+b_j$:
$F_j\ni v$.  In~the~new coordinates  
$$
S(p)=\{y_{j_1}\geqslant 0\}\times\dots\times \mathbb \{y_{j_k}\geqslant 0\}\times \mathbb R^{n-k}=\mathbb R^k_{\geqslant}
\times\mathbb R^{n-k},
$$
where for the point $p$ we have $y_{j_1}=\dots=y_{j_k}=0$ and $y_j>0$ for all the other $j$.

Let $G_{i_s}=P_{\omega_{i_s}}$. We have a decomposition 
$\{j_1,\dots,j_k\}=\omega_{i_1}(p)\sqcup\dots\sqcup\omega_{i_l}(p)$, where 
$\omega_{i_s}(p)=\omega_{i_s}\cap \{j_1,\dots,j_k\}$. Set $p_{i_s}=|\omega_{i_s}(p)|$.
Then 
$$
S(p)=\mathbb R^{p_{i_1}}_{\geqslant}\times\dots\mathbb R^{p_{i_l}}_{\geqslant}\times\mathbb R^{n-k}.
$$

Each $\mathbb R^p_{\geqslant}$ is piecewise linearly homeomorphic to $\mathbb R^{p-1}\times\mathbb R_{\geqslant}$.
Namely
$$
\mathbb R^p_{\geqslant}={\rm cone}\,(\boldsymbol{e}_1,\dots,\boldsymbol{e}_p)=
\bigcup\limits_{j=1}^p{\rm cone}\, (\boldsymbol{e}_1,\dots,\boldsymbol{e}_ {j-1},\boldsymbol{e}_1+\dots+\boldsymbol{e}_p,\boldsymbol{e}_ {j+1}, \dots\boldsymbol{e}_p).
$$
Then the mapping 
$$
\boldsymbol{e}_1\to\boldsymbol{e}_1,\dots,\boldsymbol{e}_{p-1}\to\boldsymbol{e}_{p-1}, 
\boldsymbol{e}_p\to -\boldsymbol{e}_1-\dots-\boldsymbol{e}_{p-1}, \boldsymbol{e}_1+\dots+\boldsymbol{e}_p\to\boldsymbol{e}_p
$$ 
defines a linear homeomorphism of each cone to its image and a piecewise linear homeomorphism
$\mathbb R^p_{\geqslant}\simeq \mathbb R^{p-1}\times\mathbb R_{\geqslant}$. It maps  
$\partial \mathbb R^p_{\geqslant}=\mathbb R^p_{\geqslant}\cap\bigcup_{i=1}^p\{y_i=0\}$ to $\mathbb R^{p-1}$.
Then we have a homeomorphism 
$$
S(p)=\mathbb R^{p_{i_1}}_{\geqslant}\times\dots\mathbb R^{p_{i_l}}_{\geqslant}\times\mathbb R^{n-k}\simeq
(\mathbb R^{p_1-1}\times\mathbb R_{\geqslant})\times\dots\times (\mathbb R^{p_l-1}\times\mathbb R_{\geqslant})\times\mathbb R^{n-k}\simeq \mathbb R^l_\geqslant\times \mathbb R^{n-l}, 
$$
which sends each set $G_{i_s}\cap S(p)$ to the corresponding hyperplane $\{y_s=0\}$.
\end{proof}
\begin{corollary}
Any set $P_{\omega}$, $\omega\ne\varnothing, [m]$, is a topological $n$-manifold with a boundary. 
\end{corollary}
\begin{proof}
To prove this it is sufficient to consider a coloring $c(F_i)=\begin{cases} 1,&i\in \omega\\
2,&i\notin \omega\end{cases}$.
\end{proof} 
\begin{corollary}
Each $k$-face of  $\mathcal{C}(P,c)$ is a topological $k$-manifold, perhaps with a boundary.
\end{corollary}
The proof is similar.
\begin{remark}
It follows from Lemma \ref{lem:Mc} that the polytope $P$ with the complex $\mathcal{C}(P,c)$ on its boundary 
has the structure of a manifold with facets in the sense of \cite[Definition 7.1.2]{BP15}.  
\end{remark}

\begin{proposition}\label{prop:dnr}
Let $c$ be a coloring of a simplex $\Delta^n$ in $r$ colors. Then 
there is a homeomorphism of $\Delta^n$ to the set 
$$
S^n_{r,\geqslant}=\{(x_1,\dots,x_{n+1})\in\mathbb R^{n+1}\colon x_1\geqslant0,\dots, x_r\geqslant 0,x_1^2+\dots+x_{n+1}^2=1\}\subset S^n
$$ 
such that each facet $G_i$ of $\mathcal{C}(\Delta^n,c)$ is mapped to $S^n_{r,\geqslant}\cap \{x_i=0\}$, $i=1,\dots,r$.
\end{proposition}
\begin{proof}
We can use the same argument as in the proof of Lemma \ref{lem:Mc}.
First let us realize $\Delta^n$ as a regular simplex in $\mathbb R^{n+1}$:
\begin{gather*}
\Delta^n\simeq \{(x_1,\dots,x_{n+1})\in\mathbb R^{n+1}\colon x_1\geqslant 0,\dots, x_{n+1}\geqslant 0,x_1+\dots+x_{n+1}=1\}\\
\simeq(\mathbb R^{n+1}_\geqslant\setminus\{0\})/(\boldsymbol{x}\sim t\boldsymbol{x}, t>0)\simeq S^{n}_{n+1,\geqslant}.
\end{gather*} 
Without loss of generality we can assume that 
$$
c(F_i)=\begin{cases}
1,&1\leqslant i\leqslant p_1,\\
2,&p_1+1\leqslant i\leqslant p_1+p_2,\\
\dots\\
r,&p_1+\dots+p_{r-1}+1\leqslant i\leqslant n+1
\end{cases}
$$
As in the proof of Lemma \ref{lem:Mc} we have a piecewise linear homeomorphism 
$$
\mathbb R^{n+1}_{\geqslant}\simeq \mathbb R^{p_1}_{\geqslant}\times\dots\times\mathbb R^{p_r}_{\geqslant}\to
(\mathbb R_{\geqslant}\times\mathbb R^{p_1-1})\times \dots\times (\mathbb R_{\geqslant}\times\mathbb R^{p_r-1})\simeq
\mathbb R^r_{\geqslant}\times\mathbb R^{n+1-r}
$$
which sends rays $t\boldsymbol{x}$, $t>0$, to rays $t\boldsymbol{y}$, and each set 
$\mathbb R^{n+1}_{\geqslant}\cap \{x_i=0\}$ to $(\mathbb R^r_{\geqslant}\times\mathbb R^{n+1-r})\cap \{x_{c(i)}=0\}$.
Then 
$$
\Delta^n\simeq (\mathbb R^{n+1}_\geqslant\setminus\{0\})/(\boldsymbol{x}\sim t\boldsymbol{x}, t>0)\simeq
(\mathbb R^r_{\geqslant}\times\mathbb R^{n+1-r})/(\boldsymbol{x}\sim t\boldsymbol{x}, t>0)\simeq S^n_{r,\geqslant},
$$
and each facet $F_i$ of $\Delta^n$ is mapped to $S^n_{r,\geqslant}\cap\{x_{c(i)}=0\}$.
\end{proof}
\begin{corollary}\label{cor:dnc}
The complexes $\mathcal{C}(\Delta^n,c)$ and $\mathcal{C}(\Delta^n,c')$ are equivalent if and only if the colorings 
$c$ and $c'$ have equal numbers of colors.
\end{corollary}
\begin{definition}
We will denote $\mathcal{C}(n,r)$ the equivalence class of complexes $\mathcal{C}(\Delta^n,c)$ corresponding to $r$ colors.
\end{definition}
\begin{example}\label{Ex:CG}
For any face $G=F_{i_1}\cap\dots\cap F_{i_k}$  of $P$  of codimension $k\geqslant 1$ consider the coloring 
$$
c_G(F_j)=
\begin{cases}
s,&\text{ if }j=i_s, \\
k+1,&\text{ otherwise}.
\end{cases}
$$
\end{example}
\begin{proposition}
The complex $\mathcal{C}(P,c_G)$ is equivalent to $\mathcal{C}(n,k+1)$.
\end{proposition}
\begin{proof}
A central projection from a point $p\in{\rm relint}\,G$ induces a homeomorphism 
between $P$ and the set 
$$
B^n_{k,\geqslant}=\{(x_1,\dots,x_n)\in\mathbb R^n\colon x_1\geqslant 0,\dots, x_k\geqslant 0,x_1^2+\dots+x_n^2\leqslant 1\}
$$
such that each facet $F_{i_s}$ is mapped to the set $B^n_{k,\geqslant}\cap\{x_s=0\}$, $s=1,\dots, k$, and all the other facets are mapped 
to $B^n_{k,\geqslant}\cap\{x_1^2+\dots+x_n^2=1\}$. Hence, the complexes $\mathcal{C}(P,c_G)$ and $\mathcal{C}(Q,c_{G'})$ are equivalent, if  $P$ and $Q$ are simple $n$-polytopes and $\dim G=\dim G'$. In particular, $\mathcal{C}(P,c_G)$ is equivalent to
$\mathcal{C}(\Delta^n,c_{\Delta^{n-k}})=\mathcal{C}(n,k+1)$. 
\end{proof}
\begin{corollary} There is a homeomorphism of complexes
\begin{equation}
S^n_{r+1,\geqslant} \simeq B^n_{r,\geqslant},
\end{equation}
where one of the facets $\{x_i=0\}$ of $S^n_{r+1,\geqslant}$ is mapped to the facet $\{x_1^2+\dots+x_n^2=1\}$ of $B^n_{r,\geqslant}$.
\end{corollary}

\section{A weakly equivariant classification of spaces $N(P,\Lambda)$}\label{sec:wecL}
\begin{definition}
Two spaces $X$ and $Y$ with actions of $\mathbb Z_2^r$ are called  {\it weakly equivariantly homeomorphic}
if there is a homeomorphism $\varphi\colon X\to Y$ and 
an automorphism $\psi\colon \mathbb Z_2^r\to \mathbb Z_2^r$ 
such that $\varphi(\boldsymbol{a}\cdot \boldsymbol{x})=\psi(\boldsymbol{a})\cdot\varphi(\boldsymbol{x})$ for any 
$\boldsymbol{x}\in X$ and $\boldsymbol{a}\in\mathbb Z_2^r$.
\end{definition}

\begin{definition} 
Let $\Lambda_P$ and $\Lambda_Q$ be vector-colorings of rank $r$ of simple $n$-polytopes $P$ and $Q$.
We call the pairs $(P,\Lambda_P)$ and $(Q,\Lambda_Q)$ {\it equivalent}, 
if there is an equivalence $\sigma$ between $\mathcal{C}(P,\Lambda_P)$ and $\mathcal{C}(Q,\Lambda_Q)$ 
and a linear isomorphism $A\colon \mathbb Z_2^r \to \mathbb Z_2^r$ 
such that  $\Lambda_Q(\sigma(G_i))=A\Lambda_P(G_i)$ for all $i=1, \dots, M$.
\end{definition}
The following result generalizes the corresponding fact for linearly independent vector-colorings 
(see \cite[Proposition 1.8]{DJ91} and \cite[Proposition 7.3.8]{BP15}).
\begin{proposition}
The spaces $N(P,\Lambda_P)$ and $N(Q,\Lambda_Q)$ are weakly equivariantly homeomorphic if and only if the pairs
$(P,\Lambda_P)$ and $(Q,\Lambda_Q)$ are equivalent.
\end{proposition}
\begin{proof}
Let the pairs $(P,\Lambda_P)$ and $(Q,\Lambda_Q)$ be equivalent. We will denote by
$G_i$ the facets of $\mathcal{C}(P,\Lambda_P)$, by $G_j'$ the facets of $\mathcal{C}(Q,\Lambda_Q)$, by 
$j=\sigma(i)$ the index such that $\sigma(G_i)=G_j'$. 
Also denote $\Lambda_i=\Lambda_P(G_i)$ and  $\Lambda_j'=\Lambda_Q(G_j')$.

Define a homeomorphism $P\times \mathbb Z_2^r\to Q\times \mathbb Z_2^r$ as 
$(\boldsymbol{p},\boldsymbol{a})\to (\sigma(\boldsymbol{p}),A\boldsymbol{a})$.

If $\boldsymbol{a}_1-\boldsymbol{a}_2=\sum\limits_{i\colon\boldsymbol{p}\in F_i} \Lambda_ix_i$,
then 
\begin{multline*}
A\boldsymbol{a}_1-A\boldsymbol{a}_2=\sum\limits_{i\colon\boldsymbol{p}\in F_i} (A\Lambda_i)x_i=
\sum\limits_{i\colon\boldsymbol{p}\in G_i} (A\Lambda_i)\sum\limits_{k\colon\boldsymbol{p}\in F_k\subset G_i}x_k=
\sum\limits_{i\colon\boldsymbol{p}\in G_i} (A\Lambda_i)\widetilde{x}_i=\\
\sum\limits_{i\colon\boldsymbol{p}\in G_i} (\Lambda_{\sigma(i)}')\widetilde{x}_i=
\sum\limits_{j\colon \sigma(\boldsymbol{p})\in G_j'} \Lambda_j'\widetilde{x}_{\sigma^{-1}(j)}
=\sum\limits_{j\colon\sigma(\boldsymbol{p})\in F_k'} \Lambda_k' x_k'\text{ for some }x_k'\in \mathbb Z_2.
\end{multline*}
Thus, the mapping preserves the equivalence classes, and we obtain the homeomorphism 
$\varphi\colon N(P,\Lambda_P)\to N(Q,\Lambda_Q)$. Moreover, 
\begin{gather*}
\varphi\left(\boldsymbol{a}\cdot \left[\boldsymbol{p},\boldsymbol{b}\right]\right)= 
\varphi\left[\boldsymbol{p},\boldsymbol{a}+\boldsymbol{b}\right]=
\left[\sigma(\boldsymbol{p}),A\left(\boldsymbol{a}+\boldsymbol{b}\right)\right]=\\
\left[\sigma(\boldsymbol{p}),A\boldsymbol{a}+A\boldsymbol{b}\right]=
(A\boldsymbol{a})\cdot\left[\sigma(\boldsymbol{p}),A\boldsymbol{b}\right]=
(A\boldsymbol{a})\cdot\varphi\left[\boldsymbol{p},\boldsymbol{b}\right]
\end{gather*}
Thus, $\varphi$ is a weakly equivariant homeomorphism.

Now assume that there is a weakly equivariant homeomorphism 
$\varphi\colon N(P,\Lambda_P)\to N(Q,\Lambda_Q)$. Then there is $A\in Gl_r(\mathbb Z_2)$ such that 
$\varphi(\boldsymbol{a}\cdot [\boldsymbol{p},\boldsymbol{b}])=(A\boldsymbol{a})\cdot \varphi[\boldsymbol{p},\boldsymbol{b}]$
for all $\boldsymbol{p}\in P$ and $\boldsymbol{a},\boldsymbol{b}\in\mathbb Z_2^r$.
Since $\varphi$ is weakly equivariant, it induces a homeomorphism of orbit spaces
$\widehat{\varphi}\colon P\to Q$, where $\widehat{\varphi}(\partial P)=\partial Q$. 
Moreover, the points in $N(P,\Lambda_P)$ with a stabilizer $H\subset \mathbb Z_2^r$
are mapped by $\varphi$ to the points in $N(Q,\Lambda_Q)$ with the stabilizer $A(H)$. 
For a facet $G_i$ of $\mathcal{C}(P,\Lambda_P)$ define its {\it relative interior} ${\rm relint}\,G_i$ 
to be the~interior of $G_i$ as a~subset of~$\partial P$. 
Then the points over ${\rm relint}\,G_i$ have the stabilizer $\langle \Lambda_i\rangle$
and are mapped to the points over relative interiors of the~facets $G_{j_1}'$, $\dots$, $G_{j_l}'$ 
of $\mathcal{C}(Q,\Lambda_Q)$ with the stabilizer $\langle A\Lambda_i\rangle$. 
Since ${\rm relint}\,G_i$ is path-connected and each ${\rm relint}\,G_{j_s}'$  is a connected component of 
$\bigcup_s{\rm relint}\,G_{j_s}'$ 
(because $G_{j_s}'\cap G_{j_t}'=\varnothing$ for $s\ne t$), 
we have $\widehat{\varphi}({\rm relint}\,G_i)={\rm relint}\,G_{j_s}'$ for a single facet $G_{j_s}'$. 
Also $\widehat{\varphi}(\partial G_i)=\partial G_{j_s}'$, since $\widehat{\varphi}$ is continuous.
Thus, $\widehat{\varphi}$ is an equivalence between $\mathcal{C}(P,\Lambda_P)$ and $\mathcal{C}(Q,\Lambda_Q)$
such that $A\Lambda_P(G_i)=\Lambda_Q(\sigma(G_i))$. The proof is finished.      
\end{proof}

\section{A weakly equivariant classification of spaces defined by affine colorings}\label{sec:acol}
Remark \ref{remor} leads to the following definition.
\begin{definition}\label{def:acol}
We call a mapping $\lambda\colon \{F_1,\dots,F_m\}\to \mathbb Z_2^r$  such that the images $\lambda_j$
of the facets $F_j$ affinely span $\mathbb Z_2^r$  {\it an affine coloring} of rank $r$. If, additionally,
$$
(**)\quad\text{for any face $F_{i_1}\cap\dots\cap F_{i_k}$ of $P$ the points $\lambda_{i_1}$, $\dots$,
$\lambda_{i_k}$ are affinely independent} 
$$
we call $\lambda$ {\it an affinely independent coloring}.
\end{definition}
\begin{definition}
Let $\lambda$ be an affine coloring of a simple $n$-polytope $P$. Define $\Lambda_i=(1,\lambda_i)\in\mathbb Z_2^{r+1}$.
We call the space $N(P,\lambda)=N(P,\Lambda)$ a {\it space defined by an affine coloring $\lambda$}.  Set $H(\lambda)=H(\Lambda)$.
\end{definition}
By definition $N(P,\lambda)$ is a closed orientable pseudomanifold and any closed orientable pseudomanifold $N(P,\Lambda)$
has this form. There is a canonical action of $\mathbb Z_2^{r+1}$ on $N(P,\lambda)$, and the~subgroup
of orientation-preserving involutions is 
$$
H_0'=\mathbb Z_2^r=\{(x_0,\dots,x_r)\in\mathbb Z_2^{r+1}\colon x_0=0\}.
$$
This subgroup can be considered as a vector space associated to the affine space $\mathbb Z_2^r$
generated by the points $\lambda_1$, $\dots$, $\lambda_m$.

The following results follow from Proposition \ref{propH12} and Corollary \ref{H'or}.
\begin{corollary}
We have $H(\lambda_1)\subset H(\lambda_2)$ if and only if there is 
an~affine surjection $\widehat{\Pi}\colon \mathbb Z^{r_1}\to\mathbb Z^{r_2}$ such that $\lambda_2=\widehat{\Pi}\circ\lambda_1$.
In this case $N(P,\lambda_2)=N(P,\lambda_1)/H'$, where $H'\simeq H(\lambda_1)/ H(\lambda_2)$.
\end{corollary}
\begin{corollary}\label{cor:affH'}
For a~subgroup $H'\subset \mathbb Z_2^{r+1}$ the~space $N(P,\lambda)/H'$ is a~closed orientable
pseudomanifold if and only if $H'\subset\mathbb Z_2^r=H_0'$. In this case $N(P,\lambda)/H'=N(P,\widehat{\Pi}\circ\lambda)$, 
where $\widehat{\Pi}\colon\mathbb Z_2^r\to\mathbb Z_2^r/H'$ is an~affine surjection.
\end{corollary}
\begin{corollary}\label{cor:affH'P}
For an~affine coloring $\lambda$ of rank $r$ of a~simple $n$-polytope $P$ 
the~subgroups $H\colon H(\lambda)\subset H\subset H_0\subset\mathbb Z_2^m$
are in bijection with 
\begin{itemize}
\item affine surjections $\widehat{\Pi}\colon\mathbb Z_2^r\to\mathbb Z_2^l$ defined up to affine changes of coordinates in 
$\mathbb Z_2^l$; 
\item affine colorings $\lambda'$ of rank $l$ of the form $\lambda'=\widehat{\Pi}\circ\lambda$ 
defined up to affine changes of coordinates in 
$\mathbb Z_2^l$; 
\item subgroups 
$H'\subset\mathbb Z_2^r=H_0'\subset\mathbb Z_2^{r+1}$ of involutions preserving the~orientation of $N(P,\lambda)$. 
\end{itemize}
The correspondence between the projections and the subgroups is given as 
$$
H'\to \left[\mathbb Z_2^r\to \mathbb Z_2^r/H'\simeq\mathbb Z_2^l\right],\quad
\left[A\boldsymbol{x}+\boldsymbol{b}\colon \mathbb Z_2^r\to \mathbb Z_2^l\right]\to {\rm \Ker A}.
$$  
\end{corollary}
\begin{definition} 
Let $\lambda_P$ and $\lambda_Q$ be affine colorings of rank $r$ of simple $n$-polytopes $P$ and $Q$.
We call the pairs $(P,\lambda_P)$ and $(Q,\lambda_Q)$ {\it equivalent}, 
if there is an equivalence $\sigma$ between $\mathcal{C}(P,\lambda_P)$ and $\mathcal{C}(Q,\lambda_Q)$ 
and an affine isomorphism $\mathcal{A}\colon \mathbb Z_2^r \to \mathbb Z_2^r$ 
such that  $\lambda_Q(\sigma(G_i))=\mathcal{A}\lambda_P(G_i)$ for all $i=1, \dots, M$.
\end{definition}
\begin{corollary}
The spaces $N(P,\lambda_P)$ and $N(Q,\lambda_Q)$ are weakly equivariantly homeomorphic if and only if the pairs
$(P,\lambda_P)$ and $(Q,\lambda_Q)$ are equivalent.
\end{corollary}
\begin{proof}
Indeed, linear isomorphisms $\mathbb Z_2^{r+1}\to\mathbb Z_2^{r+1}$
such that the vectors $(1,\lambda_i)$ spanning $\mathbb Z_2^{r+1}$ are mapped to vectors $(1,\lambda_j')$
have the form $(1,\boldsymbol{x})\to (1, C\boldsymbol{x}+\boldsymbol{b})$, where $\det C=1$, that is they correspond to 
affine isomorphisms $\mathbb Z_2^r\to \mathbb Z_2^r$.
\end{proof}
\section{A criterion when $N(P,\Lambda)$ is a manifold}\label{sec:NPman}

\begin{theorem}\label{th:ZHM}
The space $N(P,\Lambda)$ defined by a vector-coloring $\Lambda$ of a rank $r$ of a simple $n$-polytope $P$
is a closed topological manifold if and only if all the vectors $\Lambda_i$ are nonzero and 
for any vertex $v=F_{i_1}\cap\dots\cap F_{i_n}$ of $P$ all the different vectors among $\{\Lambda_{i_1},\dots,\Lambda_{i_n}\}$ 
are linearly independent.  
It is a topological manifold with a boundary if and only if $\Lambda_j=0$ for some $j$, and for any vertex $v$ 
all the nonzero different vectors among $\{\Lambda_{i_1},\dots,\Lambda_{i_n}\}$ are linearly independent. In this case the
boundary is glued of copies of facets $F_j$ with $\Lambda_j=0$.
\end{theorem}
\begin{remark}
Theorem \ref{th:ZHM} can be extracted from general results by A.V.~Mikhailova \cite{M85} and C.~Lange \cite{L19}.
Nevertheless, we give a short self-sufficient proof here. For $r=m-n+1$ Theorem \ref{th:ZHM} also 
follows from results of \cite{G23}.
\end{remark}

\begin{example}
In the case of $3$-polytopes the first condition means that at each vertex $v=F_i\cap F_j\cap F_k$ either
$\Lambda_i=\Lambda_j=\Lambda_k$, or for a relabelling $\Lambda_i\ne\Lambda_j$ and $\Lambda_k\in \{\Lambda_i,\Lambda_j\}$,
or the vectors $\Lambda_i$, $\Lambda_j$, and $\Lambda_k$ are linearly independent.
\end{example}
\begin{corollary}
The space $N(P,\Lambda)$ defined by a vector-coloring $\Lambda$
is a closed topological manifold if and only if $\Lambda$ induces a~linearly independent
coloring of the complex $\mathcal{C}(P,\Lambda)$. 
\end{corollary}
\begin{corollary} The space $N(P,\lambda)$ defined by an affine coloring $\lambda$
is a closed orientable topological manifold if and only if $\lambda$ induces an~affinely independent
coloring of the complex $\mathcal{C}(P,\lambda)$.
\end{corollary}
\begin{proof}[Proof of Theorem \ref{th:ZHM}]
Consider the complex $\mathcal{C}(P,\Lambda)$. By construction the mapping 
$\Lambda$ induces the vector-coloring of its facets $G_1,\dots, G_M$.
We have 
\begin{equation}\label{NPC}
N(P,\Lambda)=P\times \mathbb Z_2^r/\sim, \text{ where }(p,a)\sim(q,b)\text{ if and only if }p=q \text{ and }a-b
\in\langle \Lambda_i\colon p\in G_i\rangle.
\end{equation}
If at each vertex $v=F_{i_1}\cap\dots\cap F_{i_n}$ all the different vectors among $\{\Lambda_{i_1},\dots,\Lambda_{i_n}\}$ 
are linearly independent, then for each point $p\in \partial P$, which 
belongs to exactly $l$ facets $G_{i_1}$, $\dots$, $G_{i_l}$, the vectors  $\Lambda_{i_1},\dots, \Lambda_{i_l}$ are linearly
independent. 
By Lemma \ref{lem:Mc} $p$ has a neighbourhood in $P$ homeomorphic to $\mathbb R^l_{\geqslant}\times\mathbb R^{n-l}$.
Then in $N(P,\Lambda)$ for the point $p\times a$ these neighbourhoods are glued to~the~neighbourhood homeomorphic to 
$\mathbb R^l\times \mathbb R^{n-l}$. Indeed, in $p\times a$ the copies $P\times (a+\varepsilon_1\Lambda_{i_1}+
\dots+\varepsilon_l\Lambda_{i_l})$, $\varepsilon_s=\pm1$, are glued locally as the sets 
$\{\varepsilon_1y_1\geqslant 0,\dots, \varepsilon_ly_l\geqslant 0\}$,  where the addition of the vector $\Lambda_{i_s}$ corresponds to the operation $y_s\to -y_s$. Hence, $N(P,\Lambda)$ is a closed topological manifold.

On the other hand, if $\Lambda_j=0$ for some $j$ but  at each vertex $v=F_{i_1}\cap\dots\cap F_{i_n}$ 
all the nonzero different vectors among $\{\Lambda_{i_1},\dots,\Lambda_{i_n}\}$ 
are linearly independent, then for the the points $p$ lying in the facets $G_j$ with $\Lambda_j=0$ the neighbourhoods
of the form $\mathbb R^l_{\geqslant}\times\mathbb R^{n-l}$ are glued to $\mathbb R_{\geqslant}\times \mathbb R^{n-1}$,
where the coordinate $y_s\geqslant 0$ corresponds to the facet $G_j$. Thus, $N(P,\Lambda)$ is topological manifold with 
a boundary glued from copies of the facets $G_j$ with $\Lambda_j=0$.

Now assume that at some vertex $v=F_{i_1}\cap\dots\cap F_{i_n}$ we have 
$\Lambda_{j_k}=\Lambda_{j_1}+\dots+\Lambda_{j_{k-1}}$ for $\{j_1,\dots,j_k\}\subset\{i_1,\dots,i_n\}$ 
and all the vectors $\Lambda_{j_1}$, $\dots$, $\Lambda_{j_k}$ are nonzero and different (in particular, $k\geqslant 3$). 
Moreover, assume that $k$ is minimal. In particular, the vectors $\Lambda_{j_1}, \dots, \Lambda_{j_{k-1}}$ are linearly independent.
Consider a point $p$ such that $G_{j_1}$, $\dots$, $G_{j_k}$ are exactly the facets containing this point. Such a point
exists by Lemma \ref{lem:Mc} applied to the point $v$. Also by this lemma some neighbourhood of $p$ in $P$ is homeomorphic to 
$\mathbb R^k_{\geqslant}\times\mathbb R^{n-k}$, and the facets $G_{j_s}$ are mapped to the hyperplanes 
$y_s=0$. Then for the space $N(P,\Lambda)$ in the point $p\times a$ the copies 
$(P\setminus G_{j_k})\times (a+\varepsilon_1\Lambda_{j_1}+\dots+\varepsilon_{k-1}\Lambda_{j_{k-1}})$, 
$\varepsilon_s=\pm1$, are glued locally as the sets 
$\{\varepsilon_1y_1\geqslant 0,\dots, \varepsilon_{k-1}y_{k-1}\geqslant 0,y_k>0\}$ and form 
$\mathbb R^{k-1}\times \mathbb R_>\times\mathbb R^{n-k}$, where the addition of the vector $\Lambda_{j_s}$ 
corresponds to the operation $y_s\to -y_s$. 
The points in $G_{j_k}\subset P$ correspond to the points in 
$\mathbb R^{k-1}\times\{0\}\times\mathbb R^{n-k}$. In $N(P,\Lambda)$ for these points
we have the additional identification $(x,a)\sim (x,a+\Lambda_{j_k})=(x,a+\Lambda_{j_1}+\dots+\Lambda_{j_{k-1}})$.
This means that the point $(y_1,\dots, y_{k-1},0,y_{k+1},\dots, y_n)$ is identified with 
$(-y_1,\dots,-y_{k-1},0,y_{k+1},\dots,y_n)$. Equivalently, the copies of $\mathbb R^k_{\geqslant}\times\mathbb R^{n-k}$ are
glued to the space $\mathbb R^n/\sim$, where $(y_1,\dots,y_k,y_{k+1},\dots,y_n)\sim (-y_1,\dots,-y_k,y_{k+1},\dots,y_n)$,
and the point $p\times a$ corresponds to the equivalence class $[\boldsymbol{y}_0]$ of some point 
$\boldsymbol{y}_0=(0,\dots,0,y_{k+1}^0,\dots,y_n^0)$.
In $\mathbb R^n$ the point $\boldsymbol{y}_0$ has a~ball neighbourhood $B$ of radius $\varepsilon$  with the boundary sphere $S^{n-1}$
homeomorphic to the join 
$$
S^{k-1}*S^{n-k-1}=S^{k-1}\times S^{n-k-1}\times [0,1]/(a_1,b,0)\sim (a_2,b,0),\,(a,b_1,1)\sim (a,b_2,1)
$$
via the mapping $S^{k-1}*S^{n-k-1}\to S^{n-1}\colon (a,b,t)\to (\sqrt{t}a,\sqrt{1-t}b)$. 
There is a homeomorphism $B\simeq CS^{n-1}\simeq C(S^{k-1}*S^{n-k-1})$, where $CX$ is the cone over $X$.
In  $\mathbb R^n/\sim$ this  gives a neighbourhood homeomorphic to $C(\mathbb RP^{k-1}*S^{n-k-1})=C\Sigma^{n-k}
\mathbb RP^{k-1}$, where $\Sigma X$ is a suspension over $X$.
Then  
\begin{multline*}
H_i(N(P,\Lambda),N(P,\Lambda)\setminus [p\times a])\simeq 
H_i(C\Sigma^{n-k}\mathbb RP^{k-1},C\Sigma^{n-k}\mathbb RP^{k-1}\setminus \text{apex})\simeq\\
H_i(C\Sigma^{n-k}\mathbb RP^{k-1},\Sigma^{n-k}\mathbb RP^{k-1})\simeq \widetilde{H}_{i-1}(\Sigma^{n-k}\mathbb RP^{k-1})
\simeq \widetilde{H}_{i+k-n-1}(\mathbb RP^{k-1}).
\end{multline*}
In particular, for $k\geqslant 3$ we have $H_{n+2-k}(N(P,\Lambda),N(P,\Lambda)\setminus [p\times a])=\mathbb Z_2$, and 
$N(P,\Lambda)$ is not a manifold.
\end{proof}
\begin{corollary}\label{cor:manor}
For any affine coloring of a~simple $3$-polytope $P$ the space $N(P,\lambda)$ is a closed
orientable manifold. 
\end{corollary}
\begin{proof}
This follows from the fact that any two or three different points in $\mathbb Z_2^r$ are affinely independent.
\end{proof}
\begin{corollary} Let $e_1,\dots,e_r$ be a basis in $\mathbb Z_2^r$. Then for any mapping 
$\Lambda\colon  \{F_1,\dots, F_m\}\to \{e_1,\dots, e_r,e_1+\dots+e_r\}$ the space $N(P,\Lambda)$
is a closed topological manifold. Moreover, for odd $r$ it is orientable.
\end{corollary}

\begin{construction}\label{con:spn}
Let $P$ be a simple $n$-polytope and $\lambda$ be its affine coloring of rank $r$. 
If~the~complex $\mathcal{C}(P,\lambda)$ is equivalent to $\mathcal{C}(n,r+1)$ then the~induced coloring
is affinely independent, the polytope is homeomorphic to $S^n_{r+1,\geqslant}$, 
and the manifold $N(P,\lambda)$ is homeomorphic to $S^n$ glued from $2^{r+1}$ copies of $S^n_{r+1,\geqslant}$. 
\end{construction}
\begin{example}\label{ex:spnG}
Examples for Construction \ref{con:spn} are provided by Example \ref{Ex:CG}.
Each face $G=F_{i_1}\cap\dots \cap F_{i_k}$ corresponds to an affine coloring 
$$
\lambda_i=\begin{cases}
\boldsymbol{e}_s,&\text{ if }i=i_s, s=1,\dots, k,\\
\boldsymbol{0},&\text{ otherwise},
\end{cases}
$$
where $\boldsymbol{e}_1=(1,0,\dots,0)$, $\dots$, $\boldsymbol{e}_{k}=(0,\dots, 0,1)\in\mathbb Z_2^k$.
Then the subgroup $H_G=H(\lambda)$ of rank $m-k-1$ 
is defined in $\mathbb Z_2^m$ by the equations $x_{i_1}=0$, $\dots$, $x_{i_k}=0$,
and $x_1+\dots+x_m=0$. This is the intersection of the subgroup $H_0$ consisting of all the orientation preserving involutions
with the coordinate subgroup corresponding to $G$. We have $\mathbb{R}\mathcal{Z}_P/H_G\simeq S^n$.

In particular, each vertex $v\in P$ corresponds to a subgroup $H_v$ of rank $m-n-1$ such that 
$\mathbb{R}\mathcal{Z}_P/H_v\simeq S^n$. The particular case of this construction was presented in \cite{G19}.
This corresponds to the case when $P=\Delta^{n_1}\times\dots\times\Delta^{n_k}$ and $v$ is any vertex.
We obtain an~action of~$\mathbb Z_2^{k-1}$ on~$S^{n_1}\times\dots\times S^{n_k}$ with the orbit space $S^{n_1+\dots+n_k}$.
\end{example}
\begin{conjecture}
The space $N(P,\lambda)$ corresponding to an affine coloring $\lambda$ of rank $r$ of a simple $n$-polytope $P$
is homeomorphic to $S^n$ if and only if $\mathcal{C}(P,\lambda)\simeq\mathcal{C}(n,r+1)$.
\end{conjecture}
\begin{example}
In dimension $n=1$ we have $P=I^1=\Delta^1$ and the conjecture is valid.

In dimension $n=2$ the complex $\mathcal{C}(P_m,\lambda)$ corresponding to an $m$-gon is equivalent 
ether to $\mathcal{C}(2,1)$, or to $\mathcal{C}(2,2)$, or to a complex $\mathcal{C}(P_l,\lambda')$ corresponding to
an~affinely independent coloring of an $l$-gon $P_l$, $l\geqslant 3$. In the latter case $N(P_m,\lambda)=N(P_l,\lambda')$ 
is a sphere with $g$ handles, where $\chi(N(P_l,\lambda'))=2-2g=2^{r-1}l-2^rl+2^{r+1}$. Therefore,
$g=1+2^{r-2}(l-4)$ and  $N(P_m,\lambda)\not \simeq S^2$ for $l>3$. Thus, the conjecture is valid.

As we will see in Section \ref{Sec:3sph} the conjecture is valid in dimension $n=3$. 

As it will be shown in \cite{E24b} the conjecture is also valid in dimension $n=4$.

\end{example}

Now we will prove a fact about skeletons of the complexes $\mathcal{C}(P,\Lambda)$
and $\mathcal{C}(P,\Pi\circ \Lambda)$ which we will
need below.
\begin{proposition}
Let $\Lambda$ be a vector-coloring of rank $r$ of a simple $n$-polytope $P$ 
such that $N(P,\Lambda)$ is a manifold, and $H'\subset\mathbb Z_2^r$ be a subgroup of rank $k$
corresponding to a vector-coloring $\Lambda'=\Pi\circ \Lambda$, where 
$\Pi\colon \mathbb Z_2^r\to \mathbb Z_2^r/H'\simeq\mathbb Z_2^{r-k}$ is the canonical projection. Then 
any $q$-skeleton $\mathcal{C}^q(P,\Lambda)$ belongs to the $(q+k)$-skeleton $C^{q+k}(P,\Lambda')$.
\end{proposition}
\begin{proof}
Consider a point $\boldsymbol{x}\in \mathcal{C}^q(P,\Lambda)$. It lies in the intersection 
of $(n-q)$ facets $G_{i_1}$, $\dots$, $G_{i_{n-q}}$. Let $F_{j_1}\cap\dots\cap F_{j_l}$ be the minimal face of $P$
containing $\boldsymbol{x}$. Then  
$\{\Lambda(F_{j_1}), \dots, \Lambda(F_{j_l})\}=\{\Lambda(G_{i_1}),\dots, \Lambda(G_{i_{n-q}})\}$
and the latter set of vectors in linearly independent. If the set $\{\Lambda'(G_{i_1}),\dots, \Lambda'(G_{i_{n-q}})\}$
consists of $n-s$ different vectors, then $\boldsymbol{x}\in \mathcal{C}^s(P,\Lambda')$. We have
\begin{multline*}
n-s\geqslant \dim \langle\Lambda'(G_{i_1}),\dots, \Lambda'(G_{i_{n-q}})\rangle=\\
=\dim \langle\Lambda(G_{i_1}),\dots, \Lambda(G_{i_{n-q}})\rangle-
\dim {\rm Ker}\,\Pi\left.\right|_{\langle\Lambda(G_{i_1}),\dots, \Lambda(G_{i_{n-q}})\rangle}\geqslant \\
\geqslant \dim \langle\Lambda(G_{i_1}),\dots, \Lambda(G_{i_{n-q}})\rangle-
\dim {\rm Ker}\,\Pi=n-q-k.
\end{multline*}
Thus, $s\leqslant q+k$ and $\boldsymbol{x}\in \mathcal{C}^{q+k}(P,\Lambda')$.
\end{proof}
\begin{corollary}\label{cor:C01inv}
Let $\Lambda$ be a vector-coloring of rank $r$ of a simple $n$-polytope $P$ 
such that $N(P,\Lambda)$ is a manifold, and $\tau\subset\mathbb Z_2^r$ be an involution.
Then any vertex of $\mathcal{C}(P,\Lambda)$  is either a~vertex of~$C(P,\Lambda_\tau)$ or belongs to its $1$-face,
where $\Lambda_\tau=\Pi\circ \Lambda$, and $\Pi\colon \mathbb Z_2^r\to \mathbb Z_2^r/\langle\tau\rangle\simeq\mathbb Z_2^{r-1}$ is~the~canonical projection.
\end{corollary}
\section{Manifolds with torus actions}\label{sec:mantorus}

Results obtained in Section \ref{sec:NPman} can be generalized to~actions of~compact
torus $\mathbb T^m=(S^1)^m$ instead of $\mathbb Z_2^m$. 
Namely, let us identify $S^1$ with $\mathbb R/\mathbb Z$ and $\mathbb T^r$ with $\mathbb R^r/\mathbb Z^r$. Then for 
a~mapping $\Lambda\colon \{F_1,\dots,F_m\}\to \mathbb Z^r$ such that 
$\langle\Lambda_1,\dots\Lambda_m\rangle=\mathbb Z^r$ one can define a~space
$$
M(P,\Lambda)=P\times \mathbb T^r/\sim,
$$
where $(\boldsymbol{p}_1, \boldsymbol{t}_1) \sim (\boldsymbol{p}_2, \boldsymbol{t}_2)$
if and only if $\boldsymbol{p}_1=\boldsymbol{p}_2$ and 
$\boldsymbol{t}_1-\boldsymbol{t}_2\in \left\{\sum\limits_{i\colon\boldsymbol{p}_1\in F_i} 
\Lambda_i\varphi_i, \varphi_i\in \mathbb R/\mathbb Z\right\}$. 

We will call  the~mapping $\Lambda$ {\it an integer vector-coloring of rank $r$}.

The space $M(P,\Lambda)$ has a~canonical action of~$\mathbb T^r$ and $M(P,\Lambda)/\mathbb T^r=P$.

When $\Lambda$ has an additional property 
\begin{equation}
\{\Lambda_{i_1},\dots,\Lambda_{i_k}\}\text{ is a part of some basis in }\mathbb Z^r\text{ if }F_{i_1}\cap\dots\cap F_{i_k}\ne\varnothing,
\end{equation} 
then it is known that $M(P,\Lambda)$ is a topological (even smooth) manifold obtained as~an~orbit space
of~a~free action of~the~group 
$$
H(\Lambda)=\{(\varphi_1,\dots, \varphi_m)\in\mathbb T^m\colon \Lambda_1\varphi_1+\dots+\Lambda_m\varphi_m=\boldsymbol{0}\}\simeq \mathbb T^{m-r}
$$ 
on~the~moment-angle manifold $\mathcal{Z}_P=M(P,E)$, $E(F_i)=\boldsymbol{e}_i$, where $\boldsymbol{e}_1$, $\dots$, $\boldsymbol{e}_m$ is the~standard basis in~$\mathbb Z^m$ (see \cite{DJ91, BP15}).  
We have the~following generalization.
\begin{proposition}\label{prop:ta}
Let $P$ be a~simple $n$-polytope and  $\Lambda\colon \{F_1,\dots,F_m\}\to \mathbb Z^r\setminus\{\boldsymbol{0}\}$ 
be~an integer vector-coloring of rank $r$ such that for any
vertex  $v=F_{i_1}\cap\dots\cap F_{i_n}$ all the different vectors among $\{\Lambda_{i_1},\dots,\Lambda_{i_n}\}$ 
form a part of a~basis in $\mathbb Z^r$. Then $M(P,\Lambda)$ is a~closed topological $(n+r)$-manifold.
\end{proposition}
\begin{proof}
Consider the complex $\mathcal{C}(P,\Lambda)$. There is an induced mapping  
$\Lambda$ for the set of its facets $G_1,\dots, G_M$. 
For each point $p\in \partial P$, which belongs to exactly $l$ facets $G_{i_1}$, $\dots$, $G_{i_l}$, 
the vectors  $\Lambda_{i_1},\dots, \Lambda_{i_l}$ form a part of a~basis in $\mathbb Z^r$. 
By Lemma \ref{lem:Mc} the point $p$ 
has a neighbourhood in $P$ homeomorphic to $\mathbb R^l_{\geqslant}\times\mathbb R^{n-l}$.
The open set in $M(P,\Lambda)$ over this neighbourhood is homeomorphic~to 
$$
\mathbb R^l_{\geqslant}\times\mathbb R^{n-l}\times \mathbb T^l\times\mathbb T^{r-l}/\sim
\simeq \mathbb C^l\times  \mathbb R^{n-l}\times \mathbb T^{r-l}.
$$
Thus, $M(P,\Lambda)$ is a~closed topological $(n+r)$-manifold. 
\end{proof}
This result can  be obtained as a~corollary of general results in \cite{S09} and also of \cite[Theorem 1.1]{AGo24}.

\begin{proposition}\label{prop:spn+r}
Let $P$ be a simple $n$-polytope and $\Lambda\colon\{F_1,\dots,F_m\}\to\{\boldsymbol{e}_1,\dots,\boldsymbol{e}_r\}$
be~an~epimorphism, where  $\{\boldsymbol{e}_1,\dots,\boldsymbol{e}_r\}$ is~a~basis in~$\mathbb Z^r$.
If~the~complex $\mathcal{C}(P,\Lambda)$ is equivalent to $\mathcal{C}(n,r)$ then the~polytope 
is~homeomorphic to $S^n_{r,\geqslant}$, and the manifold $M(P,\Lambda)$ is homeomorphic to $S^{n+r}$. 
\end{proposition}
\begin{proof}
Indeed, 
$S^n_{r,\geqslant}\times \mathbb T^r/\sim\simeq S^{n+r}$,
and the~homeomorphism is~given as 
{\small 
$$
\left[(x_1,\dots, x_{n+1}), (\varphi_1,\dots,\varphi_r)\right]\to
(x_1\cos(2\pi\varphi_1), x_1\sin(2\pi\varphi_1),
\dots,x_r\cos(2\pi\varphi_r), x_r\sin(2\pi\varphi_r),x_{r+1},\dots, x_{n+1}).
$$}
\end{proof}
\begin{example}\label{ex:spn+k+1G}
Examples for Proposition \ref{prop:spn+r} are provided by Example \ref{Ex:CG}.
Each face $G=F_{i_1}\cap\dots \cap F_{i_k}$ corresponds to~a~mapping 
$$
\Lambda_i=\begin{cases}
\boldsymbol{e}_s,&\text{ if }i=i_s, s=1,\dots, k,\\
\boldsymbol{e}_{k+1},&\text{ otherwise},
\end{cases}
$$
where $\boldsymbol{e}_1=(1,0,\dots,0)$, $\dots$, $\boldsymbol{e}_{k+1}=(0,\dots, 0,1)\in\mathbb Z^{k+1}$.
Then the subgroup $H_G=H(\Lambda)\simeq \mathbb T^{m-k-1}$ 
is defined in $\mathbb T^m$ by the equations $\varphi_{i_1}=0$, $\dots$, $\varphi_{i_k}=0$,
and $\varphi_1+\dots+\varphi_m=0$. We have $\mathcal{Z}_P/H_G\simeq S^{n+k+1}$.
\end{example}
\begin{example}\label{ex:AGU}
For each polytope $P$ the~mapping $\Lambda_i=1\in\mathbb Z$ gives the complex 
$\mathcal{C}(P,\Lambda)\simeq \mathcal{C}(n,1)$. The subgroup $H_0=H(\Lambda)$
is defined by the equation $\varphi_1+\dots+\varphi_m=0$. We have $\mathcal{Z}_P/H_0=S^{n+1}$.

For any vector-coloring $\Lambda$ such that there is~a~function $\boldsymbol{c}=(c_1,\dots,c_r)\in(\mathbb Z^r)^*$
with $\boldsymbol{c}\Lambda_i=1$ for all $i$ we have $H(\Lambda)\subset H_0$ and on
the space $M(P,\Lambda)$ there is~an~action of~$H_0'=H_0/H(\Lambda)\simeq \mathbb T^{r-1}$ 
such that $M(P,\Lambda)/H_0'=\mathcal{Z}_P/H_0\simeq S^{n+1}$. The subgroup $H_0'$
is defined in $\mathbb T^r$ by the equation $c_1\psi_1+\dots+c_r\psi_r=0$.

In particular, for the product of polytopes $P^n=P_1^{n_1}\times\dots\times P_k^{n_k}$ 
each facet has the form $P_1\times\dots F_{i,j}\times\dots\times P_k$, where $F_{i,j}$
is a facet of $P_i$. We have a mapping $\Lambda(P_1\times\dots F_{i,j}\times\dots\times P_k)=\boldsymbol{e}_i$,
where $\boldsymbol{e}_1$, $\dots$, $\boldsymbol{e}_k$ is~the~standard basis in~$\mathbb Z^k$.
For the function $\boldsymbol{c}=(1,\dots, 1)\in(\mathbb Z^k)^*$ we have $\boldsymbol{c}\boldsymbol{e}_i=1$
for all $i$. Then
$$
\mathcal{Z}_P=\mathcal{Z}_{P_1}\times\dots\times\mathcal{Z}_{P_k}\text{ and }
M(P,\Lambda)=\mathcal{Z}_P/(H_{1,0}\times\dots\times H_{k,0})=S^{n_1+1}\times\dots\times S^{n_k+1}.
$$
On this manifold there is~an~action of $H_0'=H_0/(H_{1,0}\times\dots\times H_{k,0})\simeq \mathbb T^{k-1}$.
This  subgroup is defined in  $\mathbb T^k$ by the equation $\psi_1+\dots+\psi_k=0$.
Then $S^{n_1+1}\times\dots\times S^{n_k+1}/H_0'\simeq \mathcal{Z}_P/H_0\simeq S^{n+1}$.
This torus analog of Dmitry Gugnin's construction from \cite{G19} was described in \cite{AGu23}.

The latter example can be generalized as follows. Given integer vector-colorings $\Lambda_{P_i}$
of ranks $r_i$ on polytopes $P_i$ such that $\mathcal{C}(P_i,\Lambda_{P_i})\simeq \mathcal{C}(n,r_i)$
we have the product coloring $\Lambda_P$ on $P_1\times\dots\times P_k$ such that 
$M(P,\Lambda)\simeq S^{n_1+r_1}\times\dots\times S^{n_k+r_k}$ and an action of $H_0'\simeq \mathbb T^{r_1+\dots+r_k-1}$
such that $S^{n_1+r_1}\times\dots\times S^{n_k+r_k}/\mathbb T^{r_1+\dots+r_k-1}\simeq S^{n+1}$. 
\end{example}

\section{Boolean simplices and simplicial prisms}\label{sec:bospx}
In this section we will give definitions and prove basic facts 
about the notions we will need in subsequent sections.

\begin{definition}
Let us call an affinely independent set of points $\{\boldsymbol{p}_1,\dots,\boldsymbol{p}_{r+1}\}\in\mathbb Z_2^N$
a~{\it boolean $r$-simplex} and denote it $\Delta_2^r$. By definition set $\Delta^{-1}=\varnothing$. 
Let us call a~set of points $S\subset\mathbb Z_2^N$
affinely equivalent to~the~direct product $\Delta^{r-1}_2\times\mathbb Z_2$ a~{\it boolean simplicial prism} and denote 
it $\Pi^r$. We have $\Pi^1=\mathbb Z_2=\Delta^1$ and $\Pi^2=\mathbb Z_2^2$.
\end{definition}
A boolean simplicial prism $\Pi^r$ consists of two disjoint boolean $(r-1)$-simplices (``bases'')
$\boldsymbol{a}_1,\dots,\boldsymbol{a}_r$ and $\boldsymbol{b}_1,\dots,\boldsymbol{b}_r$ in $\mathbb Z_2^r$
such that any two points $\boldsymbol{a}_i,\boldsymbol{b}_i$ form a~boolean line parallel
to~the~same vector $\boldsymbol{l}$ (``main direction'') that is not parallel to~bases. This means that 
$\boldsymbol{l}=\boldsymbol{a}_i+\boldsymbol{b}_i$ for all $i$, and the~disjoint union of any base
and a vertex of the other base is an $r$-simplex. It is easy to see that for any $i$ there is a unique 
affine isomorphism exchanging $\boldsymbol{a}_i$ and $\boldsymbol{b}_i$ and 
leaving all $\boldsymbol{a}_j$ and $\boldsymbol{b}_j$ with $j\ne i$ fixed. 
\begin{lemma}\label{affindab}
A subset of $\Pi^r=\{\boldsymbol{a}_1,\boldsymbol{b}_1,\dots,\boldsymbol{a}_r,\boldsymbol{b}_r\}$
is affinely independent if and only if it contains at most one pair $\{\boldsymbol{a}_i,\boldsymbol{b}_i\}$.
\end{lemma}
\begin{proof}
The proof is straightforward using the equality $\boldsymbol{a}_i+\boldsymbol{b}_i+\boldsymbol{a}_j+\boldsymbol{b}_j=0$.
\end{proof}
\begin{corollary}\label{cor:a2p}
A subset $S\subset \Pi^r$ is an affine $2$-plane if and only if 
$S=\{\boldsymbol{a}_i,\boldsymbol{b}_i,\boldsymbol{a}_j,\boldsymbol{b}_j\}$ for $i\ne j$.
\end{corollary}
\begin{proof}
Indeed, the points $\boldsymbol{a}_i,\boldsymbol{b}_i,\boldsymbol{a}_j$ are affinely
independent and $\boldsymbol{b}_j=\boldsymbol{a}_i+\boldsymbol{b}_i+\boldsymbol{a}_j$.
Hence, $\{\boldsymbol{a}_i,\boldsymbol{b}_i,\boldsymbol{a}_j,\boldsymbol{b}_j\}$ is an~affine $2$-plane. On the other hand,
if $S$ does not contain two pairs $\{\boldsymbol{a}_i,\boldsymbol{b}_i\}$, then 
$S$ is affinely independent.
\end{proof}
\begin{definition}
Consider two subsets $S_1$, $S_2$ of the affine space $\mathbb Z_2^N$. If the planes ${\rm aff}(S_1)$ and 
${\rm aff}(S_2)$ are skew, that is they do not intersect and
the intersection of the corresponding vector subspaces is zero, then we call the set $S_1\sqcup S_2$ 
a {\it join} of $S_1$ and $S_2$ and denote it $S_1*S_2$. If $S_1$ and $S_1'$ are affinely equivalent as well as 
$S_2$ and $S_2'$, then $S_1*S_2$ and $S_1'*S_2'$ are also affinely equivalent. 
Therefore, up to an affine equivalence we can define a join of any two sets $S_1$, $S_2\subset\mathbb Z_2^N$,
if we put them to skew planes. Then $(S_1*S_2)*S_3=S_1*(S_2*S_3)$.

A join of a set $S$ and a point $\boldsymbol{p}$ is called a {\it cone} over $S$ and is denoted $CS$.
By definition the~cone $CS$ is a disjoint union of $S$ and a point $\boldsymbol{p}\notin {\rm aff}(S)$. 
We have $C^kS=\Delta^{k-1}_2*S$.
The~boolean simplex $\Delta^r_2$ is a join of its vertices and a cone over $\Delta^{r-1}_2$.
\end{definition} 
\begin{lemma}\label{lem:SPK}
Any full-dimensional subset $S$ of $\Pi^r$ of cardinality $r+k$ is affinely isomorphic to
$\Delta^{r-k-1}_2*\Pi^k=C^{r-k}\Pi^k$. In particular, for $k=1$ it is $\Delta^r_2$, and
for $k=2$ it is $\Delta_2^{r-3}*\mathbb Z_2^2$.
\end{lemma}
\begin{proof}
Indeed, $S$ is affinely isomorphic to $\{\boldsymbol{a}_1,\dots, \boldsymbol{a}_r,\boldsymbol{b}_1,\dots, \boldsymbol{b}_k\}$.
We have 
$$
{\rm aff}(\boldsymbol{a}_1,\dots, \boldsymbol{a}_k,\boldsymbol{b}_1,\dots, \boldsymbol{b}_k)=
{\rm aff}(\boldsymbol{a}_1,\dots, \boldsymbol{a}_k, \boldsymbol{b}_1).
$$ 
This plane is skew with ${\rm aff}(\boldsymbol{a}_{k+1},\dots, \boldsymbol{a}_r)$, since the points 
$\{\boldsymbol{a}_1,\dots, \boldsymbol{a}_r,\boldsymbol{b}_1\}$ are affinely independent.
\end{proof}
\begin{lemma}
For $k\geqslant 2$ and $r\geqslant k+1$ the subsets $\Delta^{r-k-1}$ and $\Pi^k$
are affine invariants of~the~join $\Delta^{r-k-1}*\Pi^k$.
\end{lemma}
\begin{proof}
Indeed, $\Delta^{r-k-1}$ consists of points not lying in the affine hull of the rest points.
\end{proof}
\begin{lemma}\label{lem:pklk}
For $r\geqslant 3$, the main direction $\boldsymbol{l}$ 
is a unique direction $\boldsymbol{d}$ such that $\Pi^r$ consists of $r$ lines of this direction. 
In particular, $\boldsymbol{l}$ is an affine invariant of the boolean simplicial prism $\Pi^r$.
For $r=2$ we have $\Pi^2=\mathbb Z_2^2$ and any direction can be chosen as~main.
\end{lemma}
\begin{proof}
Indeed, if $\boldsymbol{d}\ne \boldsymbol{l}$,
then without loss of generality we may assume that one line consists of~$\boldsymbol{a}_i$
and $\boldsymbol{a}_j$. Then $\boldsymbol{d}=\boldsymbol{a}_i+\boldsymbol{a}_j=\boldsymbol{b}_i+\boldsymbol{b}_j$. 
Since $r\geqslant 3$, there are at least three lines. 
Then either $\boldsymbol{d}=\boldsymbol{a}_k+\boldsymbol{a}_l=\boldsymbol{b}_k+\boldsymbol{b}_l$
or $\boldsymbol{d}=\boldsymbol{a}_k+\boldsymbol{b}_l=\boldsymbol{b}_k+\boldsymbol{a}_l$ for some $k\ne l$ such that 
$\{k,l\}\cap \{i,j\}=\varnothing$. We obtain a contradiction to Lemma \ref{affindab}. 
\end{proof}
\begin{lemma}\label{lem:DPkl}
For $r\geqslant k\geqslant 3$ the main direction $\boldsymbol{l}$ of $\Pi^k$ is a unique direction $\boldsymbol{d}$
such that the~image of~$\Delta^{r-k-1}_2*\Pi^k$ under the~projection 
$\mathbb Z_2^r\to \mathbb Z_2^r/\langle\boldsymbol{d}\rangle$
is $\Delta^{r-1}_2$. 

For $k=2$ such directions are  three main directions of $\Pi^2$. 

For $k=1$ there are $\frac{r(r-1)}{2}$ such directions corresponding 
to the pairs of vertices of $\Delta^{r-k-1}_2*\Pi^k=\Delta^r_2$.
\end{lemma}
\begin{proof}
For $k=1$ the statement is trivial. Assume that $k\geqslant 2$. 
The set  $\Delta^{r-k-1}_2*\Pi^k$ consisting of $r+k\geqslant r+2$ points lies on $r$ lines of direction $\boldsymbol{d}$.
Since at least two lines contain two points and any point of $\Delta^{r-k-1}_2$ does not lie in the affine hull
of all the other points of $\Delta^{r-k-1}_2*\Pi^k$, each point of $\Delta^{r-k-1}_2$ is a single point on the corresponding line.
Hence, $\Pi^k$ consists of $k$ lines of direction $\boldsymbol{d}$ and by Lemma \ref{lem:pklk} these
lines have a main direction. Lemma \ref{lem:SPK} implies that a~main direction satisfies the desired condition.
\end{proof}

\section{Special hyperelliptic manifolds $N(P,\Lambda)$}\label{sec:shm}
\begin{definition}
Following \cite{VM99S1} we call a closed $n$-manifold $M$ {\it hyperelliptic} if it has an involution $\tau$ such that
the orbit space $M/\langle\tau\rangle$ is homeomorphic to an $n$-sphere. The corresponding involution $\tau$
is called a {\it hyperelliptic involution}.
\end{definition}
In this section we consider hyperelliptic involutions $\tau$ in the group $\mathbb Z_2^{r+1}$ canonically acting 
on~the~closed manifold $N(P,\Lambda)$ corresponding to~a~vector-coloring of rank $r+1$ of~a~simple $n$-polytope $P$.  
By Corollary \ref{H'or} the~manifold $N(P,\Lambda)$ should be orientable. Hence, $N(P,\Lambda)=N(P,\lambda)$ 
for an~affine coloring $\lambda$ of rank $r$. Moreover, by Corollary \ref{cor:affH'} the involution 
$\tau$ preserves the orientation, that is 
$\tau\in\mathbb Z_2^r=H_0'$. Corollary \ref{cor:affH'P} implies the following result.
\begin{lemma}\label{lem:affhl}
Let $\lambda$ be an affine coloring of rank~$r$ of a simple $n$-polytope $P$.
An involution $\tau\in\mathbb Z_2^r$ is~hyperelliptic if~and only if 
$N(P,\lambda)/\langle\tau\rangle=N(P,\lambda_\tau)$ is homeomorphic to $S^3$, where
$\lambda_\tau$ is~the~composition $\widehat{\Pi}\circ\lambda$ of $\lambda$ and the~affine 
surjection $\widehat{\Pi}\colon \mathbb Z_2^r\to \mathbb Z_2^r/\langle\tau\rangle$.
\end{lemma}
\begin{definition}
Let us call an involution $\tau\in\mathbb Z_2^r$ {\it special}, if the complex $\mathcal{C}(P,\lambda_{\tau})$ is equivalent 
to $\mathcal{C}(n,r)$.
\end{definition}
\begin{proposition}
Any special involution is hyperelliptic.
\end{proposition}
\begin{proof}
This follows from Construction \ref{con:spn}. 
\end{proof}
\begin{definition}
Let us call a manifold $N(P,\lambda)$ equipped with a special involution $\tau$ a {\it special hyperelliptic manifold of rank $r$}.
\end{definition}

It follows from the definition that any special hyperelliptic manifold is obtained by~the~following construction.
\begin{construction}[A special hyperelliptic manifold]\label{con:shm}
Let $P$ be a simple $n$-polytope and $c$ be its coloring in $r\geqslant 1$ colors such that the~complex 
$\mathcal{C}(P,c)$ is equivalent to $\mathcal{C}(n,r)$. Let $G_1$, $\dots$, $G_r$ be its facets. Choose any
coloring $\chi$ of $P$ in two colors $0$ and $1$ such that at least one restriction $\chi\left.\right|_{G_i}$ is non-constant.
Define a space $N(P, c,\chi)=N(P,\lambda(c,\chi))$, where 
$$
\lambda(c,\chi)(F_i)=
\begin{cases}\boldsymbol{a}_{c(i)},&\text{if }\chi(F_i)=1,\\
\boldsymbol{b}_{c(i)},&\text{if }\chi(F_i)=0,
\end{cases}
$$
and $\{\boldsymbol{a}_1,\dots,\boldsymbol{a}_r,\boldsymbol{b}_1,\dots,\boldsymbol{b}_r\}\subset\mathbb Z_2^r$
is a~boolean simplicial prism of dimension $r$. If we exchange the colors $0$ and $1$ at one 
facet $G_i$, then $\lambda(c,\chi)$ will be changed to an affinely equivalent coloring, and 
the weakly equivariant type of $N(P, c,\chi)$ will remain the same.
If $N(P,c,\chi)$ is~a~manifold, then by definition it~is~a~special hyperelliptic manifold of rank $r$ 
with the special  involution $\boldsymbol{l}=\boldsymbol{a}_i+\boldsymbol{b}_i\in\mathbb Z_2^r$. 
\end{construction}
\begin{remark}
The image of the mapping $\lambda(c,\chi)\colon\{F_1,\dots,F_m\}\to \mathbb Z_2^r$
consists of $r+k$ points if and only if $\chi\left.\right|_{G_i}$ is non-constant exactly for $k$ facets $G_i$.
\end{remark}

Lemma \ref{affindab} and Corollary \ref{cor:a2p} imply the following criterion.
\begin{corollary}\label{cor:hypman}
The space $N(P,c,\chi)$ is a manifold if and only if
one of the following equivalent conditions hold:
\begin{enumerate}
\item $F_i\cap F_j\cap F_k\cap F_l=\varnothing$ whenever $c(F_i)=c(F_j)\ne c(F_k)=c(F_l)$
and $\chi(F_i)=\chi(F_k)\ne\chi(F_j)=\chi(F_l)$;
\item  $F_i\cap F_j\cap F_k\cap F_l=\varnothing$ whenever 
$\lambda(c,\chi)(\{F_i,F_j,F_k,F_l\})=\{\boldsymbol{a}_p,\boldsymbol{b}_p, \boldsymbol{a}_q,\boldsymbol{b}_q\}$ for $p\ne q$.
\item  $F_i\cap F_j\cap F_k\cap F_l=\varnothing$ whenever 
$\lambda(c,\chi)(\{F_i,F_j,F_k,F_l\})$ is an affine $2$-plane.
\end{enumerate}
\end{corollary}
\begin{corollary}\label{n3chi}
In dimension $n=3$ in Construction \ref{con:shm} the space $N(P,c,\chi)$ is a special hyperelliptic manifold for any $\chi$.
\end{corollary}
\begin{remark}
Corollary \ref{n3chi} also follows from Corollary \ref{cor:manor}.
\end{remark}
\begin{proposition}\label{prop:conhi}
Any special hyperelliptic manifold can be obtained by Construction \ref{con:shm}.
\end{proposition}
\begin{proof}
Indeed, if $\tau$ is a special involution on the manifold $N(P,\lambda)$, then 
$\mathcal{C}(P,\lambda_{\tau})\simeq \mathcal{C}(n,r)$. Hence, we can choose $c=\lambda_{\tau}$.
The image of $c$ consists of affinely independent points $\boldsymbol{p}_1$, $\dots$, 
$\boldsymbol{p}_r\in \mathbb Z_2^r/\langle\tau\rangle$ corresponding to~facets $G_1$, $\dots$, $G_r$ of 
$\mathcal{C}(P,\lambda_{\tau})$. Let $\widehat{\Pi}\colon \mathbb Z_2^r\to\mathbb Z_2^r/\langle\tau\rangle$ 
be the~canonical projection.
Choose for each $i$ some facet $F_{j_i}\subset G_i$ and set $\boldsymbol{a}_i=\lambda_{j_i}$. Then the~points 
$\boldsymbol{a}_1$, $\dots$, $\boldsymbol{a}_r$ are affinely independent and 
$\widehat{\Pi}^{-1}(\boldsymbol{p}_i)=\{\boldsymbol{a}_i,\boldsymbol{b}_i\}$ for each $i$, where 
$\boldsymbol{b}_i=\boldsymbol{a}_i+\tau$.
Thus, setting $\boldsymbol{l}=\tau$ and $\chi(F_i)=\begin{cases}1,&\text{if }\lambda_i=\boldsymbol{a}_i,\\
0,&\text{if }\lambda_i=\boldsymbol{b}_i
\end{cases}$ we finish the proof.
\end{proof}

Now let us enumerate all special involutions on a manifold $N(P,\lambda)$. 
\begin{proposition}\label{prop:clhi}
Let $N(P,c,\chi)$ be a special hyperelliptic manifold of rank $r$ and $G_1$, $\dots$, $G_r$ be facets
of $\mathcal{C}(P,c)\simeq \mathcal{C}(n,r)$. 
\begin{itemize}
\item If $\chi\left.\right|_{G_i}$ is non-constant exactly for one facet $G_i$ (that is, 
the image of $\lambda(c,\chi)\colon\{F_1,\dots,F_m\}\to \mathbb Z_2^r$
is a boolean simplex), 
then $\tau\in\mathbb Z_2^r$ is a~special involution if and only if the~vector
$\tau$ connects two vertices of the simplex and $\mathcal{C}(P,\lambda_\tau)\simeq\mathcal{C}(n,r)$. 
There are at most $\frac{r(r+1)}{2}$ such involutions.
\item If $\chi\left.\right|_{G_i}$ is non-constant exactly for two facets $G_i$ and $G_j$
(that is, the image of $\lambda(c,\chi)$ is $\Delta^{r-3}*\Pi^2$), then $\tau\in\mathbb Z_2^r$ is 
a~special involution if and only if $\tau\in\{\boldsymbol{l},\boldsymbol{a}_i+\boldsymbol{a}_j,
\boldsymbol{a}_i+\boldsymbol{b}_j\}$ (that is, $\tau$ is a main direction of $\Pi^2$) 
and $\mathcal{C}(P,\lambda_\tau)\simeq\mathcal{C}(n,r)$.
There are at most three such involutions.
\item If $\chi\left.\right|_{G_i}$ is non-constant  for more than two facets $G_i$ (that is, 
the image of $\lambda(c,\chi)$ is $\Delta^{r-k-1}*\Pi^k$ for $k\geqslant 3$), then $\tau\in\mathbb Z_2^r$ is 
a~special involution if and only if $\tau=\boldsymbol{l}$ (the main direction of $\Pi^k$). 
That is, there is only one special involution.   
\end{itemize} 
\end{proposition}
\begin{proof}
The Proposition follows from Lemma \ref{lem:DPkl}.
\end{proof}

We can summarise the above results as follows.
\begin{definition}\label{def:iml}
For an affine coloring $\lambda$ of rank $r$ of a simple $n$-polytope $P$ denote
$I(\lambda)=\{\lambda_1,\dots,\lambda_m\}\subset\mathbb Z_2^r$. For a subset $S\subset\mathbb Z_2^r$
denote $G(S)=\bigcup\limits_{q\colon \lambda_q\in S}F_q\subset \partial P$.
\end{definition}
\begin{theorem}\label{th:shmclass}
Let $\lambda$ be an~affine coloring of rank $r$ of~a~simple $n$-polytope $P$. 
The space $N(P,\lambda)$ is a~special hyperelliptic manifold if and only if $1\leqslant r\leqslant n+1$ and 
one of the following conditions hold:
\begin{enumerate}
\item $I(\lambda)=\{\boldsymbol{p}_1, \dots, \boldsymbol{p}_{r+1}\}$ is a boolean $r$-simplex, 
and at least for one direction $\tau=\boldsymbol{p}_i+\boldsymbol{p}_j$, $i\ne j$, 
the complex $\mathcal{C}(P,\lambda_\tau)$ is equivalent to $\mathcal{C}(n,r)$. 
In this case each special involution $\tau\in \mathbb Z_2^r$ has this form and there are at
most $\frac{r(r+1)}{2}$ such involutions. 
\item $I(\lambda)=\Delta^{r-3}*\Pi^2$, $\bigcap_{\lambda_j\in\Pi_2}G(\lambda_j)=\varnothing$
and at least for one main direction $\tau$ of $\Pi^2$
the~complex $\mathcal{C}(P,\lambda_\tau)$ is equivalent to $\mathcal{C}(n,r)$. In this case each special involution 
$\tau\in \mathbb Z_2^r$ has this form and there are at most three such involutions. 
\item $I(\lambda)=\Delta^{r-k-1}*\Pi^k$, $k\geqslant 3$, $\bigcap_{\lambda_j\in\Pi_2}G(\lambda_j)=\varnothing$ 
for any $2$-plane  $\Pi^2\subset \Pi^k$, and for the~main direction $\tau$ of $\Pi^k$ the complex 
$\mathcal{C}(P,\lambda_\tau)$ is equivalent to $\mathcal{C}(n,r)$. In this case the~main direction $\tau$
is a unique special  involution in $\mathbb Z_2^r$.
\end{enumerate}
Moreover, in all these cases any vertex of $\mathcal{C}(P,\lambda)$ belongs 
to~the~$1$-skeleton of  $\mathcal{C}(P,\lambda_\tau)\simeq \mathcal{C}(n,r)$.
\end{theorem}
\begin{proof}
The proof follows from Corollary \ref{cor:hypman}, Propositions \ref{prop:conhi} and  \ref{prop:clhi}, and  Corollary \ref{cor:C01inv}.
\end{proof}
We will specify this result for $3$-dimensional polytopes in Section \ref{sec:3hm}.

\begin{corollary}
If $N(P,\lambda)$ is a~special hyperelliptic manifold of rank $r$,  where $1\leqslant r\leqslant n-2$, then 
the~complex $C(P,\lambda)$ has no vertices. 
\end{corollary}
\begin{proof}
If follows from the fact that the $1$-skeleton of the complex $\mathcal{C}(n,r)\simeq \mathcal{C}(P,\lambda_\tau)$ 
is empty for $r\leqslant n-2$, since the intersection of all its facets is $S^{n-r}$, $n-r\geqslant 2$.
\end{proof}
\begin{example}
Example \ref{ex:spnG} produces the following special hyperelliptic manifolds.
Each face $G=F_{i_1}\cap\dots \cap F_{i_k}$ and an epimorphism  
$\chi\colon \{F_1,\dots,F_m\}\setminus\{F_{i_1},\dots, F_{i_k}\}\to\{0,1\}$ correspond to an~affine coloring of~rank $k+1$ 
$$
\lambda_i=\begin{cases}
\boldsymbol{e}_s,&\text{ if }i=i_s, s=1,\dots, k;\\
\boldsymbol{e}_{k+1},&\text{ if }i\notin\{i_1,\dots, i_k\}\text{ and }\chi(F_i)=1;\\
\boldsymbol{0},&\text{ if }i\notin\{i_1,\dots, i_k\}\text{ and }\chi(F_i)=0,
\end{cases}
$$
where $\boldsymbol{e}_1=(1,0,\dots, 0)$, $\dots$, 
$\boldsymbol{e}_{k+1}=(0,\dots,1)\in\mathbb Z_2^{k+1}$.
Then the subgroup $H_{G,\chi}=H(\lambda)$ of rank $m-k-2$ 
is defined in $\mathbb Z_2^m$ by the equations $x_{i_1}=0$, $\dots$, $x_{i_k}=0$,
$x_1+\dots+x_m=0$, and $\sum_{i\colon \chi(F_i)=1}x_i=0$. The space $N(P,\lambda)$ is a special hyperelliptic manifold of
rank $k+1$ with a special involution $\boldsymbol{e}_{k+1}\in\mathbb Z_2^{k+1}$.
\end{example}
\begin{example}\label{ex:cnindep}
If $\lambda$ is an affinely independent coloring of a simple $n$-polytope $P$ and $N(P,\lambda)$
is a special hyperelliptic manifold of rank $r$, then $n-1\leqslant r\leqslant n+1$, and all the vertices
of $P$ belong to the $1$-skeleton of $\mathcal{C}(P,\lambda_{\tau})\simeq \mathcal{C}(n,r)$, which is a subset of 
the graph of $P$. For $r=n-1$, this $1$-skeleton is a single circle without vertices. We have a simple edge-cycle in the graph
of $P$ containing all its vertices. Such cycles are called {\it Hamiltonian}.
For $r=n$ the $1$-skeleton of $\mathcal{C}(P,\lambda_{\tau})\simeq \mathcal{C}(n,r)$ is a 
graph with two vertices and $n$ multiple edges. For $r=n+1$
it is a complete graph $K_{n+1}$.
\end{example}
\begin{example}
For $n=1$ the only small cover over $P=I^1=\Delta^1$ is $N(P,\lambda)=\mathbb RP^1\simeq S^1$, and it is not a special hyperelliptic manifold.

For $n=2$ any orientable small cover $N(P_k,\lambda)$ over a $k$-gon $P_k$ is a special hyperelliptic manifold. In this
case $k$ is even and $\lambda$ corresponds to a coloring of edges of $P_k$ in two colors
such that adjacent edges have different colors.

For $n=3$ special hyperelliptic small covers $N(P,\lambda)$ correspond to Hamiltonian cycles on $P$.
We will see such examples in Sections \ref{Sec:RHS} and \ref{sec:3Ham}. For example, there
is a special hyperelliptic small cover over the dodecahedron with three special involutions, see
Fig. \ref{Hamdod}. It is a classical fact that not any simple
$3$-polytope admits a Hamiltonian cycle (see \cite{T46, G68}).

For $n=4$ if a polytope $P$ admits a special hyperelliptic small cover, then $P$ has
a Hamiltonian cycle $\gamma$ and all the facets of $P$ can be colored in $3$ colors in such a way that
any edge of $\gamma$ is an~intersection of $3$ facets of different colors. 
Moreover, the union of all the facets of each color is a $3$-disk. Since $P$ has at least $5$ facets, 
there are two adjacent facets $F_i$ and $F_j$ of the same color. Then 
no edge of the polygon $F_i\cap F_j$ belongs to $\gamma$, and
at each vertex of this polygon $\gamma$ passes through two complementary edges of $P$. Then 
the colors of~the~facets $F_k\ne F_i, F_j$ containing the successive edges of $F_i\cap F_j$ alter.
Thus, $F_i\cap F_j$ has an even number of edges. Moreover, at each vertex of $P$ there are 
exactly two facets of the same color. Therefore, this vertex lies on exactly one such an even-gon.

\begin{proposition}
If a simple $4$-polytope $P$ admits a special hyperelliptic small cover, then all the vertices of $P$ lie on a disjoint
union of $2$-faces with even numbers of edges.
\end{proposition} 
\begin{corollary}
The simplex $\Delta^4$ and the $120$-cell have no special hyperelliptic small covers.
\end{corollary}

Moreover, it can be shown than the products $\Delta^3\times I$, $\Delta^2\times I^2$,  $\Delta^2\times\Delta^2$,
and the cube $I^4$ also admit no special hyperelliptic small covers.

It will be shown in \cite{E24b} that if a $4$-polytope admits a special hyperelliptic small cover, then it 
has a triangular or a quadrangular $2$-face. In particular,  this is impossible for any compact right-angled hyperbolic $4$-polytope.

An example of a four-dimensional hyperelliptic small cover was built by Alexei Koretskii \cite{K24} over a polytope with $9$ facets.
The vertices of this polytope lie on a disjoint union of $6$ quadrangles, and $9$ facets are split into $3$ triples of~the~same color.
\end{example}

\section{A structure of the complex $\mathcal{C}(P,c)$ for $3$-polytopes}\label{sec:CPc3}
\subsection{Basic facts from the graph theory}
\begin{agreement}
In this article by a spherical graph we mean a graph realized on the sphere $S^2$ 
piecewise linearly in some triangulation of $S^2$.
\end{agreement}
For additional details on the graph theory see \cite{BE17I}.
\begin{definition}
A graph is {\it  simple} if it has no loops and multiple edges.

Following \cite{Z95} we call a~connected graph $G$ with at least two edges {\it $2$-connected} if 
it has no loops and a~deletion of any vertex with all incident edges leaves the graph connected.

A connected graph $G$ with at least four edges is called {\it $3$-connected},
if it is simple and a~deletion of any vertex or any two 
vertices with all incident edges leaves the graph connected.

A {\it face} of a spherical graph $G\subset S^2$ is a connected component of the complement
$S^2\setminus G$. A~vertex or an edge of $G$ is {\it incident} to a face if it belongs to its closure. 
By definition a vertex of an edge is incident to it.

Two spherical graphs are called {\it combinatorially equivalent}, if there is a bijection
between the~sets of their vertices, edges and faces preserving the incidence relation.  

A {\it bridge} of a graph $G$ is an edge such that a deletion of this edge makes
the~graph disconnected.
\end{definition}
The proof of the following classical facts can be found in \cite[Lemmas 2.4.1 and 2.4.2]{BE17I} and \cite[Lemma 1.27]{BE17S}.
\begin{lemma}\label{lem:sgc}
A spherical graph $G$ with more than one vertex is connected if and only if any its face is a disk (equivalently, has 
one connected component of the boundary).
\end{lemma}

\begin{lemma}
A simple spherical graph $G$ with more than one vertex is $3$-connected if
and only if any its face is bounded by a simple cycle and if the boundary cycles of two faces intersect,
then their intersection is~a~vertex or~an~edge.  
\end{lemma}
To characterize the graphs of $3$-polytopes we will use the following result (see \cite{Z95}).
\begin{theorem}[The Steinitz theorem]\label{th:St}
A simple graph $G$ is a graph of some $3$-polytope if and only if it is planar and $3$-connected.
\end{theorem}
Moreover, by H.~Whitney's theorem (see \cite{Z95}) any two spherical realizations of the graph of a $3$-polytope
are combinatorially equivalent.  
\begin{corollary}
A connected  simple spherical graph with more than one vertex is combinatorially equivalent to a graph of
a $3$-polytope if and only if any its face is bounded by a simple cycle and if the~boundary cycles of two faces intersect,
then their intersection is~a~vertex or~an~edge.  
\end{corollary}

\begin{lemma}\label{lem:2con}
For a  connected $3$-valent spherical graph $G$ the following
conditions are equivalent:
\begin{enumerate}
\item $G$ is $2$-connected (in particular, it has no loops);
\item $G$ has no bridges;
\item any face of $G$ is a disk bounded by a~simple edge-cycle.
\end{enumerate}
\end{lemma}
\begin{proof}
If $G$ is $2$-connected, then it has no bridges, since the deletion of any 
vertex of a bridge disconnects the graph. If $G$ has no bridges, then it has no loops since
the vertex of a loop necessarily belongs to a bridge.  Also $G$ has at least $3$ edges, since it is $3$-valent.
If~a~deletion of~a~vertex and incident edges makes the graph disconnected, then at least one edge in this vertex
is a bridge. A contradiction. Thus, $G$ is $2$-connected and items (1) and (2) are equivalent.

If each face of $G$ is a disk bounded by a~simple edge-cycle, then  $G$ has no bridges since 
a~bridge has the~same face on both sides and the boundary cycle of this face is not simple. 
Let $G$ have no bridges. Since $G$ is connected, each its face is a disk. If 
a~boundary cycle passes a~vertex more than once, then it passes an edge more than once since $G$ is $3$-valent.
Then this edge has the same face on both sides. Hence, it is a bridge, which is a contradiction.  Thus, items (2) and (3) are equivalent. 
\end{proof}
\begin{lemma}\label{lem:even}
Any $3$-valent graph $G$ has an even number of vertices. 
\end{lemma}
\begin{proof}
Indeed, if we cut each edge in two parts, then each vertex is incident to three such parts, hence 
$3V=2E$, where $V$ and  $E$ are numbers of vertices and edges. In particular, $V$ is even. 
\end{proof}

\subsection{A characterization of complexes $\mathcal{C}(P,c)$ of $3$-polytopes}
In dimension $n=3$ each facet of the complex $\mathcal{C}(P,c)$ with a non-constant mapping $c$
is a sphere with holes. Its boundary consists of $1$-faces and $0$-faces, which we call vertices. 
Each $1$-face belongs to two different facets and each vertex -- to three different facets and three different $1$-faces. 
Each $1$-face is~either the~whole circle without vertices, 
or a simple path connecting two different vertices. 
\begin{definition}
We call $1$-faces of $\mathcal{C}(P,c)$ containing no vertices {\it circles}, 
and $1$-faces connecting two vertices {\it edges}. 
\end{definition}

Consider the $1$-skeleton $\mathcal{C}^1(P,c)$, which is the union of all vertices and $1$-faces. Each its connected component
is either a circle without vertices or a connected $3$-valent spherical graph without loops and bridges. 
Indeed, a bridge should have the same facet  on both sides,
hence it can not be an intersection of two different facets.  
\begin{theorem}\label{th:CPc3}
Complexes $\mathcal{C}(P,c)$ corresponding to $3$-polytopes $P$ are exactly subdivisions 
of~the~$2$-sphere arising from disjoint unions (perhaps empty) of simple closed curves and 
connected $3$-valent graphs without bridges.  
\end{theorem}
\begin{proof}
We have already proof the theorem in one direction. Consider the other direction. By~Lemma \ref{lem:2con}
each connected $3$-valent spherical graph without bridges has no loops.
We~will call by ``facets'' the connected components of~the~complement 
in $S^2$ to a disjoint union of~simple closed curves and connected $3$-valent graphs without bridges, and
by ``circles'' simple closed curves from the union. 

The empty union corresponds to a constant function $c$ on any polytope. Now let us assume that the union is non-empty.

Consider a facet $C$ and a component $\gamma$ of $\partial C$ that is not a circle.
There is a vertex on $\gamma$. This vertex belongs to three different edges and to closures of three different facets,
for otherwise some of the edges is a bridge. Two of these edges belong to $\gamma$ and 
the third edge does not belong. Then $\gamma$ is a simple edge-cycle, since it passes each vertex at most 
once. Also $C$ is~a~sphere with holes bounded by such simple edge-cycles  and circles from the union. Each 
edge or~circle belongs to~the~closures of exactly two different facets, and each vertex -- to the closures of three different facets. 

Now we will add edges to this data to obtain a $1$-skeleton
of some simple $3$-polytope. Each edge will have two new different $3$-valent vertices and 
will divide a facet into two new different facets. 
If a facet $C$ is not a disk, we can first add edges connecting points on the same
boundary component to subdivide $C$ into rings, and then for each ring add three edges connecting different boundary
components to subdivide it into three ``quadrangles'' (see Fig. \ref{S2holes}a). 
\begin{figure}[h]
\begin{center}
\includegraphics[width=0.6\textwidth]{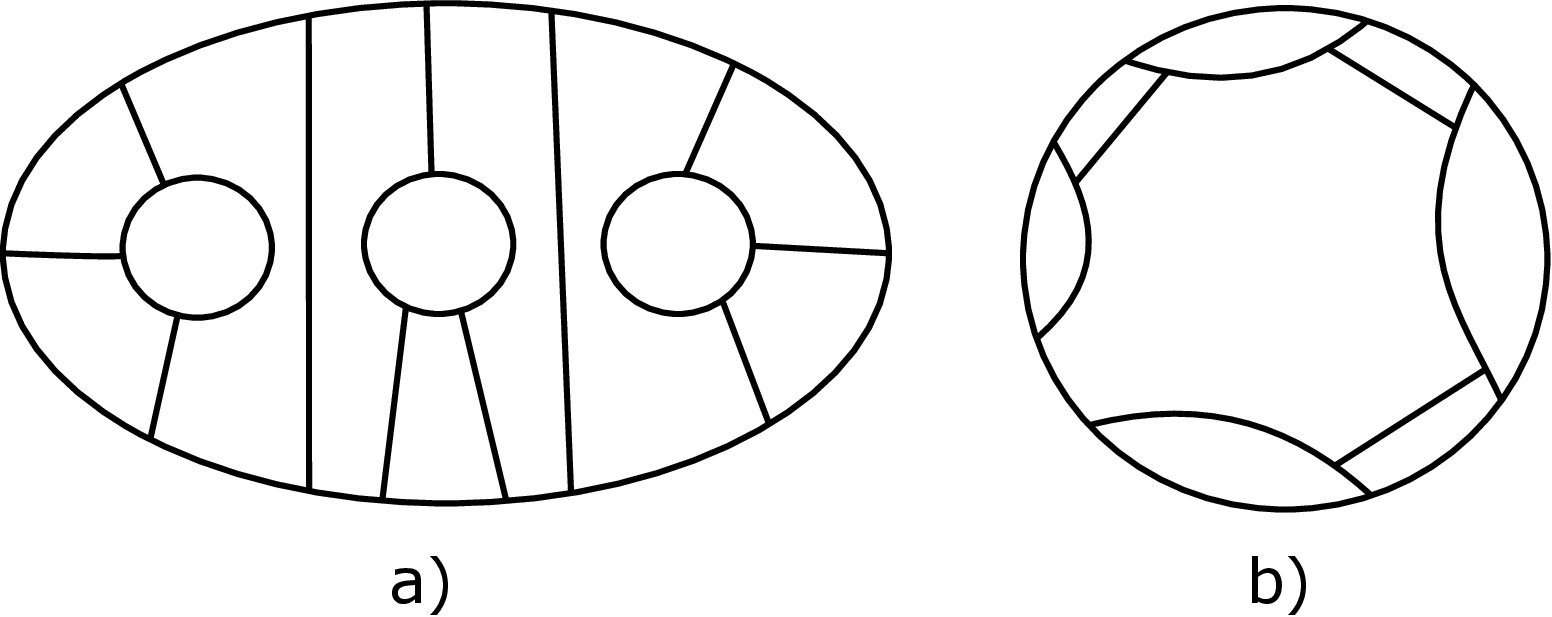}
\end{center}
\caption{a) A subdivision of a sphere with holes; b) Cutting off the common edges}\label{S2holes}
\end{figure}

After this procedure we obtain a new subdivision of a $S^2$ with $3$-valent
vertices and each facet being a disk bounded by a simple edge-cycle or a circle without vertices. 
In the latter case $\mathcal{C}(P,c)$ consists of two disks glued along the common boundary circle. 
We can add two edges to these disks to obtain the boundary complex of a simplex. 
Thus, we can assume that each facet has at least one vertex on the boundary.  Then there are at least two vertices, 
for otherwise the adjacent facet is not bounded by a simple cycle. If there are exactly 
two vertices, we add an edge separating the $2$-gon into two triangles. Repeating this step for all $2$-gons,
we obtain a~$3$-valent partition of $S^2$ into polygons with at list $3$-edges. The graph defining this partition is simple. 
Indeed, there are no loops by construction. 
If two edges have the same vertices,  then they form a simple closed curve dividing the sphere into two disks. 
The third edges at both vertices should lie in the same disk, for
otherwise there arise two equal facets in both vertices. Thus, two multiple edges bound a $2$-gon. A contradiction. 

At the end of this step we obtain a simple spherical graph with each facet bounded by a~simple cycle with at least $3$ edges.
Now we will add edges to this partition to obtain another $3$-valent partition such that
each facet is bounded by a simple edge-cycle with at least $3$ edges 
and the closures of two different facets have at most one edge in common. The last condition
is equivalent to the condition that all the edges of any facet belong to different facets surrounding it.  
The graph of the new partition is $3$-connected and by the Steinitz theorem it 
corresponds to a boundary of a simple $3$-polytope $P$. Then the original complex 
is obtained from $P$ by a sequence of operations of a deletion of an edge and has the form $\mathcal{C}(P,c)$,
where  $c(F_i)=c(F_j)$ if and only if the facets of $P$ belong to the same facet in the initial partition.

Now let us describe the last step.  If  the closure of a facet has with the closure of another facet more than one common edge, 
then their intersection consists of a disjoint set of edges lying on the boundary of each facet.  
We can ``cut off'' all but one these edges. Namely,
for each edge we add inside the first facet an edge with vertices on its boundary close to
the vertices of the chosen edge outside it.  As a result the edge is substituted by a quadrangle 
adjacent to $4$ different facets (see Fig. \ref{S2holes}b).  
Repeating this procedure we will obtain a new partition of the sphere such that
all the edges of the chosen facet belong to different facets and all the arising quadrangles also satisfy this condition.
Applying this argument to all the facets one by one we see that at each step there arise no new ``bad'' facets, and their  total 
number is decreasing~by~one. 
\end{proof}

\section{A criterion when $N(P,\Lambda)$ is a sphere for $3$-polytopes}\label{Sec:3sph}
In this section we will give a criterion when a manifold $N(P,\Lambda)$ corresponding to a vector-coloring $\Lambda$ of rank 
$r+1$ of a simple $3$-polytope is homeomorphic to a sphere $S^3$. Since $N(P,\Lambda)$ should be closed and orientable,
it has the form $N(P,\lambda)$ for an affine coloring  $\lambda$ of rank $r$. Thus we will consider only affine colorings.

Following \cite{VM99S1} we call a $3$-valent graph consisting of $2$ vertices and three multiple edges connecting them
a {\it theta-graph}. By $K_n$ we denote a complete graph on $n$-vertices. 

\begin{theorem}\label{th:NP3sp}
Let $\lambda$ be an~affine coloring  of rank $r$ of a~simple $3$-polytope $P$. The space $N(P,\lambda)$ is 
homeomorphic to $S^3$ if and only if $\mathcal{C}(P,\lambda)$ is equivalent to $\mathcal{C}(3,r+1)$. In other words, 
if and only if one of the~following conditions holds: 
\begin{enumerate}
\item $r=0$ and $\mathcal{C}^1(P,\lambda)$ is empty;
\item $r=1$ and  $\mathcal{C}^1(P,\lambda)$ is a circle;
\item $r=2$ and $\mathcal{C}^1(P,\lambda)$ is a theta-graph;
\item $r=3$ and $\mathcal{C}^1(P,\lambda)$ is the complete graph $K_4$.
\end{enumerate}
In all these cases the image of $\lambda$ is a boolean $(r+1)$-simplex.
\end{theorem}
\begin{remark}
The spheres in the theorem arise in Construction \ref{con:spn} and can be imagined as follows.
In the first case $S^3$  is glued of two copies of a polytope along the boundaries. In the second case  -- of $4$ copies of the ball  
with the boundary sphere subdivided into two hemispheres. If we glue two copies along the hemispheres we obtain a ball with the boundary subdivided into two hemispheres again. Then we glue two copies of this ball along boundaries.
In the third case the sphere $N(P,\lambda)$ is glued of $8$ copies of the ball 
with the boundary sphere subdivided into three $2$-gons by the theta-graph. 
Let the vertices of the theta-graph be the north and the south poles and edges be three meridians. The sphere and the ball
are subdivided by the equatorial plane into two balls combinatorially equivalent to a $3$-simplex $\Delta^3$. 
Then $8$ copies of this simplex are glued at one vertex to an octahedron as the coordinate octants in $\mathbb R^3$. 
The resulting sphere is glued of two copies of this octahedron along the boundaries.
In the case of $K_4$ the space $N(P,\lambda)$ is homeomorphic to $\mathbb R\mathcal{Z}_{\Delta^3}\simeq S^3$.
All these $4$ cases arise if we subdivide the standard $3$-sphere in $\mathbb R^4$ into $3$-disks by 
$1$, $2$, $3$, or $4$ coordinate hyperplanes.
\end{remark}
\begin{remark}
It will be shown in \cite{E24b} that analogs of Theorem \ref{th:NP3sp} and Corollary \ref{cor:3hem} hold for $n=4$.
\end{remark}

\begin{proof}[Proof of Theorem \ref{th:NP3sp}]
The ``if'' direction follows from Construction \ref{con:spn}. 

Now let us prove the theorem in the ``only if'' direction. By Corollary \ref{cor:manor} $N(P,\lambda)$ is 
a~closed  orientable $3$-manifold for any affine coloring $\lambda$ of a~simple $3$-polytope $P$.

If a facet $G_j$ of $\mathcal{C}(P,\lambda)$ is a sphere with at least two holes,  
then there is a simple closed curve $ \gamma$ inside $G_j$ separating its 
two boundary components.
Then $\mathcal{C}(P,\lambda)$ can be represented as~a~connected sum of complexes $\mathcal{C}(P,\lambda')$
and $\mathcal{C}(P,\lambda'')$ arising if we change the points of~the~affine coloring at all the facets of $P$ 
inside one of the connected component of $\partial P\setminus\gamma$ to $\lambda_j$. 
Denote $r'={\rm rk }\,\lambda'$ and $r''={\rm rk }\,\lambda''$.  
Both spaces $N(P,\lambda')$ and $N(P,\lambda'')$ are closed orientable manifolds by Corollary \ref{cor:manor}.
\begin{lemma}\label{ml1} There is a homeomorphism 
\begin{equation}\label{eq:s2holes}
N(P,\lambda)\simeq N(P,\lambda')^{\#2^{r-r'}}\#N(P,\lambda'')^{\#2^{r-r''}}
\#(S^2\times S^1)^{\#\left[2^r-2^{r-r'}-2^{r-r''}+1\right]}
\end{equation}
\end{lemma}   
The proof is similar to the proof of \cite[Proposition 3.6]{E22M}.
\begin{corollary}\label{mcor1}
If $\mathcal{C}(P,\lambda)$ has a facet, which is a sphere with at least two holes, 
then in~the~Knezer-Milnor prime decomposition of the orientable manifold 
$N(P,\lambda)$ there is a summand $S^1\times S^2$. In particular,
$N(P,\lambda)$ is not homeomorphic to a sphere and it is not a homology sphere for any coefficient group.
\end{corollary}
\begin{proof}
Indeed, in the~Knezer-Milnor decomposition  of $N(P,\lambda)$  
there is a summand $\#(S^2\times S^1)^{\#\left[2^r-2^{r-r'}-2^{r-r''}+1\right]}$.
But $1\leqslant r',r''\leqslant r$, since on both sides of the curve $\gamma$ there is a facet  with $\lambda_i\ne \lambda_j$,
where  $G_j$ is a chosen facet, which is a sphere with at least two holes.
Also $r'+r''=r+\dim{\rm aff}(\lambda')\cap{\rm aff}(\lambda'')\geqslant r$, 
since $\lambda_j\in {\rm aff}(\lambda')\cap {\rm aff}(\lambda'')$. Hence,
\begin{multline*}
2^r-2^{r-r'}-2^{r-r''}+1=2^{r-r'}(2^{r'}-1)-(2^{r-r''}-1)\geqslant \\
\geqslant 2^{r-r'}(2^{r'}-1)-(2^{r'}-1)=(2^{r-r'}-1)(2^{r'}-1)\geqslant 0
\end{multline*}
Moreover, if the left part is equal to zero, then $r'=r-r''$ and either $r'=0$ or $r=r'$ (then $r''=0$). A contradiction.
\end{proof}
If a facet of $\mathcal{C}(P,\lambda)$ is the whole sphere, then $\mathcal{C}^1(P,\lambda)=\varnothing$.
Thus, we can assume that each facet of $\mathcal{C}(P,\lambda)$ is a disk. 
If the intersection of two facets
$G_i$ and $G_j$ is a boundary circle of both facets, then $\mathcal{C}^1(P,\lambda)$ is a single circle. 
Thus, we can assume that a nonempty intersection of each two disks $G_i$ and $G_j$ consists of a disjoint union of edges.
If there are more then one edge, consider a simple closed curve $\gamma$ consisting of two simple paths
connecting the points inside two common edges -- one path inside $G_i$ and the other inside $G_j$.

Then $\mathcal{C}(P,\lambda)$ can be represented as a connected sum of complexes $\mathcal{C}(P,\lambda')$
and $\mathcal{C}(P,\lambda'')$ arising if we change the points of~the~affine coloring at all the facets of $P\setminus G_i$ 
inside one of~the~connected component  of $\partial P\setminus\gamma$ to $\lambda_j$. Denote $r'={\rm rk }\,\lambda'$ and $r''={\rm rk }\,\lambda''$.  
Both spaces $N(P,\lambda')$ and $N(P,\lambda'')$ are closed orientable manifolds by Corollary \ref{cor:manor}.

\begin{lemma} \label{lem:NPL2e}
There is a homeomorphism 
\begin{equation}\label{eq:2fe2}
N(P,\lambda)\simeq N(P,\lambda')^{\#2^{r-r'}}\#N(P,\lambda'')^{\#2^{r-r''}}
\#(S^2\times S^1)^{\#\left[2^{r-1}-2^{r-r'}-2^{r-r''}+1\right]}
\end{equation}
\end{lemma}   
The proof is similar to the proof of \cite[Proposition 3.6]{E22M}.
\begin{corollary}\label{mcor2}
Let each facet of $\mathcal{C}(P,\lambda)$ be a disk and the intersection of some two different facets be~a~disjoint 
set of~at~least two edges. Then in the  Knezer-Milnor prime decomposition of~the~orientable manifold $N(P,\lambda)$ 
there is a summand $S^1\times S^2$. In particular,
$N(P,\lambda)$ is not homeomorphic to a sphere and it is not a homology sphere for any coefficient group.
\end{corollary}

\begin{proof}
Indeed, in the Knezer-Milnor prime decomposition of $N(P,\lambda)$ there is a summand 
$\#(S^2\times S^1)^{\#\left[2^{r-1}-2^{r-r'}-2^{r-r''}+1\right]}$.
But $2\leqslant r',r''\leqslant r$, since on both sides of the curve $\gamma$ there is~a~vertex of~a~common edge, 
and therefore a facet  with $\lambda_k\notin \{\lambda_i, \lambda_j\}$,
where $G_i$ and $G_j$ are the facets under consideration.
Also $r'+r''=r+\dim{\rm aff}(\lambda')\cap{\rm aff}(\lambda'')\geqslant r+1$,
since $\lambda_i,\lambda_j\in {\rm aff}(\lambda')\cap{\rm aff}(\lambda'')$. Hence,
\begin{multline*}
2^{r-1}-2^{r-r'}-2^{r-r''}+1=2^{r-r'}(2^{r'-1}-1)-(2^{r-r''}-1)\geqslant \\
\geqslant 2^{r-r'}(2^{r'-1}-1)-(2^{r'-1}-1)=(2^{r-r'}-1)(2^{r'-1}-1)\geqslant 0
\end{multline*}
Moreover, if the left part is equal to zero, then $r'-1=r-r''$ and either $r'=1$ or $r=r'$ (then $r''=1$). A contradiction.
\end{proof}
Thus, we can assume that any facet of $\mathcal{C}(P,\lambda)$ is a disk bounded by a simple edge-cycle and 
any nonempty intersection of two facets is an edge. We know, that the boundary cycle of a~facet can not contain
only one vertex. If there are only two vertices $v$ and $w$ on the boundary of a facet $G_i$, then the vertex $v$ 
belongs to some other facets $G_j$ and $G_k$. Moreover, each facet $G_j$ and $G_k$ has a common edge with $G_i$,
and this edge contains $w$. Then $G_j\cap G_k$ is an edge connecting $v$ and $w$, and $\mathcal{C}^1(P,\lambda)$ is a theta-graph.

Now assume that each facet has at least $3$ vertices on its boundary. Then   $\mathcal{C}^1(P,\lambda)$ has no multiple edges,
for otherwise a $2$-gonal facet arises. Then $\mathcal{C}^1(P,\lambda)$ is a simple planar $3$-connected graph
with at least $4$ edges, and by the Steinitz theorem it 
corresponds to~a~boundary of~some simple $3$-polytope $Q$. This polytope has an induced affinely 
independent coloring $\lambda$ and $N(P,\lambda)=N(Q,\lambda)$, where $N(Q,\lambda)$ is a quotient space of a free action
of a subgroup $K\subset\mathbb Z_2^{m_Q}$ 
on $\mathbb R\mathcal{Z}_Q$. In particular, it is covered by $\mathbb R\mathcal{Z}_Q$.
Hence, if $N(P,\lambda)$ is a sphere, then $N(P,\lambda)=\mathbb R\mathcal{Z}_Q$. 

Assume that $Q\ne\Delta^3$. If $Q$ has a $3$-belt, that is a triple of facets
$G_i$, $G_j$ and $G_k$ with an~empty intersection such that any two of them are adjacent, 
then $Q$ is a~connected sum of two polytopes $Q_1$ and $Q_2$ along vertices (see details in \cite{E22M}). It is
proved in \cite[Corollary 3.8]{E22M} that there is a homeomorphism
$$
\mathbb R\mathcal{Z}_Q\simeq\mathbb{R} \mathcal{Z}_{Q_1}^{\#2^{m_Q-m_1}}\#\mathbb{R} \mathcal{Z}_{Q_2}^{\#2^{m_Q-m_2}}\#(S^2\times S^1)^{\#\left[(2^{m_Q-m_1}-1)\cdot(2^{m_Q-m_2}-1)\right]},
$$
where $m_Q$, $m_1$ and $m_2$ are the numbers of facets of $Q$, $Q_1$ and $Q_2$ respectively. 
Also $m_1,m_2\leqslant m_Q-1$.
Hence, if $Q$ contains a $3$-belt, then $\mathbb R\mathcal{Z}_Q$ contains a summand $S^2\times S^1$ in its
Knezer-Milnor decomposition. If $Q\ne \Delta^3$ has no $3$-belts, then $Q$ is a flag polytope and  $\mathbb R\mathcal{Z}_Q$
is aspherical (that is $\pi_i(\mathbb R\mathcal{Z}_Q)=0$ for $i\geqslant 2$, 
see \cite[Theorem 2.2.5]{DJS98} or \cite[Proposition 1.2.3]{D08}). 
Thus, if $N(P,\lambda)\simeq S^3$, then $Q=\Delta^3$ and the theorem is proved.
\end{proof}
\begin{corollary}\label{cor:3hem}
Let  $\lambda$ be an affine coloring of rank $r$ of a simple $3$-polytope $P$.
Then any hyperelliptic involution $\tau\in\mathbb Z_2^r$ is special, that is $\mathcal{C}(P,\lambda_\tau)\simeq \mathcal{C}(3,r)$.
\end{corollary}
\begin{definition}
Let us call by a~{\it theta-subgraph} and a~{\it $K_4$-subgraph} of $P$
the image of~an~embedding of~the~theta-graph or the~compete graph $K_4$
to~the~$1$-skeleton of~$P$ such that each vertex of~the~embedded graph is mapped to~a~vertex of~$P$
and each edge -- to a~simple edge-path.
\end{definition}
\begin{corollary}\label{cor:spheres}
Let $P$ be a simple $3$-polytope. The subgroups $H\ne H_0$ of $\mathbb Z_2^m$ such that $N(P,H)\simeq S^3$
are in one-to-one correspondence with simple edge-cycles, theta-subgraphs and $K_4$-subgraphs of $P$.
The  subgroup corresponding to a subgraph is defined by the linear equations  
$\sum_{F_i\subset G}x_i=0$ corresponding to its facets $G$. 
\end{corollary}

\begin{example} 
Any facet $F_i$ is bounded by a simple edge-cycle. This fits Example \ref{Ex:CG} for $G=F_i$.  
\end{example}
\begin{example}
It is known that for any two different vertices of $P$ there is a theta-subgraph with these vertices.
This is one of the equivalent definitions of the $3$-connectivity of the graph (see \cite[Section 11.3]{Gb03}).
Each edge $F_i\cap F_j$ of $P$ corresponds to a theta-subgraph according to Example \ref{Ex:CG}. 
Its two additional edges are formed by edges of the facets $F_i$ and $F_j$ complementary to $F_i\cap F_j$.
\end{example}
\begin{example}
Each vertex $F_i\cap F_j\cap F_k$ of $P$ corresponds to a~$K_4$-subgraph according to Example \ref{Ex:CG}. 
Its edges are $F_i\cap F_j$, $F_j\cap F_k$, $F_k\cap F_i$, and three additional edges formed by edges of the facets $F_i$, $F_j$ and $F_k$ complementary to the first three edges. 
\end{example}

\begin{example}
It is known that any simple $3$-polytope can be combinatorially obtained from $\Delta^3$ by a sequence of operations of  
cutting off a vertex or a set of successive edges of some facet by a single plane 
(V.~Eberhard (1891), M.~Br\"{u}ckner (1900), see \cite{Gb03}).  Each operation corresponds to a subdivision of a facet
of a graph into two facets by a~new edge. Each sequence of such operations connecting $\Delta^3$ and $P$ 
corresponds a $K_4$-subgraph~of ~$P$. 
\end{example}

There is the~following characterisation of  complexes $\mathcal{C}(3,k)$.
\begin{lemma}\label{lem:c1pc} 
Let $c$ be a coloring of a simple $3$-polytope $P$. Then 
\begin{enumerate}
\item $\mathcal{C}^1(P,c)$ is empty (equivalently, $\mathcal{C}(P,c)\simeq \mathcal{C}(3,1)$) if and only if the complex 
$\mathcal{C}(P,c)$ has exactly one facet;
\item $\mathcal{C}^1(P,c)$ is a~circle (equivalently, $\mathcal{C}(P,c)\simeq \mathcal{C}(3,2)$) 
if and only if the complex $\mathcal{C}(P,c)$ has exactly two facets;
\item $\mathcal{C}^1(P,c)$ is a~theta-graph (equivalently, $\mathcal{C}(P,c)\simeq \mathcal{C}(3,3)$) 
if and only if $\mathcal{C}(P,c)$ has exactly three facets and all of them are disks;
\item $\mathcal{C}^1(P,c)$ is a~$K_4$-graph (equivalently, $\mathcal{C}(P,c)\simeq \mathcal{C}(3,4)$) 
if and only if $\mathcal{C}(P,c)$ has exactly four facets, all of them are disks and any
two of them intersect.
\end{enumerate}
\end{lemma}
\begin{proof}
The ``only if'' part follows from the definition. If $\mathcal{C}(P,c)$ has exactly two facets, then both of them are 
disks and they intersect at the common boundary circle $\mathcal{C}^1(P,c)$. If $\mathcal{C}(P,c)$ has exactly three facets
and all of them are disks, consider two of them. Their intersection should be an edge, and the complement to their union
is the interior of the third disk. Thus, $\mathcal{C}^1(P,c)$ is a~theta-graph. 
If $\mathcal{C}(P,c)$ has exactly four facets, all of them are disks and any two of them intersect, consider
two disks. Their intersection can be either an edge, or a pair of edges, for otherwise there are more than $4$ facets.
If the intersection is a pair of edges, then the complementary two facets do not intersect, which is a contradiction.
Thus, the intersection of any two facets is an edge and any edge belongs to two facets. Then any facet is a triangle and 
$\mathcal{C}^1(P,c)$ is a~$K_4$-graph. 
\end{proof}

\section{Hyperelliptic manifolds $N(P,\lambda)$ over $3$-polytopes}\label{sec:3hm}

\begin{definition}
A {\it Hamiltonian cycle} of a polytope $P$ is a simple edge-cycle in the graph of $P$ containing all the vertices of $P$.
Let us call a~theta-subgraph or a~$K_4$-subgraph of $P$ {\it Hamiltonian} if it contains all the vertices of $P$.
More generally, for a coloring $\kappa$ of a simple polytope $P$ 
we call an~empty set $\varnothing$, a~simple cycle, 
a~theta-subgraph or a~$K_4$-subgraph of~$\mathcal{C}^1(P,\kappa)$ {\it Hamiltonian},
if~it~contains all the vertices of $\mathcal{C}(P,\kappa)$. Here by a~simple cycle we mean either a~circle 
(that is a $1$-face without vertices) or a~simple edge-cycle in $\mathcal{C}^1(P,\kappa)$.
In particular, if an~empty set or~a~circle is Hamiltonian, then $\mathcal{C}^1(P,\kappa)$  has no vertices, and 
it is~a~disjoint union of~circles.
\end{definition}

In the papers \cite{M90, VM99M, VM99S2} the authors constructed examples of hyperelliptic $3$-manifolds
in five of eight Thurston's geometries: $\mathbb R^3$, $\mathbb H^3$, $\mathbb S^3$, $\mathbb H^2\times\mathbb R$, 
and $\mathbb S^2\times\mathbb R$. In each case $M$ is obtained as $X/\pi$, where $X$ is a geometry and $\pi$ is a 
discrete group of isometries acting freely on~$X$. 
These examples were build using a right-angled $3$-polytope $P$ equipped with a~Hamiltonian cycle, 
a~Hamiltonian theta-subgraph, or a~Hamiltonian $K_4$-subgraph.  

In this section we will enumerate all hyperelliptic $3$-manifolds $N(P,\lambda)$ corresponding to
affine colorings of rank $r$ such that the hyperelliptic involution belongs to the group $\mathbb Z_2^r=H_0'$ 
canonically acting on $N(P,\lambda)$. In turns out that in the case of a right-angled polytope $P$ 
and an~affinely independent coloring $\lambda$ these are exactly manifolds built by  A.D.~Mednykh and A.Yu.~Vesnin.
In general case these manifolds correspond to proper Hamiltonian cycles, theta- and $K_4$-subgraphs 
in the complexes $\mathcal{C}(P,\kappa)$ defined by colorings $\kappa$ of simple $3$-polytopes.
\begin{construction}[An affine coloring induced by a Hamiltonian subgraph]\label{con:acH}
Let $\kappa$ be a~coloring  of a simple polytope $P$.
Given a~proper Hamiltonian cycle, theta-, or $K_4$-subgraph $\Gamma\subset\mathcal{C}^1(P,\kappa)$ 
one can define an affine coloring  $\Lambda_\Gamma$ {\it induced by} $\Gamma$ 
and a special hyperelliptic manifold $N(P,\kappa,\Gamma)=N(P,\lambda_\Gamma)$ as follows. 

Consider a~facet $D$ of~$\Gamma$ such that $D$ is a~union of more than one facets of~$\mathcal{C}(P,\kappa)$. 
Such a~facet exists~if $\Gamma\ne \mathcal{C}^1(P,\kappa)$. The facet $D$  is a~disk bounded by a simple cycle 
of~$\mathcal{C}^1(P,\kappa)$ and 
containing no vertices of~$\mathcal{C}^1(P,\kappa)$ in its interior. 
Consider the~adjacency graph $G_D$ of~the~facets of~$\mathcal{C}(P,\kappa)$  lying in $D$. Its vertices are facets and 
its edges correspond to~$1$-faces of $\mathcal{C}(P,\kappa)$ lying in two facets. 
The graph $G_D$ is connected. If its edge $E$ corresponds to an edge $e$ of 
$\mathcal{C}(P,\kappa)$, then $e$ has vertices on $\partial D$ and $E$ is a bridge. 
If $E$ corresponds to the circle of $\mathcal{C}(P,\kappa)$, then $E$ is~also a~bridge. Thus, 
$G_D$ is a~tree and its vertices can be colored in two colors such that adjacent vertices have different colors.  
Hence, the facets of $\Gamma$ define a~coloring $c$ of $P$ constant on them, 
and the tree corresponding to each facet defines the~$0/1$-coloring $\chi$ in Construction \ref{con:shm}. 
We obtain an~affine coloring $\lambda_\Gamma=\lambda(c,\chi)$
and a~special hyperelliptic manifold
$N(P,\kappa,\Gamma)=N(P,\lambda_{\Gamma})$ of rank $r$, where $r=2$ for a~Hamiltonian cycle, $r=3$
for a~Hamiltonian theta-subgraph, and $r=4$ for a~Hamiltonian $K_4$-subgraph.  Moreover, 
$\mathcal{C}^1(P,(\lambda_\Gamma)_{\tau})=\Gamma$.

Similarly, a~proper Hamiltonian empty set $\Gamma=\varnothing$ induces an affine coloring $\lambda_\Gamma$
and defines a~special hyperelliptic manifold $N(P,\kappa,\Gamma)=N(P,\lambda_\Gamma)$ of rank $r=1$. 
Namely, if the~complex $\mathcal{C}(P,\kappa)$ has no vertices, then $\mathcal{C}^1(P,\kappa)$ is a~disjoint union 
of~circles and each circle divides the~sphere $\partial P$ into 
two disks. Then the~adjacency graph of~facets of~$\mathcal{C}(P,\kappa)$ is a~tree and we can define 
the~$0/1$-coloring $\chi$ and the~constant coloring $c$ in~Construction \ref{con:shm}.
\end{construction}
\begin{remark}
It is not true that if the manifolds $N(P,\kappa,\Gamma)$ and $N(Q,\kappa',\Gamma')$ are weakly equivariantly
homeomorphic, then there is an equivalence $\mathcal{C}(P,\kappa)\to \mathcal{C}(Q,\kappa')$ such that
$\Gamma\to \Gamma'$. Two combinatorially different Hamiltonian subgraphs in $\mathcal{C}(P,\kappa)$
may induce the~same affine coloring $\lambda(c,\chi)$. In Fig. \ref{P8Ham} there is a polytope $P$
with three Hamiltonian cycles inducing the~same affine coloring of rank $2$ in four colors. Two of these
cycles can be moved to each other by~a~combinatorial equivalence of~$P$, but
the~third can not.
\end{remark}
\begin{definition} \label{def:graph}
A {\it matching} of a graph $G$ is a disjoint set of edges. A matching is {\it perfect}, if it contains
all the vertices of $G$. Perfect matching is also called a {\it $1$-factor}. A~{\it $1$-factorization} is~a~partition
of~the~set of~edges of~$G$ into disjoint $1$-factors. A {\it perfect pair} from a $1$-factorization is a pair of $1$-factors 
whose union is~a~Hamiltonian cycle. A {\it perfect $1$-factorization} of a graph is a $1$-factorization 
having the property that every pair of $1$-factors is a perfect pair.

Any Hamiltonian cycle $\Gamma$ in a~$3$-valent graph $G$
defines the following $1$-factorisation of $G$.  Each edge of $G$ not lying 
in $\Gamma$ connects two different vertices of $G$ and any vertex belongs 
to~a~unique edge of this type. We obtain a~$1$-factor. Then there are even
number of vertices and edges in $\Gamma$ and it is partitioned into two 
additional $1$-factors. 

We will call a Hamiltonian cycle in a $3$-valent graph {\it $k$-Hamiltonian}, if 
the corresponding $1$-factorization has exactly $k$ perfect pairs.
\end{definition}
\begin{theorem}\label{th:hypHam}
Let $\lambda$ be an affine  coloring of rank $r$ of a simple $3$-polytope $P$.
Then $N(P,\lambda)$ is~a~hyperelliptic manifold with a~hyperelliptic involution lying in~the~group 
$\mathbb Z_2^r=H_0'$ of~orientation preserving involutions canonically acting on~$N(P,\lambda)$ 
if and only if $1\leqslant r\leqslant 4$ and  $\lambda$ is~induced~by 
\begin{enumerate}
\item a~Hamiltonian empty set in $\mathcal{C}^1(P,\lambda)$ for $r=1$;
\item a~Hamiltonian cycle in $\mathcal{C}^1(P,\lambda)$ for $r=2$;
\item a~Hamiltonian theta-subgraph  in $\mathcal{C}^1(P,\lambda)$ for $r=3$;
\item a~Hamiltonian $K_4$-subgraph in $\mathcal{C}^1(P,\lambda)$ for $r=4$.
\end{enumerate}
Hyperelliptic involutions in $\mathbb Z_2^r$ bijectively correspond to the~Hamiltonian subgraphs of the above type
inducing the~coloring $\lambda$.  Moreover,
\begin{enumerate}
\item for $r=1$ there is a~unique hyperelliptic involution;
\item for $r=2$ there can be $1$, $2$ or $3$ such involutions. If the Hamiltonian cycle is a circle,
then there is a unique hyperelliptic involution. For the~Hamiltonian edge-cycle each involution
corresponds to~a~perfect pair of~$1$-factors. In particular,
there are $k\geqslant 2$ hyperelliptic involutions if and only if $\mathcal{C}^1(P,\lambda)$  
is~a~connected $3$-valent graph and $\lambda$ is~induced by~a~$k$-Hamiltonian cycle.
\item for $r=3$ and 
\begin{enumerate}
\item $I(\lambda)=4$ there can be $1$, $2$, $3$, $4$ or $6$ hyperelliptic involutions;
\item $I(\lambda)=5$ there can be $1$, $2$ or $3$ such involutions;
\item $I(\lambda)=6$  there is a~unique hyperelliptic involution;
\end{enumerate}
\item for $r=4$ and
\begin{enumerate}
\item $I(\lambda)=5$ there can be $1$, $2$ or $6$ hyperelliptic involutions;
\item $I(\lambda)=6$ there can be $1$ or $2$ such involutions;
\item $I(\lambda)\in\{7,8\}$  there is a~unique hyperelliptic involution;
\end{enumerate}
\end{enumerate}
\end{theorem}

We will obtain this result as a~corollary of the following lemma and a~more technical theorem.

\begin{lemma}
Let $\lambda$ be an affine  coloring of rank $r$ of a simple $3$-polytope $P$ and $\tau\in \mathbb Z_2^r$.
Then $\mathcal{C}(P,\lambda_\tau)\simeq \mathcal{C}(3,r)$ (that is, $\tau$ is~a~hyperelliptic involution) if and only if 
one of~the~following conditions hold:
\begin{enumerate}
\item $r=1$ and $\mathcal{C}^1(P,\lambda_\tau)$ is~a~Hamiltonian empty set in~$\mathcal{C}^1(P,\lambda)$;
\item $r=2$ and $\mathcal{C}^1(P,\lambda_\tau)$ is~a~Hamiltonian cycle in~$\mathcal{C}^1(P,\lambda)$;
\item $r=3$ and $\mathcal{C}^1(P,\lambda_\tau)$ is~a~Hamiltonian theta-subgraph in~$\mathcal{C}^1(P,\lambda)$;
\item $r=4$ and $\mathcal{C}^1(P,\lambda_\tau)$ is~a~Hamiltonian $K_4$-subgraph in~$\mathcal{C}^1(P,\lambda)$.
\end{enumerate}
In all these cases $\lambda$ is induced by the corresponding Hamiltonian subgraph. 
\end{lemma}
\begin{proof}
The lemma follows from Theorem \ref{th:NP3sp} and Corollary \ref{cor:C01inv}.
\end{proof}

\begin{theorem}\label{th:hyperell}
Let $\lambda$ be an affine  coloring of rank $r$ of a simple $3$-polytope $P$.
Then $N(P,\lambda)$ is a hyperelliptic manifold with a hyperelliptic involution lying in 
the~group $\mathbb Z_2^r=H_0'$ of orientation preserving involutions canonically acting on~$N(P,\lambda)$ 
if and only if $1\leqslant r\leqslant 4$ and one of~the~ following conditions holds:
\begin{enumerate}
\item $I(\lambda)=\{\boldsymbol{p}_1, \dots, \boldsymbol{p}_{r+1}\}$ is a boolean $r$-simplex, $1\leqslant r\leqslant 4$,
and at least for one vector $\tau=\boldsymbol{p}_i+\boldsymbol{p}_j$, $i\ne j$, the complex $\mathcal{C}(P,\lambda_\tau)$
is equivalent to $\mathcal{C}(3,r)$. Each hyperelliptic involution $\tau\in \mathbb Z_2^r$ has this form and there are at
most $\frac{r(r+1)}{2}$ such involutions. 
More precisely,
an~involution $\tau\in\mathbb Z_2^r$ is hyperelliptic if and only if $\tau=\boldsymbol{p}_i+\boldsymbol{p}_j$, $i\ne j$, 
and for
\begin{itemize}
\item $r=1$ it is equal to $1\in\mathbb Z_2$. This is always a~unique hyperelliptic involution.
\item $r=2$ the~set $G(\boldsymbol{p}_k)$, $\{i,j,k\}=\{1,2,3\}$, is a disk. There can be $0$,
$1$, $2$, or $3$ such involutions.
\item $r=3$ each set $G(\boldsymbol{p}_i, \boldsymbol{p}_j)$, $G(\boldsymbol{p}_k)$, $G(\boldsymbol{p}_l)$,
$\{i,j,k,l\}=\{1,2,3,4\}$, is a~disk. There can be $0$, $1$, $2$, $3$, $4$ or $6$ such involutions.
\item $r=4$ each set $G(\boldsymbol{p}_i, \boldsymbol{p}_j)$, $G(\boldsymbol{p}_k)$, $G(\boldsymbol{p}_l)$, 
$G(\boldsymbol{p}_s)$, $\{i,j,k,l,s\}=\{1,2,3,4,5\}$, is a~disk and any two of these disks intersect. There can be 
$0$, $1$, $2$ or $6$ hyperelliptic involutions. 
\end{itemize}
The classification of~complexes with more than one hyperelliptic involution and the corresponding manifolds $N(P,\lambda)$ 
is presented in~Fig.~\ref{rkckp1}.
\begin{figure}[h]
\begin{center}
\includegraphics[width=\textwidth]{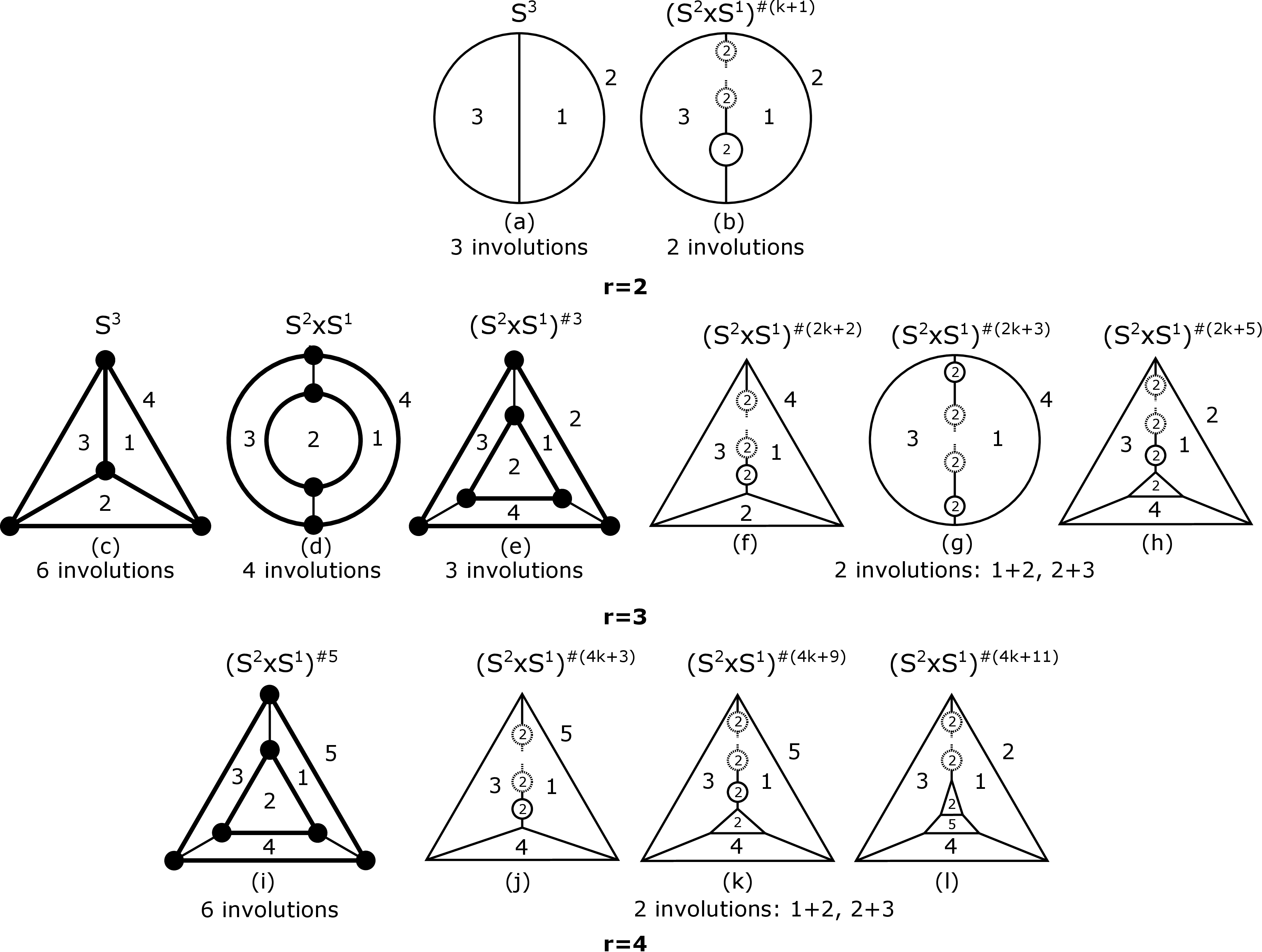}
\end{center}
\caption{All complexes with more than one hyperelliptic involution for~the~case when $I(\lambda)$
is a boolean simplex.  On the top we write the homeomorphism type of~$N(P,\lambda)$, where $k\geqslant 0$ 
is~the~number of dashed circles}\label{rkckp1}
\end{figure}
\item $I(\lambda)=\Pi^2*\Delta^{r-3}$, $2\leqslant r\leqslant 4$, where 
$\Pi^2=\{\boldsymbol{q}_1,\boldsymbol{q}_2,\boldsymbol{q}_3,\boldsymbol{q}_4\}\simeq\mathbb Z^2$
is~a~boolean $2$-plane and $\Delta^{r-3}=\{\boldsymbol{p}_1,\dots, \boldsymbol{p}_{r-2}\}$ 
is~a~boolean simplex, and at least for one vector $\tau=\boldsymbol{q}_i+\boldsymbol{q}_j$, $i\ne j$, the complex 
$\mathcal{C}(P,\lambda_\tau)$ is equivalent to $\mathcal{C}(3,r)$. 
Each hyperelliptic involution $\tau\in \mathbb Z_2^r$ has this form and there are at most three such involutions. 
More precisely,  an~involution $\tau\in\mathbb Z_2^r$ is hyperelliptic if and only if 
$\tau=\boldsymbol{q}_i+\boldsymbol{q}_j=\boldsymbol{q}_k+\boldsymbol{q}_l$
for some partition  $\{1,2,3,4\}=\{i,j\}\sqcup\{k,l\}$, and one of the following conditions holds
\begin{itemize}
\item $r=2$ and $G(\boldsymbol{q}_i,\boldsymbol{q}_j)$ is a~disk 
(then $G(\boldsymbol{q}_k,\boldsymbol{q}_l)$ is also a~disk bounded by the same 
Hamiltonian cycle $\Gamma$ from $\mathcal{C}^1(P,\lambda)$). 
There can be $0$, $1$, $2$, or $3$ such involutions.
Moreover, there are $k\geqslant 2$ hyperelliptic involutions if and 
only if $\mathcal{C}^1(P,\lambda)$ is a~connected $3$-valent graph and $\Gamma$ is~a~$k$-Hamiltonian cycle in~it.
For $k=3$ this implies that  $\mathcal{C}(P,\lambda)$ is equivalent to~the~boundary complex of~a~simple $3$-polytope $Q$.
\item $r=3$ and each set $G(\boldsymbol{q}_i, \boldsymbol{q}_j)$, 
$G(\boldsymbol{q}_k,\boldsymbol{q}_l)$ and $G(\boldsymbol{p}_1)$
is a~disk. There can be $0$, $1$, $2$ or $3$ such involutions.
Moreover, if there are $2$ hyperelliptic involutions, then $G(\boldsymbol{p}_1)$ is~a~quadrangle, a~triangle,
or~a~bigon, and the complex $\mathcal{C}(P,\lambda)$ can be reduced to~a~complex $\mathcal{C}(P,\lambda')$ 
for~an~affine coloring $\lambda'$ of rank $2$ either
\begin{itemize}
\item with $2$ or $3$ hyperelliptic involutions by reductions (a)-(d), or (f) in Fig.~\ref{2invr34}, or
\item with $2$ hyperelliptic involutions by reduction (e). 
\end{itemize}
If there are $3$ hyperelliptic involutions, then $G(\boldsymbol{p}_1)$ is a~triangle and $\mathcal{C}(P,\lambda)$ can be reduced 
to~$\mathcal{C}(P,\lambda')$ of rank $2$ with $3$ hyperelliptic involutions by reduction (e). 
\item $r=4$ and each set $G(\boldsymbol{q}_i, \boldsymbol{q}_j)$, $G(\boldsymbol{q}_k,\boldsymbol{q}_l)$, 
$G(\boldsymbol{p}_1)$ and $G(\boldsymbol{p}_2)$
is a~disk and any two of these disks intersect. There can be 
$0$, $1$ or $2$ hyperelliptic involutions. Moreover, if there are $2$ hyperelliptic involutions, then $G(\boldsymbol{p}_1, \boldsymbol{p}_2)$ is~a~quadrangle, a~triangle,
or~a~$2$-gon, and the complex $\mathcal{C}(P,\lambda)$ can be reduced to a complex $\mathcal{C}(P,\lambda')$ 
for~an~affine coloring $\lambda'$ of rank $2$ with $2$ or $3$ hyperelliptic involutions by reductions 
(a)-(f) in Fig.~\ref{2invr34}. 
\end{itemize}
\begin{figure}[h]
\begin{center}
\includegraphics[width=0.6\textwidth]{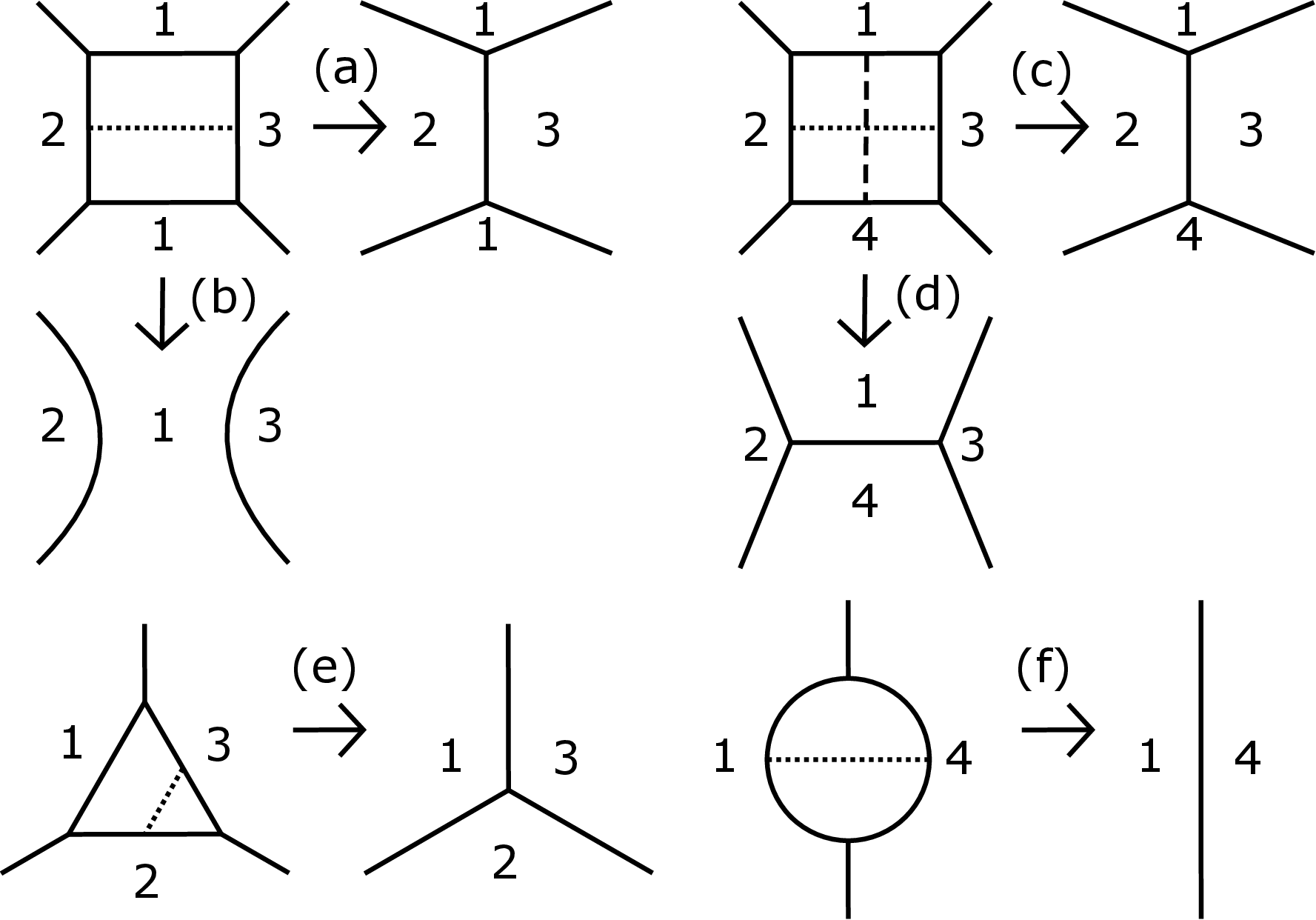}
\end{center}
\caption{Reductions for complexes with $2$ and $3$ hyperelliptic involutions for $r=3$ and $|I(\lambda)|=5$.
By dotted and dashed lines we mark possible edges for~the~case $r=4$ and $|I(\lambda)|=6$}\label{2invr34}
\end{figure}
\item $I(\lambda)=\Pi^k*\Delta^{r-k-1}$, $r\geqslant k\geqslant 3$, and for the main direction $\tau=\boldsymbol{l}$
of $\Pi^k$ the complex $\mathcal{C}(P,\lambda_\tau)$ is equivalent to $\mathcal{C}(3,r)$.  
In this case the main direction is a unique hyperelliptic  involution in $\mathbb Z_2^r$.
\end{enumerate}
\end{theorem}
\begin{proof}
The proof essentially follows from Propositions \ref{prop:conhi} and \ref{prop:clhi}, Lemma \ref{lem:c1pc}, 
Theorem \ref{th:NP3sp}, and Corollary \ref{cor:3hem}.

We need to prove only statements concerning the enumeration
of special hyperelliptic involutions in $\mathbb Z_2^r$ and the classification 
of~complexes with more than one such involutions.

Let $I(\lambda)=\{\boldsymbol{p}_1, \dots, \boldsymbol{p}_{r+1}\}$ be a boolean $r$-simplex, $1\leqslant r\leqslant 4$.
The case $r=1$ is trivial.

Let $r=2$. If all the facets $G(\boldsymbol{p}_1)$, $G(\boldsymbol{p}_2)$ and $G(\boldsymbol{p}_3)$ are disks,
then $\mathcal{C}^1(P,\lambda)$ is a theta-graph (Fig.~\ref{rkckp1}(a)) by Lemma \ref{lem:c1pc}.
If two facets $G(\boldsymbol{p}_i)$ and $G(\boldsymbol{p}_j)$ are disks and the third facet $G(\boldsymbol{p}_k)$
is not, then we have the complex draw in Fig.~\ref{rkckp1}(b).

Let $r=3$. Assume that $G(\boldsymbol{p}_1,\boldsymbol{p}_2)$, $G(\boldsymbol{p}_3)$, $G(\boldsymbol{p}_4)$
are disks, that is the involution $\boldsymbol{p}_1+\boldsymbol{p}_2$ is hyperelliptic. 
The involution $\boldsymbol{p}_3+\boldsymbol{p}_4$ is hyperelliptic if and only if both
$G(\boldsymbol{p}_1)$ and $G(\boldsymbol{p}_2)$ are also disks. We obtain 
two complexed drawn in Fig.~\ref{rkckp1}(c) and (d). They have $6$ and $4$ hyperelliptic involutions respectively. 
Now assume that one of these sets is not a disk, say $G(\boldsymbol{p}_2)$. 
Then there are at most $3$ hyperelliptic involutions and all of them have the form $\boldsymbol{p}_2+\boldsymbol{p}_i$.
If either $G(\boldsymbol{p}_1)$ is not a disk, or it is a disk and does not intersect the~disk
$G(\boldsymbol{p}_3,\boldsymbol{p}_4)$,  then $\boldsymbol{p}_1+\boldsymbol{p}_2$ is 
a~unique hyperelliptic involution. Thus, we can assume that  $G(\boldsymbol{p}_1)$ is a disk 
and it intersects the~disk $G(\boldsymbol{p}_3,\boldsymbol{p}_4)$. Then their intersection
consists of $k+2\geqslant 2$ disjoint segments,  $G(\boldsymbol{p}_2)$ 
a~disjoint union of $k+2$ disks, and the~combinatorics of the complex $\mathcal{C}(P,\lambda)$
depends on~the~position of the edge $G(\boldsymbol{p}_3)\cap G(\boldsymbol{p}_4)$
in the disk $G(\boldsymbol{p}_3,\boldsymbol{p}_4)$ in relation to these $k+2$ disk, see Fig. \ref{G34}(a). 
If $G(\boldsymbol{p}_i,\boldsymbol{p}_2)$ is a~disk for $i=3$ or $i=4$, then 
$G(\boldsymbol{p}_i)$ intersects each connected component of $G(\boldsymbol{p}_2)$. In particular, 
if this holds for both $i=3$ and $i=4$, we obtain the complex in Fig. \ref{G34}(b) and in Fig. \ref{rkckp1}(e).
If this holds only for one index, say $i=3$, then we obtain complexes in Fig. \ref{G34}(c)-(f).
The~complexes (c), (d), and (e) correspond to~the~complexes in~Fig. \ref{rkckp1}(h), (f), and (g), and 
for the~complex (f) the set $G(\boldsymbol{p}_2,\boldsymbol{p}_3)$ is~a~cylinder.

\begin{figure}[h]
\begin{center}
\includegraphics[width=\textwidth]{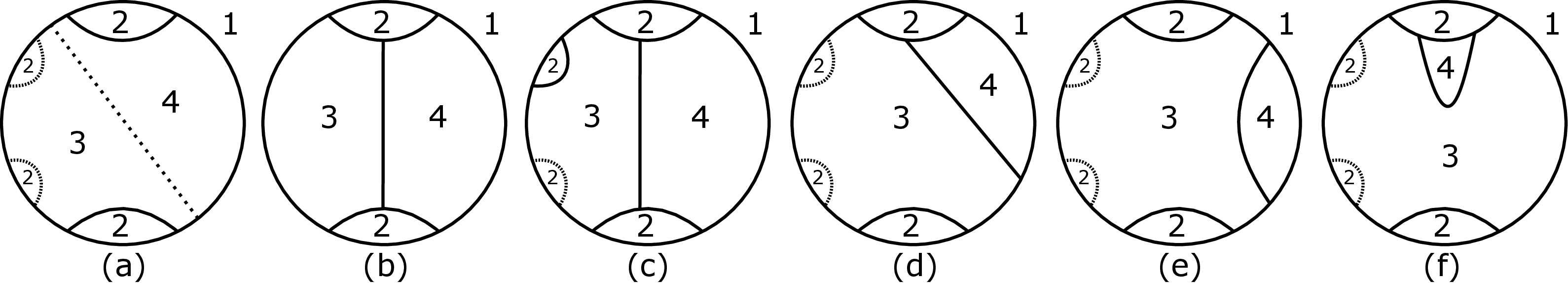}
\end{center}
\caption{A complex $\mathcal{C}(P,\lambda)$ when $G(\boldsymbol{p}_1)$ and $G(\boldsymbol{p}_3,\boldsymbol{p}_4)$ 
are intersecting disks, and $G(\boldsymbol{p}_2)$ is not a~disk}\label{G34}
\end{figure}

Let $r=4$. Assume that $G(\boldsymbol{p}_1,\boldsymbol{p}_2)$, $G(\boldsymbol{p}_3)$, $G(\boldsymbol{p}_4)$,
$G(\boldsymbol{p}_5)$ are pairwise intersecting disks, that is the involution $\boldsymbol{p}_1+\boldsymbol{p}_2$ is hyperelliptic. 
If some of the involutions $\boldsymbol{p}_3+\boldsymbol{p}_4$, $\boldsymbol{p}_3+\boldsymbol{p}_5$,
$\boldsymbol{p}_4+\boldsymbol{p}_5$ is hyperelliptic, then both $G(\boldsymbol{p}_1)$ and $G(\boldsymbol{p}_2)$
are disks, and they are glued to the disk $G(\boldsymbol{p}_1, \boldsymbol{p}_2)$ along the~common 
edge $G(\boldsymbol{p}_1)\cap G(\boldsymbol{p}_2)$. If the ends of this edge belong to the~same disk
$G(\boldsymbol{p}_i)$, $i=3$, $4$, $5$, then we obtain the~complex in~Fig. \ref{rkckp1}(j) with $k=0$ dashed
circles. It has $2$ hyperelliptic involutions. If the ends of $G(\boldsymbol{p}_1)\cap G(\boldsymbol{p}_2)$ belong to different
disks $G(\boldsymbol{p}_i)$ and $G(\boldsymbol{p}_j)$, then we obtain the~complex in Fig. \ref{rkckp1}(i) with $6$
hyperelliptic involutions. Now assume that one of the~sets $G(\boldsymbol{p}_1)$ and $G(\boldsymbol{p}_2)$ 
is not a disk, say $G(\boldsymbol{p}_2)$. Then there are at most $4$ hyperelliptic involutions and all of them have the form 
$\boldsymbol{p}_2+\boldsymbol{p}_i$. If either $G(\boldsymbol{p}_1)$ is not a disk, or it is a disk and does not intersect the~disk
$G(\boldsymbol{p}_3,\boldsymbol{p}_4,\boldsymbol{p}_5)$,  then $\boldsymbol{p}_1+\boldsymbol{p}_2$ is 
a~unique hyperelliptic involution. Thus, we can assume that  $G(\boldsymbol{p}_1)$ is a disk 
and it intersects the~disk $G(\boldsymbol{p}_3,\boldsymbol{p}_4,\boldsymbol{p}_5)$. Then their intersection
consists of $k+2\geqslant 2$ disjoint segments,  $G(\boldsymbol{p}_2)$ 
a~disjoint union of $k+2$ disks, and the~combinatorics of the complex $\mathcal{C}(P,\lambda)$
depends on~the~positions of~the~ends of~the~edges $G(\boldsymbol{p}_3)\cap G(\boldsymbol{p}_4)$, 
$G(\boldsymbol{p}_4)\cap G(\boldsymbol{p}_5)$, and $G(\boldsymbol{p}_5)\cap G(\boldsymbol{p}_3)$
on the circle
$\partial G(\boldsymbol{p}_3,\boldsymbol{p}_4, \boldsymbol{p}_5)$ in relation to these $k+2$ disks, see Fig. \ref{G345}(a). 
If $G(\boldsymbol{p}_i,\boldsymbol{p}_2)$ is a~disk for some $i=3,4,5$, then 
$G(\boldsymbol{p}_i)$ intersects each connected component of $G(\boldsymbol{p}_2)$. In particular, this
can not hold for all $i\in\{3,4,5\}$. If this holds for two values of $i$, say $i=3$ and $4$, then we obtain the complex 
in~Fig.~\ref{G345}(b) without dashed arcs. 
Now assume that only one set $G(\boldsymbol{p}_i,\boldsymbol{p}_2)$ is a~disk, say for $i=3$. 
We obtain complexes in~Fig.~\ref{G345}(b)-(g). In the~complexes (b), (d), and (g) the~set $G(\boldsymbol{p}_5)$ 
does not intersect $G(\boldsymbol{p}_1)$, hence they have a unique hyperelliptic involution $\boldsymbol{p}_1+\boldsymbol{p}_2$. 
The complexes (c), (e), and (f) have two hyperelliptic involutions and correspond to complexes (l), (k), and (j) 
in Fig.~\ref{rkckp1} (the latter with $k\geqslant 1$ dashed circles).
\begin{figure}[h]
\begin{center}
\includegraphics[width=0.6\textwidth]{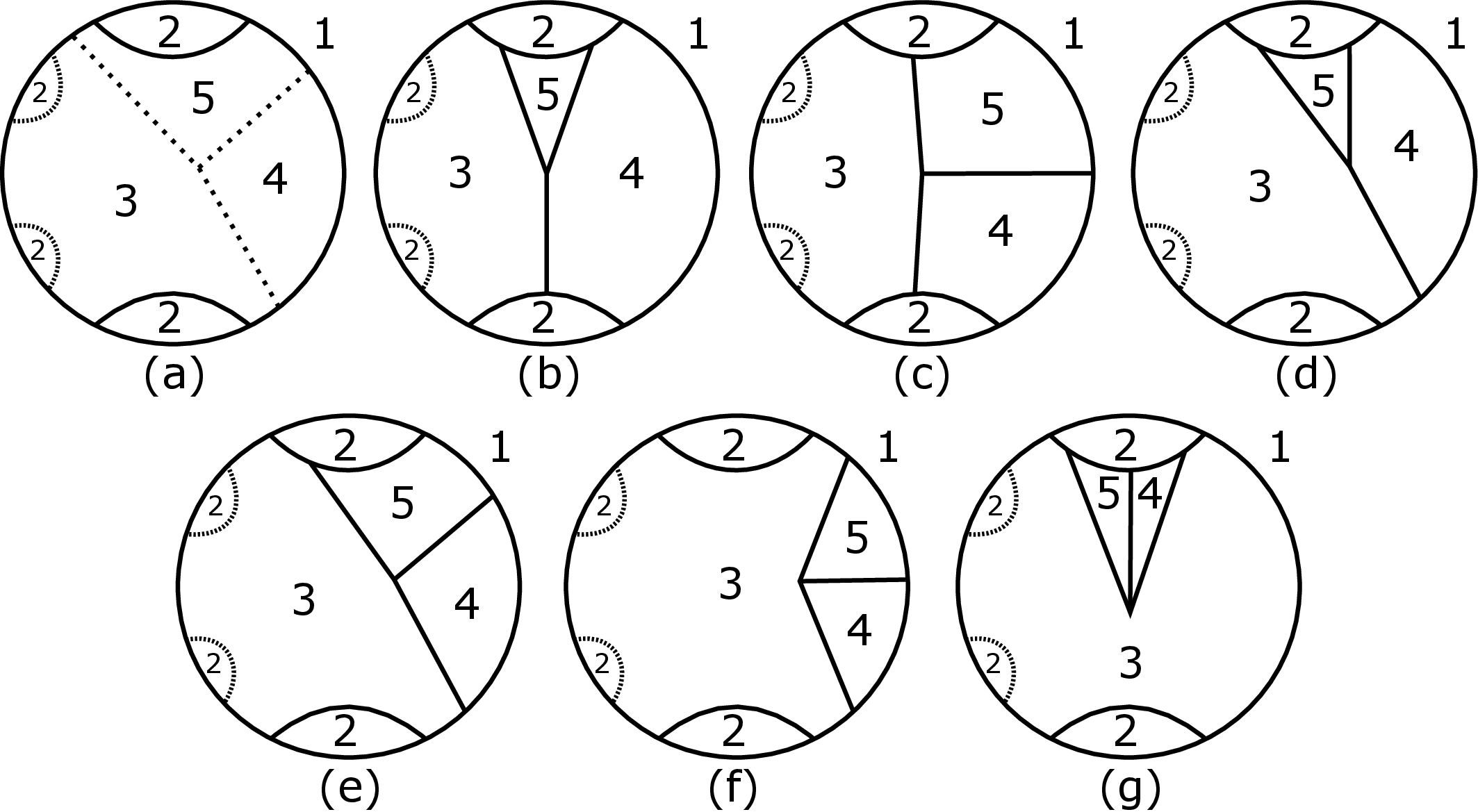}
\end{center}
\caption{A complex $\mathcal{C}(P,\lambda)$ when $G(\boldsymbol{p}_1)$ and 
$G(\boldsymbol{p}_3,\boldsymbol{p}_4, \boldsymbol{p}_4)$ are intersecting disks, and $G(\boldsymbol{p}_2)$ is not a~disk}\label{G345}
\end{figure}

The homeomorphism type of manifolds $N(P,\lambda)$ corresponding to complexes
in Fig. \ref{rkckp1} follow directly from Lemma \ref{lem:NPL2e}.

If  $I(\lambda)=\{\boldsymbol{q}_1,\boldsymbol{q}_2,\boldsymbol{q}_3, \boldsymbol{q}_4\}\simeq \Pi^2$, then special
hyperelliptic involutions are exactly sums $\boldsymbol{q}_i+\boldsymbol{q}_j=\boldsymbol{q}_k+\boldsymbol{q}_l$ 
corresponding to partitions $\{1,2,3,4\}=\{i,j\}\sqcup\{k,l\}$ such that 
$G(\boldsymbol{q}_i,\boldsymbol{q}_j)$ is a~disk (as~well as~its~complement $G(\boldsymbol{q}_k,\boldsymbol{q}_l)$). 
The boundary of this disk is a Hamiltonian cycle in~$\mathcal{C}^1(P,\lambda)$.
There can be one, two or three such partitions 
corresponding to a~Hamiltonian cycle as~it~is shown in~Fig.~\ref{Ham5pr}, \ref{Hamcube}, and \ref{Hamdod}. 
\begin{figure}[h]
\begin{center}
\includegraphics[width=0.6\textwidth]{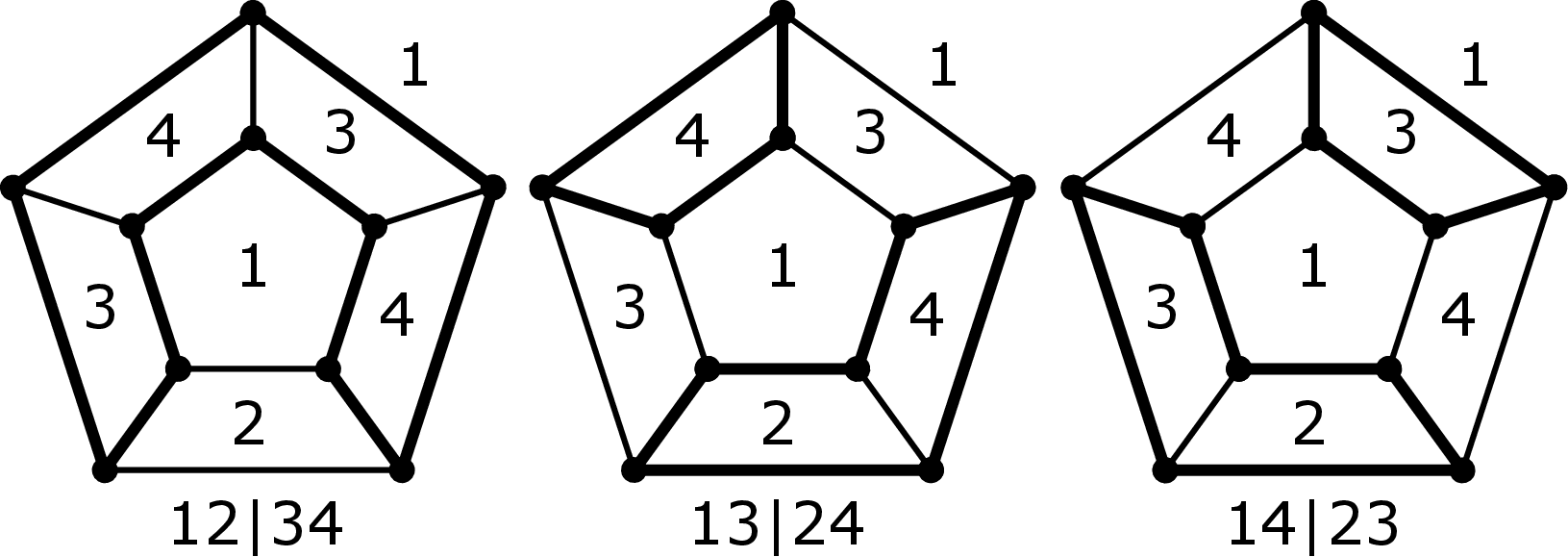}
\end{center}
\caption{The Hamiltonian cycle on the $5$-prism}\label{Ham5pr}
\end{figure}
\begin{lemma}
If $I(\lambda)\simeq \Pi^2$ and there are at least two hyperelliptic involutions in $\mathbb Z_2^2$, 
then $\mathcal{C}^1(P,\lambda)$ has no circles.
\end{lemma}
\begin{proof}
Indeed, each circle $\gamma$ is a boundary component of two facets $G_1$ and $G_2$ of different colors 
$\boldsymbol{q}_i$ and $\boldsymbol{q}_j$. 
At least for one partition $\{i,k\}\sqcup\{j,l\}=\{1,2,3,4\}$ 
the~sets $G(\boldsymbol{q}_i,\boldsymbol{q}_k)$ and $G(\boldsymbol{q}_j,\boldsymbol{q}_l)$
are disks.  Hence, their common boundary is $\gamma$. 
Since each disk consists of facets of two colors, 
each of the facets $G_1$ and $G_2$ has more than one boundary components and each component
different from $\gamma$ leads to the facet of the color $\boldsymbol{q}_k$ for $G_1$ and $\boldsymbol{q}_l$ for $G_2$.
But both sets $G(\boldsymbol{q}_k,\boldsymbol{q}_j)$ and $G(\boldsymbol{q}_k,\boldsymbol{q}_l)$
are disconnected, and they can not be disks. A contradiction.
\end{proof}
The group $\mathbb Z_2^2$ contains three hyperelliptic involutions if and only if for each of the~three partitions
$\{i,j\}\sqcup\{k,l\}=\{1,2,3,4\}$ the sets $G(\boldsymbol{q}_i,\boldsymbol{q}_j)$ are disks. 
This holds if and only if the~boundary of any of~these disks is~a~$3$-Hamiltonian cycle in $\mathcal{C}^1(P,\lambda)$.
By Lemma~\ref{lem:3HamQ} $\mathcal{C}(P,\lambda)\simeq\partial Q$ for a~simple $3$-polytope $Q$,
since  $\mathcal{C}^1(P,\lambda)$ is not a theta-graph for $|I(\lambda)|=4$. 

Assume that $I(\lambda)=\{\boldsymbol{q}_1,\boldsymbol{q}_2,\boldsymbol{q}_3, \boldsymbol{q}_4\}*
\{\boldsymbol{p}_1,\dots, \boldsymbol{p}_{r-2}\}$, $r\geqslant 3$. An~involution $\tau$
is hyperelliptic if and only if $\tau=\boldsymbol{q}_i+\boldsymbol{q}_j=\boldsymbol{q}_k+\boldsymbol{q}_l$
and $\mathcal{C}(P,\lambda_\tau)\simeq \mathcal{C}(3,r)$. Assume that there are at least two such involutions.
For each of them the~sets $G(\boldsymbol{q}_i,\boldsymbol{q}_j)$, $G(\boldsymbol{q}_k,\boldsymbol{q}_l)$, and 
$G(\boldsymbol{p}_1,\dots, \boldsymbol{p}_{r-2})$ are disks, and these disks are facets of~a~theta-graph $\Theta_{i,j}$.
Without loss of generality assume that hyperelliptic involutions correspond to partitions 
$\{1,2\}\sqcup\{3,4\}$ and $\{1,3\}\sqcup\{2,4\}$. 
Consider the vertices of $\mathcal{C}(P,\lambda)$ lying on the boundary of the disk  
$G(\boldsymbol{p}_1,\dots, \boldsymbol{p}_{r-2})$ and corresponding to~edges lying outside this disk. Each edge 
is~an~intersection of two facets of $\mathcal{C}(P,\lambda)$ of~different colors. Let 
us assign this pair of colors to the corresponding vertex. Then the two vertices corresponding to~the~vertices 
of~$\Theta_{1,2}$ have colors $(a,b)$, $a\in\{1,2\}$, $b\in \{3,4\}$, and all the other vertices -- $(1,2)$ and $(3,4)$.
Each vertex of types $(1,2)$ and $(3,4)$ necessarily
corresponds to~a~vertex of~$\Theta_{1,3}$. Therefore, there are at most two such vertices,
and $G(\boldsymbol{p}_1,\dots, \boldsymbol{p}_{r-2})$ is a~quadrangle, a~triangle, or a bigon 
(for $r=4$ we do not take into account the vertices of $G(\boldsymbol{p}_1)\cap G(\boldsymbol{p}_2)$). 
If there are two vertices, then either they both correspond to one type, say $(1,2)$,
and we obtain the configuration in~Fig. \ref{2invr34}(a),(b), or they correspond to two types and 
up to a~renumbering  of colors we obtain the~configuration in~Fig.~\ref{2invr34}(c),(d).
In the first case we can change the colors at~all~the~facets of~$P$ corresponding 
to~$G(\boldsymbol{p}_1,\dots, \boldsymbol{p}_{r-2})$ to $2$ (or to $3$)
to obtain the reduction (a), or to $1$ to obtain the reduction (b). In the second case we can change
the colors to $2$ (or $3$) to obtain the reduction (c), or to $1$ (or $4$) to obtain (d). In all these cases
each of the two Hamiltonian theta-graphs or $K_4$-graphs is reduced to~a~Hamiltonian cycle.
Moreover, in both cases the third partition $\{1,4\}\sqcup\{2,3\}$ does not give a~Hamiltonian theta-graph 
(or a~$K_4$-graph), while for the reduced complex $\mathcal{C}(P,\lambda')$ it can give.
For $r=4$ the edge $G(\boldsymbol{p}_1)\cap G(\boldsymbol{p}_2)$ in the~first case  should have 
one vertex lying on the boundary of a facet of color $2$ and the other -- of color $3$, and in the second
case these vertices can lie either on the boundaries of facets of colors $1$ and $4$, or $2$ and $3$.

If there is only one vertex of types $(1,2)$ or $(3,4)$, then up to a renumbering of colors we obtain  
the~configuration  in Fig.~\ref{2invr34}(e). Changing the colors to $1$ (or $2$, or $3$) we obtain the reduction 
(e).   For $r=3$ the reduced complex has the same number of Hamiltonian subgraphs corresponding to
the partitions of colors. For $r=4$ the~vertices of~the~edge $G(\boldsymbol{p}_1)\cap G(\boldsymbol{p}_2)$
should lie on the boundaries of facets of colors $2$ and $3$, and the third partition can not give 
the~Hamiltonian $K_4$-graph, while for the reduced complex  it can give.

If there are no vertices of types $(1,2)$ and $(3,4)$, then up to~a~renumbering of colors we obtain  
the~configuration  in~Fig. \ref{2invr34}~(f). In this case both for the~complex and for the~reduced complex 
the~third partition does not give the~Hamiltonian theta-graph ($K_4$-graph).  
For $r=4$ the~vertices of~the~edge $G(\boldsymbol{p}_1)\cap G(\boldsymbol{p}_2)$
should lie on the boundaries of facets of colors $1$ and $4$.

If  $I(\lambda)=\Pi^k*\Delta^{r-k-1}$, $r\geqslant k\geqslant 3$, then by Proposition \ref{prop:clhi}
the~main direction is a~unique hyperelliptic involution. This finishes the proof.
\end{proof}


\begin{example}
Example \ref{ex:cnindep} implies that for a simple $3$-polytope $P$ hyperelliptic manifolds $N(P,\lambda)$ 
of rank $r$ with affinely independent colorings $\lambda$ and a~hyperelliptic involution  $\tau\in\mathbb Z_2^r$ 
correspond to Hamiltonian cycles, Hamiltonian theta-subgraphs and Hamiltonian $K_4$-subgraphs of $P$ for $r=2$, 
$3$, and $4$ respectively. Indeed, in this example we showed how 
a~manifold~$N(P,\lambda)$ gives a~Hamiltonian subgraph, and  
Construction \ref{con:acH} gives the~manifold from a~subgraph.

For compact right-angled $3$-polytopes in one of the geometries 
$\mathbb R^3$, $\mathbb H^3$, $\mathbb S^3$, $\mathbb H^2\times\mathbb R$, and $\mathbb S^2\times\mathbb R$,
these are exactly examples built in \cite{M90} and \cite{VM99S2}. 
The same manifolds arise for the pairs $(P,\lambda)$ with $\mathcal{C}(P,\lambda)$  
equivalent to boundaries of right-angled polytopes. On the other hand, if $\mathcal{C}(P,\lambda)$
is not equivalent to a boundary of a  right-angled polytope, then our manifolds are not reduced to the examples 
from \cite{M90} and \cite{VM99S2}. 
\end{example}

\section{Rational homology spheres $N(P,\lambda)$ over $3$-polytopes}\label{Sec:RHS}
In this section we will classify all rational homology $3$-spheres $N(P,\Lambda)$ over simple $3$-polytopes~$P$.
\begin{definition}
We call a topological space $X$ a {\it rational homology $n$-sphere} ($n$-$RHS$), if $X$ is a closed topological $n$-manifold
and $H_k(X,\mathbb Q)=H_k(S^n,\mathbb Q)$ for all $k$.
\end{definition}  
We will use the following result,
which was first proved for small covers and $\mathbb Q$ coefficients in \cite{ST12, T12}. Let us identify the subsets
$\omega\subset[m]=\{1,\dots,m\}$ with vectors $\boldsymbol{x}\in\mathbb Z_2^m$ by~the~rule $\omega=\{i\colon x_i=1\}$.
For a vector-coloring $\Lambda$ of rank $(r+1)$ denote by ${\rm row}\,\Lambda$ the subspace in $\mathbb Z_2^m$ 
generated by the row vectors of the matrix $\Lambda$. Equivalently, 
$$
{\rm row}\,\Lambda=\{(x_1,\dots,x_m)\in\mathbb Z_2^m\colon\exists\boldsymbol{c}\in(\mathbb Z_2^{r+1})^*\colon x_i=
\boldsymbol{c}\Lambda_i, i=1,\dots,m\}.
$$
Remind that $P_{\omega}=\bigcup_{i\in\omega}F_i$.
\begin{theorem}\cite[Theorem 4.5]{CP17}\label{HNPLth}
Let $\Lambda$ be a vector-coloring of rank $(r+1)$ of a simple $n$-polytope $P$ and $R$ be
a commutative ring in which $2$ is a unit. Then there is an $R$-linear isomorphism
$$
H^k(N(P,\Lambda), R)\simeq\bigoplus_{\omega\in{\rm row}\,\Lambda}\widetilde{H}^{k-1}(P_{\omega}, R)
$$
\end{theorem}
\begin{remark}
Originally, the  theorem is formulated for a simplicial complexes $K$ and its full subcomplexes $K_{\omega}$, but
for a simple polytope $P$ and a simplicial complex $K=\partial P^*$ there is~a~homotopy equivalence 
$K_{\omega}\simeq P_{\omega}$, see \cite[The proof of Proposition 3.2.11]{BP15}.
\end{remark}
\begin{remark}
Multiplicative structure in Theorem \ref{HNPLth} was described in \cite{CP20}. 
\end{remark}
The universal coefficients formula and the Poincare duality imply
\begin{lemma}
A $3$-manifold $M$ is a rational homology $3$-sphere if and only if it is  closed, orientable, and $H^1(M,\mathbb Q)=0$.
\end{lemma}
Let is remind that a closed orientable manifold $N(P,\Lambda)$ is defined by a~an affine coloring $\lambda$
of rank $r$, where for some change of coordinates in $\mathbb Z_2^{r+1}$ we have $\Lambda_i=(1,\lambda_i)$. 
\begin{proposition}\label{pr:rhs}
Let $\lambda$ be an affine coloring of rank $r$ of a simple $3$-polytope $P$.
The space $N(P,\lambda)$ is a rational homology $3$-sphere if and only if 
one of~the~following equivalent conditions holds:
\begin{enumerate}
\item $\bigcup\limits_{i\colon \lambda_i\in \pi }F_i$ is~a~disk for any affine hyperplane $\pi\subset\mathbb Z_2^r$;
\item  $\bigcup\limits_{i\colon \lambda_i\in \pi }F_i$ is~a~disk for any affine hyperplane $\pi\subset\mathbb Z_2^r$
passing through some  pint $\boldsymbol{p}\in\mathbb Z_2^r$.
\end{enumerate}
\end{proposition}
\begin{remark}
It will be shown in \cite{E24b} that this proposition  also holds for $n=4$.
\end{remark}
\begin{remark}
Proposition \ref{pr:rhs} is a refinement of a description of rational homology $3$-spheres over right-angled polytopes
in $\mathbb S^3$, $\mathbb R^3$ and $\mathbb H^3$ used in \cite[Corollary 7.9]{FKR23} 
to build an infinite family of arithmetic hyperbolic rational homology $3$-spheres 
that are totally geodesic boundaries of compact hyperbolic $4$-manifolds, and in \cite[Proposition 3.1]{FKS21} 
to detect the  Hantzsche-Wendt manifold among  manifolds defined by linearly independent colorings of the $3$-cube.
(It is equivalent to the connectivity of the full subcomplex $K_{\omega}$ of the boundary $K=\partial P^*$ 
of the dual polytope $P^*$ for each subset $\omega=\{i\colon \lambda_i\in \pi\}$ corresponding to an affine hyperplane $\pi$.)
\end{remark}  
\begin{proof}
Linear functions $\boldsymbol{c}\in (\mathbb Z_2^{r+1})^*$ correspond to affine functions on $\mathbb Z_2^r$. 
Then $H^1(N(P,\lambda), \mathbb Q)=0$ if and only if
for any affine function  $\boldsymbol{c}$ we have $\widetilde{H}^0(P_{\omega}, \mathbb Q)=0$
for $\omega$ corresponding to~the~vector $(\boldsymbol{c}(\lambda_1),\dots, \boldsymbol{c}(\lambda_m))$.
There are two constant affine functions. For~$\boldsymbol{0}$ we have $P_{\omega}=\varnothing$, and
for~$\boldsymbol{1}$ we have $P_{\omega}=\partial P\simeq S^2$. All the other affine functions $\boldsymbol{c}$
correspond to affine hyperplanes $\boldsymbol{c}(\boldsymbol{x})=0$. 
For each affine hyperplane the set $P_{\omega}$ should be connected.
This set is~a~disjoint union of spheres with holes, and the complementary hyperplane corresponds 
to the complementary set. Both sets are connected if and only if they are disks, which is equivalent 
to the fact that one of them is a disk. Since for any affine hyperplane in $\mathbb Z_2^r$ 
the~point $\boldsymbol{p}$ either lies in~this~plane or in~the~complementary hyperplane, items (1) and (2) are equivalent.
\end{proof}
\begin{proposition}\label{prop:rhs}
If a $3$-manifold $N(P,\lambda)$ is a $3$-$RHS$, then
\begin{itemize}
\item either $\mathcal{C}(P,\lambda)\simeq\mathcal{C}(3,r+1)$, $0\leqslant r\leqslant 2$ (in this case $N(P,\Lambda)\simeq S^3$), 
\item or $\mathcal{C}(P,\lambda)\simeq\mathcal{C}(Q,\lambda')$ for~an~affinely independent coloring 
$\lambda'$ of a simple $3$-polytope $Q$ (in~this~case $N(P,\lambda)\simeq S^3$ if and only if $Q=\Delta^3$ and 
$r=3$, which is~equivalent to~the~fact that  $\mathcal{C}(P,\lambda)\simeq\mathcal{C}(3,r+1)$ and $r=3$).
\end{itemize} 
\end{proposition}
\begin{proof}
Indeed, Corollaries \ref{mcor1} and \ref{mcor2} imply that if $N(P,\lambda)$ is a $3$-$RHS$, 
then each facet of $\mathcal{C}(P,\lambda)$ is a disk and any two such disks either do not
intersect or intersect by~a~circle or~an~edge. Then by~the~Steinitz theorem 
either $\mathcal{C}(P,\lambda)\simeq\mathcal{C}(3,r+1)$ for $0\leqslant r\leqslant 2$,
or $\mathcal{C}(P,\lambda)\simeq\partial Q$ for~a~simple $3$-polytope $Q$ with an~induced 
affinely independent coloring $\lambda'$.

On the other hand, Proposition \ref{prop:rhs} can be proved directly using Proposition \ref{pr:rhs}. 
Namely, for $r\leqslant 1$ it is clear. For $r\geqslant 2$ if a facet $G_i$ of $\mathcal{C}(P,\lambda)$ 
is a sphere with at least $2$ holes, then we can take an affine hyperplane in $\mathbb Z_2^r$
containing $\lambda_i$ but not $\lambda_j$ and $\lambda_k$ 
for~facets $G_j$ and $G_k$ lying in different holes to obtain a contradiction. 
If each facet of $\mathcal{C}(P,\lambda)$ is a disk 
and an intersection of two different facets $G_i$ and $G_j$ is a disjoint set of at least two edges, then one of these edges
intersects two additional facets $G_k$ and $G_l$. Then we can take an~affine hyperplane containing $\lambda_i$
and $\lambda_j$ but not $\lambda_k$ and $\lambda_l$ to obtain a contradiction.
\end{proof}
\begin{corollary}\label{RHSH'}
Let $\lambda$ be an~affine coloring of rank $r$ of a~simple $3$-polytope $P$.
If a~$3$-manifold $N(P,\lambda)$ is a~$3$-$RHS$, then for any subgroup $H'\subset\mathbb Z_2^r=H_0'$
the~space $N(P,\lambda)/H'=N(P,\lambda')$ is also a~$3$-$RHS$.
\end{corollary}
\begin{proof}

Indeed, affine hyperplanes $\pi'$ in $\mathbb Z_2^r/H'$ bijectively 
correspond to affine hyperplanes $\pi$ in~$\mathbb Z_2^r$ parallel to $H'$.
Then $\lambda_i'=\lambda_i+H\subset \pi'$ if and only if $\lambda_i\in\pi$. Moreover, 
$\bigcup\limits_{i\colon \lambda'_i\in \pi'}F_i=\bigcup\limits_{i\colon \lambda_i\in \pi }F_i$.
\end{proof}
\begin{remark}
Corollary \ref{RHSH'} also directly follows from Theorem \ref{th:transfer}. 
\end{remark}
\begin{example}
For $r=0$ we have $N(P,\lambda)\simeq S^3$ and the condition of Proposition \ref{pr:rhs} is trivial.

For $r=1$ Proposition \ref{pr:rhs} implies that $N(P,\lambda)$ is~a~$3$-$RHS$ if~and~only~if 
$\mathcal{C}(P,\lambda)\simeq \mathcal{C}(3,2)$. In this case $N(P,\lambda)\simeq S^3$.

For $r=2$ Propositions \ref{pr:rhs} and \ref{prop:rhs} imply that $N(P,\lambda)$ 
is~a~$3$-$RHS$ if~and~only~if either 
$\mathcal{C}(P,\lambda)\simeq \mathcal{C}(3,3)$ (in this case $N(P,\lambda)\simeq S^3$) or 
$\mathcal{C}(P,\lambda)\simeq\partial Q$ for a~simple $3$-polytope $Q$ 
with the induced affinely independent coloring $\lambda'$, and 
$\bigcup\limits_{i\colon \lambda'_i\in\pi}F_i'$  is~a~disk for any line in $\mathbb Z_2^2$.
There are six lines and each pair of parallel lines corresponds to~a~partition
of $\mathbb Z_2^2$ into two pairs of points such that for each pair the~union of~facets 
of~$Q$ of the corresponding colors is a disk. Moreover, each vertex of~$Q$ lies on~the~boundary 
of~each disk. Thus, taking into account item (2) of Theorem \ref{th:hyperell} we obtain the following result.
\begin{proposition}\label{RHSr2}
Let $\lambda$ be an~affine coloring of rank $2$ of a simple $3$-polytope $P$. Then $N(P,\lambda)$
is~a~$3$-$RHS$ if and only if one of the following equivalent conditions hold:
\begin{enumerate}
\item ether $\mathcal{C}(P,\lambda)\simeq \mathcal{C}(3,3)$ or 
$\mathcal{C}(P,\lambda)\simeq \partial Q$, where $Q$ is a simple $3$-polytope, and $\lambda$ is 
induced by~a~$3$-Hamiltonian cycle on it.
\item each nonzero involution in $\mathbb Z_2^2$ is hyperelliptic.
\end{enumerate}
\end{proposition}
In Fig. \ref{Hamspx}, \ref{Ham3pr}, and \ref{Hamdod} we show that the simplex $\Delta^3$, the $3$-prism $\Delta\times I$ 
and the dodecahedron admit a~$3$-Hamiltonian cycle. Examples of such polytopes are also 
shown in~Fig.~\ref{r4dodp11}.

\begin{figure}[h]
\begin{center}
\includegraphics[width=0.6\textwidth]{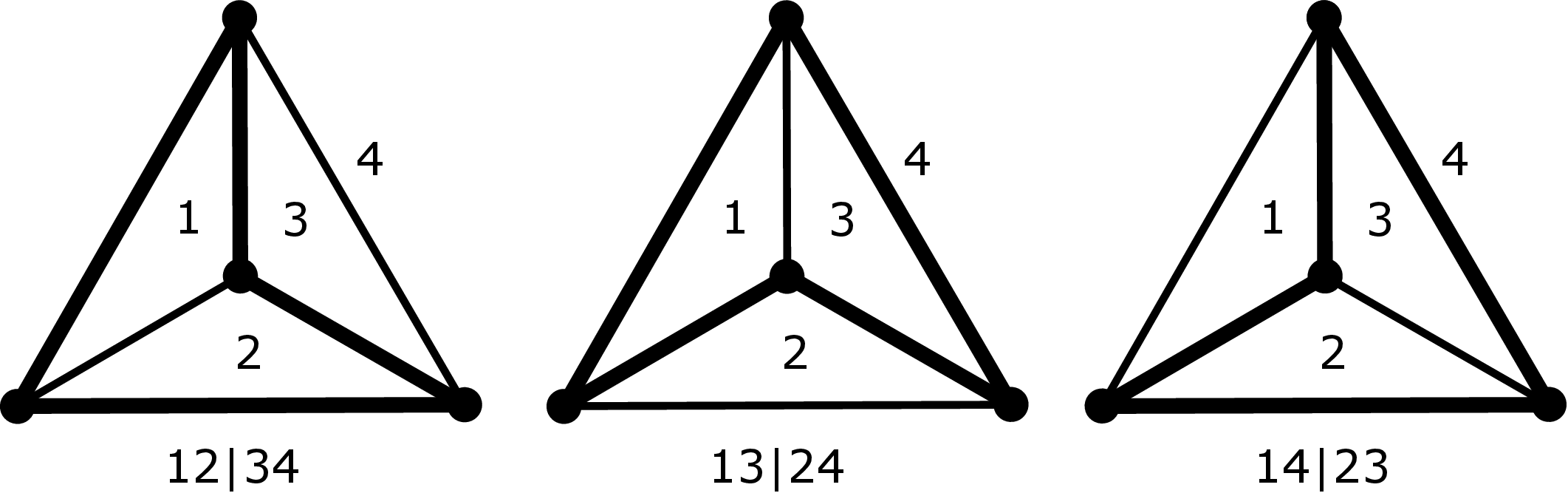}
\end{center}
\caption{Three consistent Hamiltonian cycles on the simplex}\label{Hamspx}
\end{figure}
\begin{figure}[h]
\begin{center}
\includegraphics[width=0.6\textwidth]{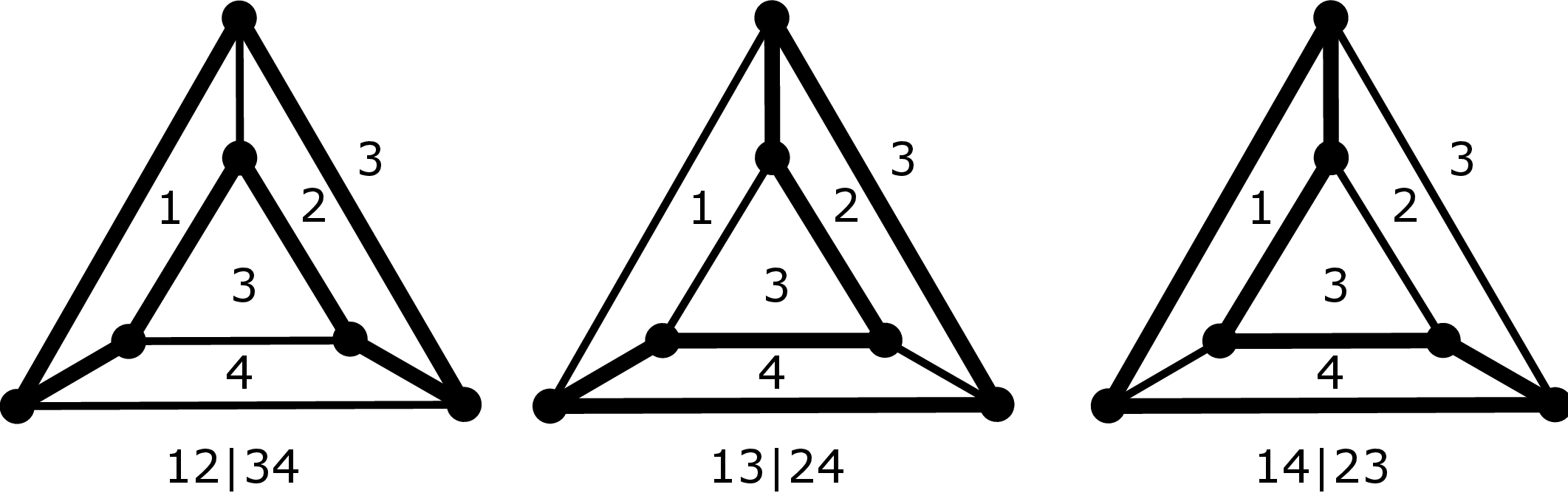}
\end{center}
\caption{Three consistent Hamiltonian cycles on the $3$-prism}\label{Ham3pr}
\end{figure}
\begin{figure}[h]
\begin{center}
\includegraphics[width=0.6\textwidth]{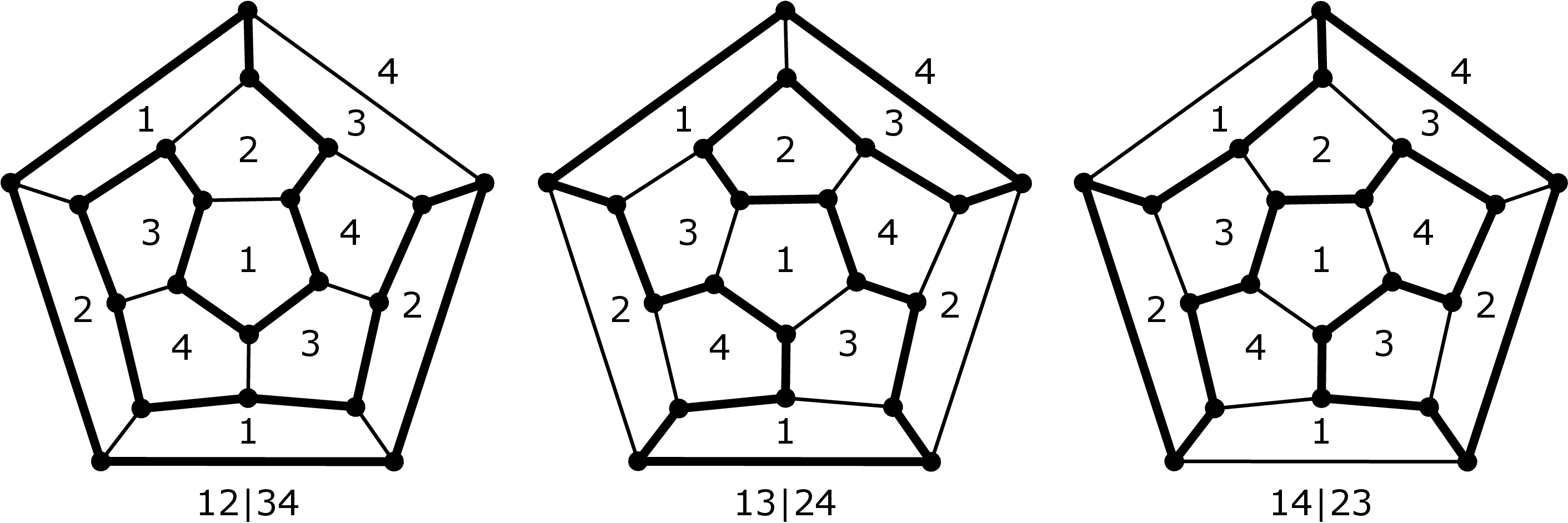}
\end{center}
\caption{Three consistent Hamiltonian cycles on the dodecahedron}\label{Hamdod}
\end{figure}

On the other hand, not any simple $3$-polytope admits a~$3$-Hamiltonian cycle. For example, the cube up to 
symmetries has only one Hamiltonian cycle drawn in Fig. \ref{Hamcube} on the left. If we draw the facets of the cube in four
colors using the Hamiltonian cycle  and group colors into pairs in three different possible ways, then we see that 
two partitions give Hamiltonian cycles and one partition gives two disjoint cycles. Thus, the $3$-cube does not admit
a~small cover that is a $3$-$RHS$.

\begin{figure}[h]
\begin{center}
\includegraphics[width=0.6\textwidth]{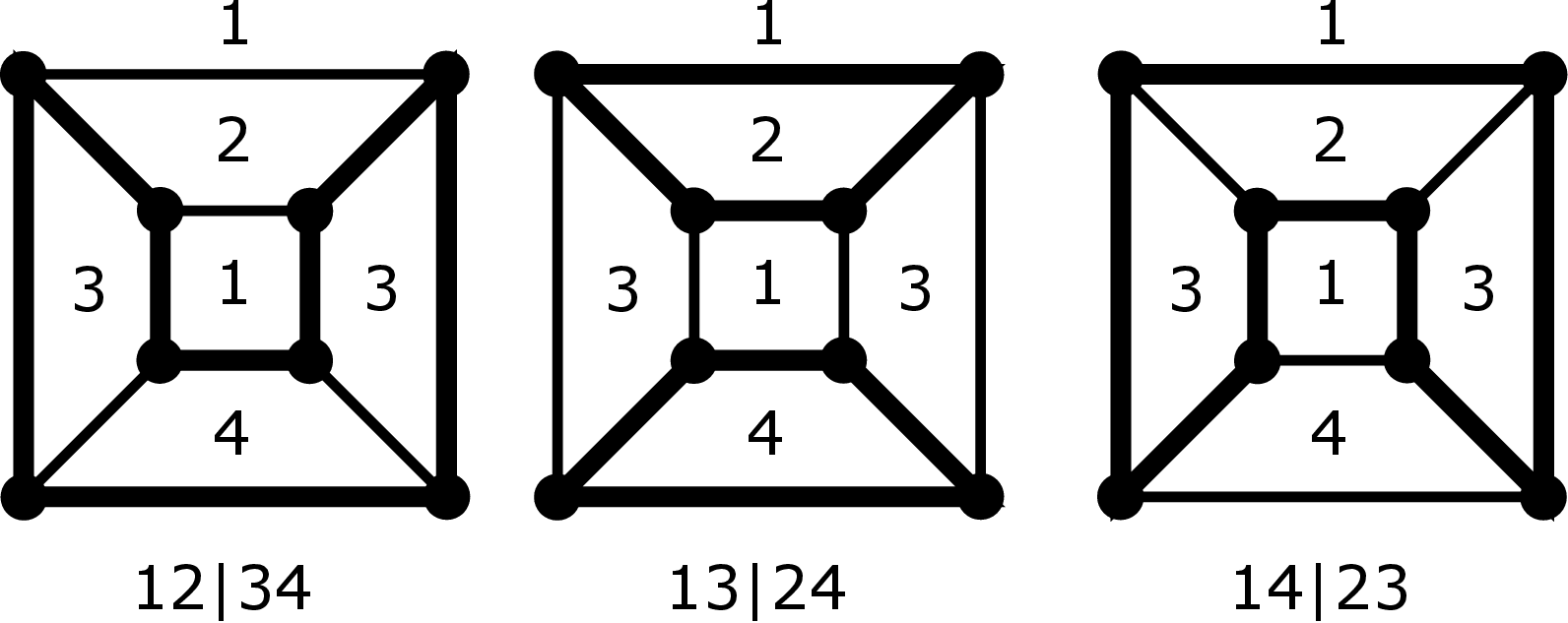}
\end{center}
\caption{The Hamiltonian cycle of the cube}\label{Hamcube}
\end{figure}

More details on simple $3$-polytopes admitting a~$3$-Hamiltonian cycle see in Section \ref{sec:3Ham}.

For $r=3$ Proposition \ref{pr:rhs} (applied for the point $\boldsymbol{p}=\boldsymbol{0}$) 
and Proposition \ref{prop:rhs} imply that $N(P,\lambda)$ is a $3$-$RHS$ if and only if 
either $\mathcal{C}(P,\lambda)\simeq \mathcal{C}(3,4)$ (in this case $N(P,\lambda)\simeq S^3$) or 
$\mathcal{C}(P,\lambda)\simeq\partial Q$ for~a~simple $3$-polytope $Q$ with the~induced affinely 
independent coloring $\lambda'$ such that
$\bigcup\limits_{i\colon \boldsymbol{a}(\lambda'_i)=0}F_i'$ is a disk for any vector 
$\boldsymbol{a}\in(\mathbb Z_2^3)^*\setminus\{\boldsymbol{0}\}$. 
For short we will identify the point $(x_1,x_2,x_3)\in\mathbb Z_2^3$ with the~number $4x_1+2x_2+x_3$ having
the~corresponding binary expression. The vectors $\boldsymbol{a}\in(\mathbb Z_2^3)^*\setminus\{\boldsymbol{0}\}$
correspond to partitions of $\mathbb Z_2^3$ into two parallel hyperplanes consisting of four points:
$$
\begin{tabular}{c|c|c}
(0,0,1)&0,2,4,6&1,3,5,7\\
(0,1,0)&0,1,4,5&2,3,6,7\\
(0,1,1)&0,3,4,7&1,2,5,6\\
(1,0,0)&0,1,2,3&4,5,6,7\\
(1,0,1)&0,2,5,7&1,3,4,6\\
(1,1,0)&0,1,6,7&2,3,4,5\\
(1,1,1)&0,3,5,6&1,2,4,7
\end{tabular} 
$$
An example of the cube with an affinely independent coloring of rank $3$ producing a~$3$-$RHS$
is~shown in~Fig. \ref{CubeR4}. It can be proved that up to a symmetry this 
is~a~unique affine coloring of~the~cube with these properties.

\begin{figure}[h]
\begin{center}
\includegraphics[width=\textwidth]{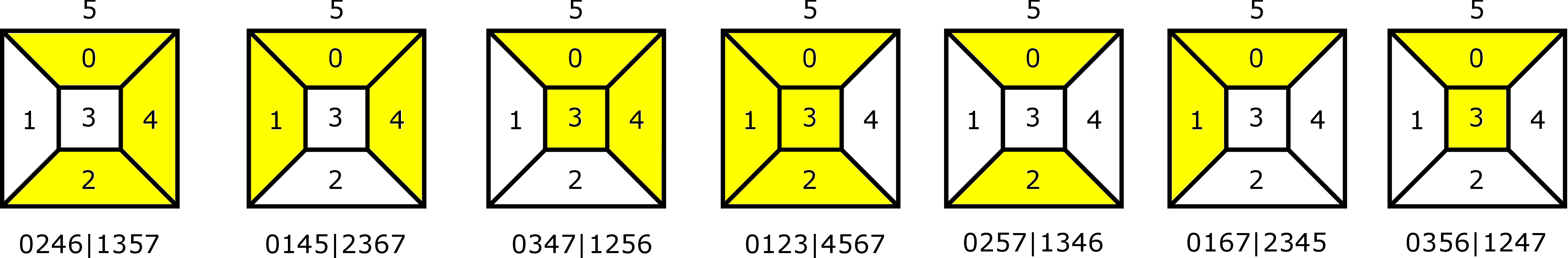}
\end{center}
\caption{The cube with an affine coloring of rank $3$ producing a~$3$-$RHS$}\label{CubeR4}
\end{figure}

An example of the $5$-prism with an affinely independent coloring of rank $3$ producing a~$3$-$RHS$
is~shown in~Fig. \ref{5prismR3}. 

\begin{figure}[h]
\begin{center}
\includegraphics[width=\textwidth]{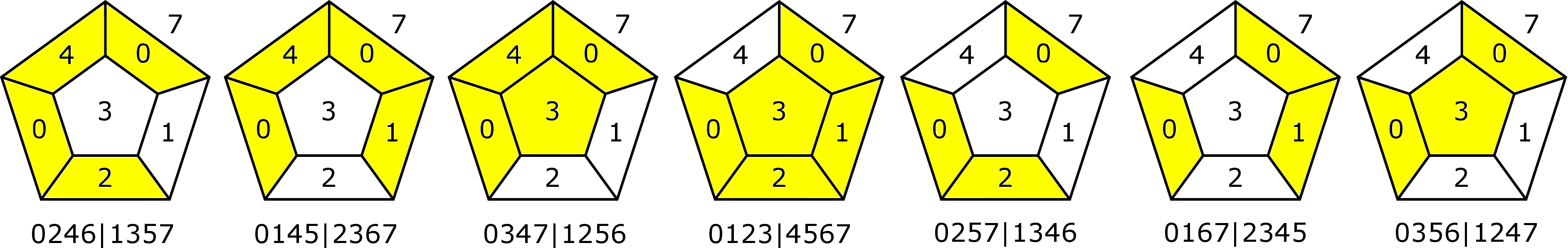}
\end{center}
\caption{The $5$-prism with an affine coloring of rank $3$ producing a~$3$-$RHS$}\label{5prismR3}
\end{figure}

An example of of~the~dodecahedron with an affinely independent coloring of rank $3$ producing 
a~$3$-$RHS$ is~shown in~Fig. \ref{r4dod}. In~Fig.~\ref{r4dodp11} we show its affine 
colorings of~rank $2$ corresponding to factorisations by $1$-dimensional subgroups in $\mathbb Z_2^3$.

\begin{figure}[h]
\begin{center}
\includegraphics[width=\textwidth]{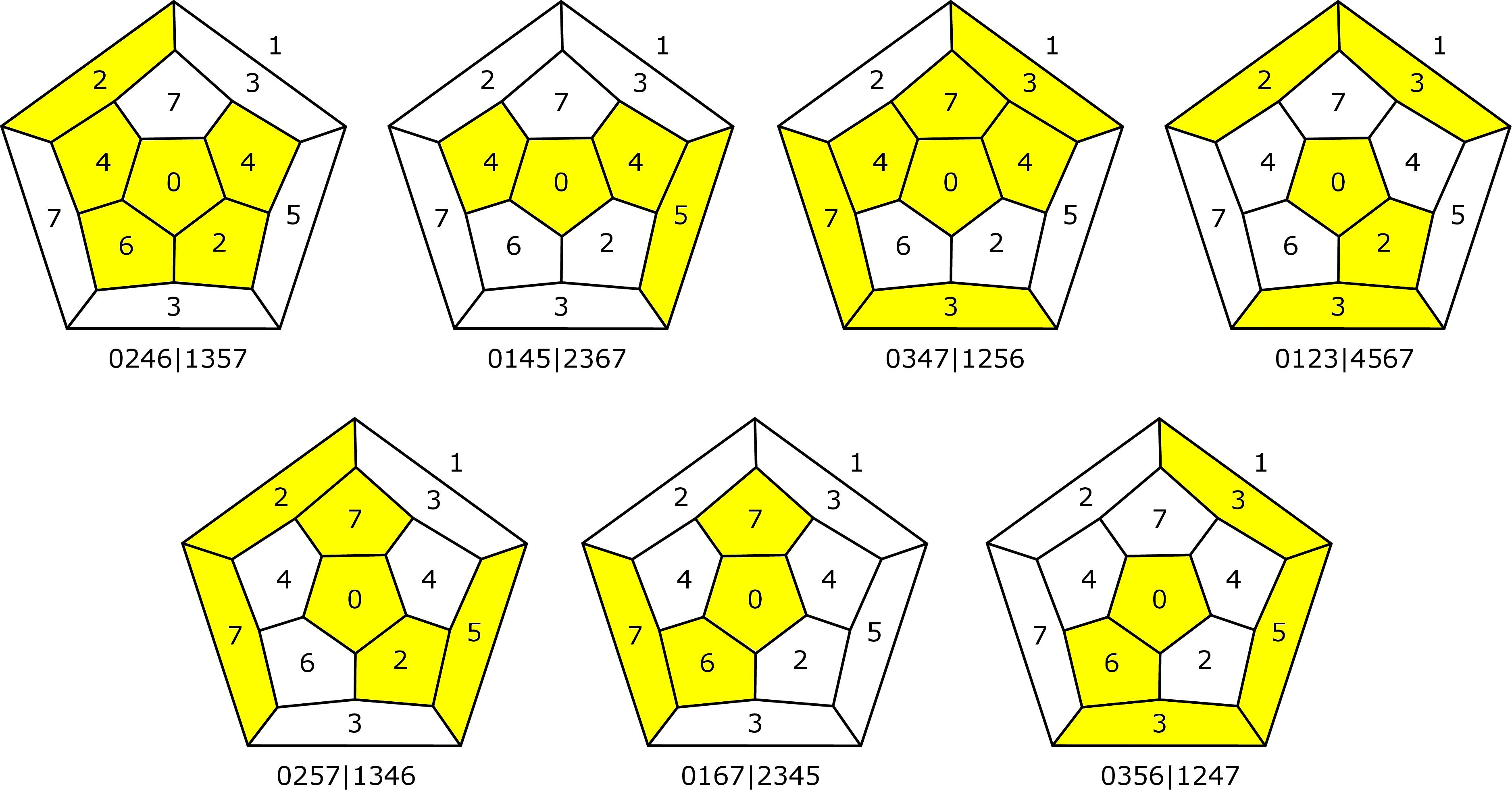}
\end{center}
\caption{The dodecahedron with an $8$-coloring producing a~$3$-$RHS$}\label{r4dod}
\end{figure}
\begin{figure}[h]
\begin{center}
\includegraphics[width=0.9\textwidth]{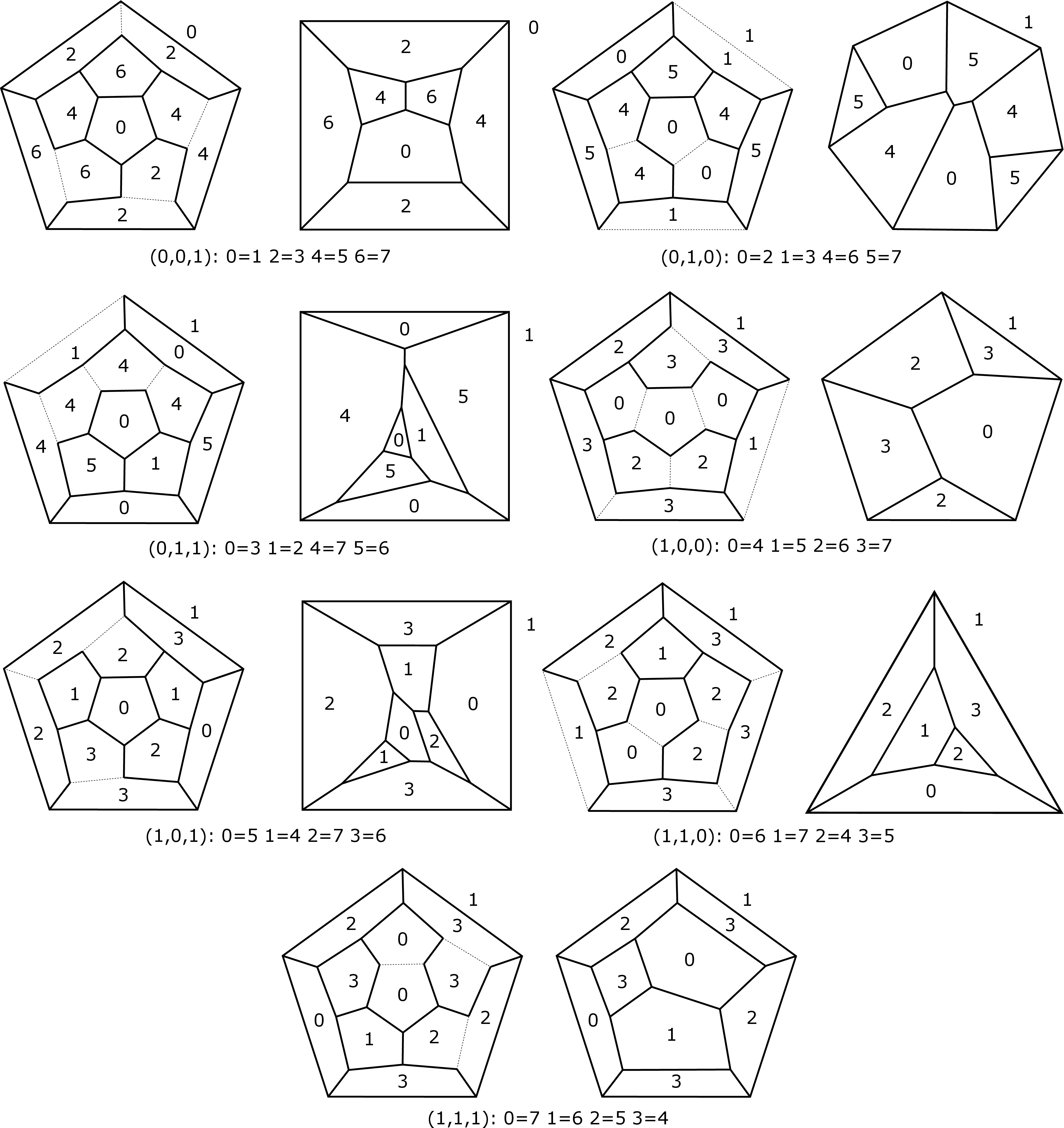}
\end{center}
\caption{The dodecahedron with $4$-colorings arising after factorisation of~the~$8$-coloring from 
Fig.~\ref{r4dod} by~$1$-dimensional subgroups in $\mathbb Z_2^3$.
Each subgroup is generated by a vector $\boldsymbol{x}\in\mathbb Z_2^3$ and gives the identification
$\lambda_i=\lambda_j$ if $\lambda_i+\boldsymbol{x}=\lambda_j$.}\label{r4dodp11}
\end{figure}
\end{example}

\begin{example}\label{ex:RHSGeom}
The simplex in~Fig.~\ref{Hamspx}, the~$3$-prism in~ Fig.~\ref{Ham3pr}, 
the~cube~in~Fig.~\ref{CubeR4}, the $5$-prism in~Fig.~\ref{5prismR3}, and the~dodecahedron in Fig.~\ref{Hamdod} 
and \ref{r4dod} give examples of manifolds that are $3$-$RHS$ and admit geometric structures
modelled on $\mathbb S^3$, $\mathbb S^2\times\mathbb R$, $\mathbb R^3$, $\mathbb H^2\times \mathbb R$, and 
$\mathbb H^3$ respectively. 
\end{example}

\section{Simple $3$-polytopes with $3$ consistent Hamiltonian cycles}\label{sec:3Ham}
\subsection{General facts}
In this section we will discuss simple $3$-polytopes $P$ admitting 
a~$3$-Hamiltonian cycle. Such a~cycle 
corresponds to~{\it $3$~consistent Hamiltonian cycles}, that is $3$ Hamiltonian cycles 
such that each edge of $P$ belongs to exactly two of them. 
This is exactly a~{\it Hamiltonian double cover} in terminology of the paper \cite{F06}. 
The graphs of such polytopes are {\it strongly Hamiltonian} 
in terminology of \cite{K63}, that is they are regular (all the vertices have equal degrees) and 
{\it perfectly $1$-factorable} (see Definition \ref{def:graph}). 
Each of the three consistent Hamiltonian cycles is a $3$-Hamiltonian cycle and defines the other two. 
In our paper three consistent Hamiltonian cycles arise in~the~classification~of 
\begin{enumerate}
\item hyperelliptic $3$-manifolds $N(P,\lambda)$ in Theorem \ref{th:hyperell}. They 
correspond to hyperelliptic manifolds $N(P,\lambda)$ with $\lambda$ of rank $2$ and $|I(\lambda)|=4$ 
having exactly three hyperelliptic involutions in $\mathbb Z_2^2$. 
\item rational homology $3$-spheres in Propositions \ref{pr:rhs} and  \ref{RHSr2}.
They correspond to rational homology $3$-spheres $N(P,\lambda)$ with $\lambda$ of rank $2$ and $|I(\lambda)|=4$ .
\end{enumerate}
\subsection{Polytopes without $3$ consistent Hamiltonian cycles}
In Section \ref{Sec:RHS} we showed that the simplex $\Delta^3$, the $3$-prism $\Delta\times I$ 
and the dodecahedron admit three consistent Hamiltonian cycles, and the cube $I^3$ does not admit.
It is not difficult to show that a situation similar to~the~case of~the~cube 
arises for all the $k$-prisms with $k\geqslant 5$. Namely, for $k$ odd
up~to~combinatorial symmetries  there is a unique Hamiltonian cycle shown in Fig. \ref{Hamkpr}. 
It exists for any~$k$. For $k$ even there is also the second Hamiltonian cycle shown in Fig. \ref{Ham2kpr}. 
Thus, $k$-prisms do not admit  small covers that are $3$-$RHS$ for $k\geqslant 4$.
\begin{figure}[h]
\begin{center}
\includegraphics[width=0.6\textwidth]{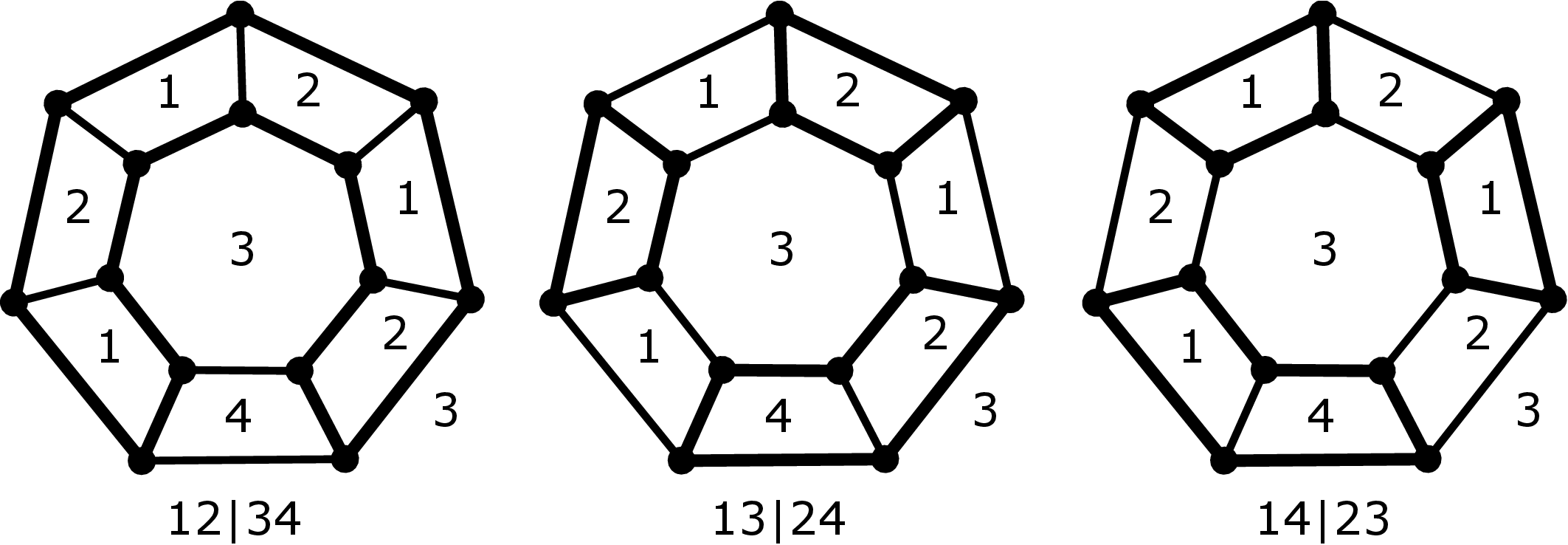}
\end{center}
\caption{A Hamiltonian cycle on the $k$-prism}\label{Hamkpr}
\end{figure}
\begin{figure}[h]
\begin{center}
\includegraphics[width=0.6\textwidth]{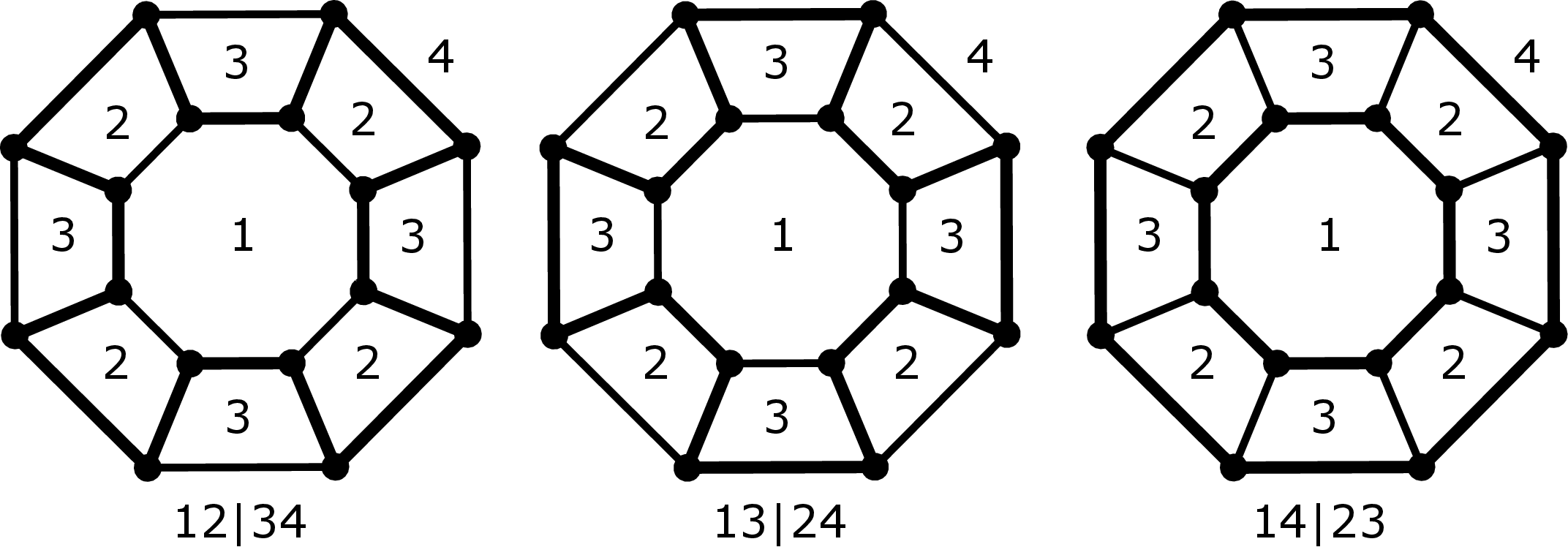}
\end{center}
\caption{A Hamiltonian cycle on the $2k$-prism}\label{Ham2kpr}
\end{figure}

Moreover, there is the following result generalizing the case of $(2k)$-prisms.
\begin{definition}
A graph $G$ is called {\it bipartite} if its vertices can be divided into two disjoint sets such that any edge
connects vertices from different sets.
\end{definition}
Any $(2k)$-prism has a bipartite graph. It is easy to see that if a simple $3$-polytope $P$ has a bipartite graph, then
any its facet has an even number of edges. The converse is also true. 
\begin{lemma}
A simple $3$-polytope $P$ has a bipartite graph if and only if any its facet has an~even number of edges.
\end{lemma}
\begin{proof}
One of the ways to prove the lemma  is to use the fact that any facet of a simple $3$-polytope $P$
has an even number of edges if and only if the facets of $P$ can be colored in $3$ colors
such that any two adjacent facets have different colors (see the proof in \cite{I01,J01}). 
Then the~vertices where the colors $1$, $2$, and $3$ follow each other clockwise and counterclockwise form  
the desired partition of the vertex set of the graph.
\end{proof}
\begin{theorem}\cite[Theorem 3]{K62}\label{thbip} 
If $G$ is a plane $3$-valent bipartite graph, then $G$ cannot possibly have a Hamiltonian double cover. 
\end{theorem}
\begin{corollary}
If a simple $3$-polytope $P$ has three consistent Hamiltonian cycles, then $P$ has a facet with an odd number of edges.
\end{corollary}
A short proof of Theorem \ref{thbip} was given in \cite[Theorem 12]{F06}. Is is based on two facts.
\begin{lemma}\cite[Remark 10]{F06}\label{lem:3HamQ}
Let $G$ be a connected $3$-valent planar graph. If it admits three consistent Hamiltonian cycles, then either $G$ is 
a~theta-graph or a~graph of a simple $3$-polytope.
\end{lemma}
\begin{proof}
Indeed, $G$ can not have loops. If $G$ has two edges connecting the same vertices, then one of the Hamiltonian cycles consists
of these two edges. Then $G$ has no other vertices and $G$ is the theta-graph. Thus, we can assume that the graph 
$G$ is simple. If the boundary cycle of some its facet is not simple, then there is a bridge which 
belongs to all the three Hamiltonian cycles. A contradiction.  
If the boundary cycles of two facets have in common two disjoint edges, 
then the~deletion of these edges makes the graph disconnected. Hence, all the three Hamiltonian cycles contain these edges,
which is a contradiction. Then the graph $G$ is simple and $3$-connected and by the Steinitz theorem it 
corresponds to a boundary of a simple $3$-polytope. 
\end{proof}
\begin{lemma}\cite[Remark 11]{F06}\label{lem:4gon}
If a simple $3$-polytope $P$ admits $3$ consistent Hamiltonian
cycles and $P$ has a quadrangular facet, then there is a pair of opposite edges of this facet such that the deletion 
of them produces the theta-graph or a graph of another simple $3$-polytope $Q$ with $3$ consistent Hamiltonian cycles.
\end{lemma}
\subsection{Reductions}
The reduction from Lemma \ref{lem:4gon} can be generalized as follows.
If a simple $3$-polytope $P$ has $3$ consistent Hamiltonian
cycles and a triangular facet, then this facet can be shrinked to a point to produce either the theta-graph or a graph of another 
simple $3$-polytope $Q$ with three induced consistent Hamiltonian cycles. More generally, if $P$ has a~$3$-belt,
that is a triple of facets $(F_i, F_j,F_k)$ such that any two of them are adjacent and $F_i\cap F_j\cap F_k=\varnothing$,
then $P$ can be cut along the triangle with vertices at midpoints of $F_i\cap F_j$, $F_j\cap F_k$ and $F_k\cap F_i$,
and each arising triangle can be shrinked to a point to produce  two simple $3$-polytopes $Q_1$ and $Q_2$ such that
$P$ is a connected sum of $Q_1$ and $Q_2$ at vertices. Then $P$ has $3$ consistent Hamiltonian cycles if and only if
$Q_1$ and $Q_2$ both have this property. 

If $P$ has a $4$-belt,  that is a cyclic sequence of facets $(F_i,F_j,F_k,F_l)$ such that the facets are adjacent if and only if
they follow each other, then combinatorially $P$ 
can be similarly cut along this belt to two simple polytopes $Q_1$ and $Q_2$ such that
$P$ is a connected sum of $Q_1$ and $Q_2$ along quadrangles (details see in~\cite{E22M}). 
It turns out that there can be $Q_1$ and $Q_2$ both
admitting no $3$-Hamiltonian cycles such that $P$ admits. The example is given by 
the connected sum of  two $5$-prisms along quadrangles such that the prisms are ``twisted'': base facets of one prism
correspond to side facets of the other. We proved above that $5$-prisms does not admit $3$ consistent Hamiltonian cycles,
while the resulting polytope admits, as it is shown on Fig. \ref{P8Ham}. 
\begin{figure}[h]
\begin{center}
\includegraphics[width=0.7\textwidth]{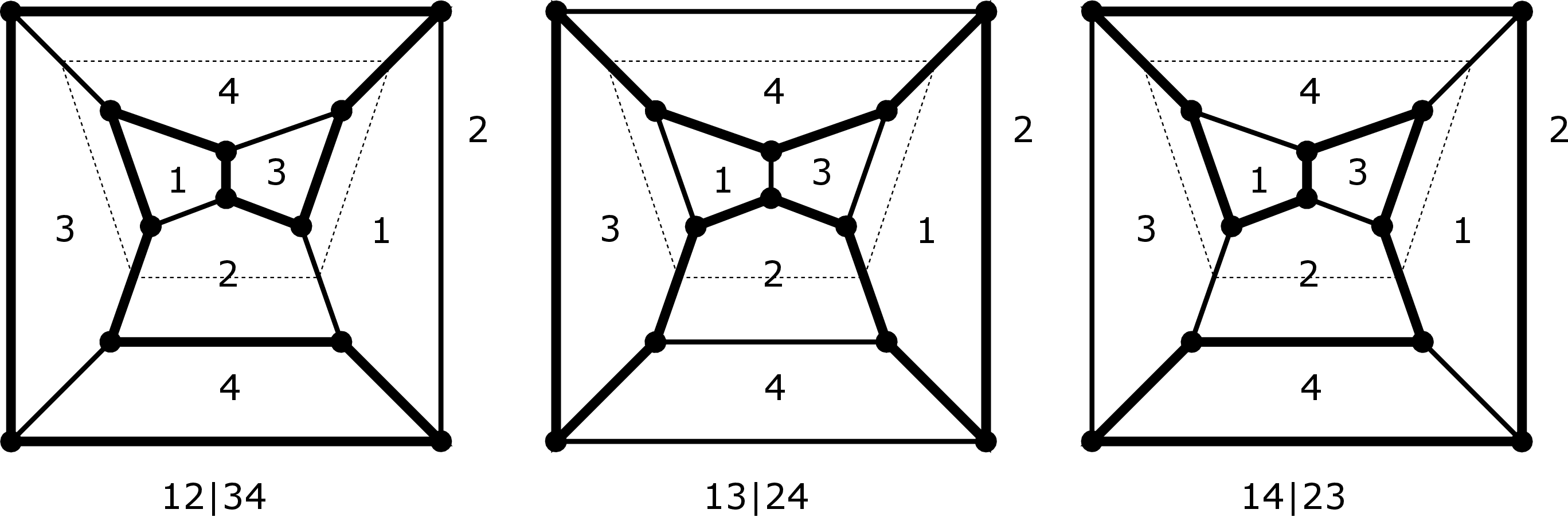}
\end{center}
\caption{Three consistent Hamiltonian cycles on the connected sum of two $5$-prisms along quadrangles}\label{P8Ham}
\end{figure}

\begin{problem}
To find a set of reductions and a set of initial polytopes such that any simple $3$-polytope $P$
with a $3$-Hamiltonian cycle can be reduced to an initial polytope by a sequence of~these reductions
in such a way that all intermediate polytopes also have a $3$-Hamiltonian~cycle.
\end{problem}

\subsection{Fullerenes} 
{\it Fullerenes} are simple $3$-polytopes with all facets pentagons and hexagons. They model spherical carbon molecules.
As was shown by F.~Kardo\v{s} in \cite{K14} any fullerene admits a Hamiltonian cycle 
(it is not valid for all simple $3$-polytopes, see \cite{T46, G68}). The simplest fullerene is the dodecahedron. 
As we have shown above it admits $3$ consistent Hamiltonian cycles. 
The next fullerene is the {\it $6$-barrel} shown in Fig.~\ref{6barrelHam}. 
It is also known as a {\it L\"obell polytope} $L(6)$ (see \cite{V87}). 
Using the fact that locally near any $6$-gon a Hamiltonian cycle has one of~the~types shown in~Fig.~\ref{6gonHam} 
\begin{figure}[h]
\begin{center}
\includegraphics[width=0.6\textwidth]{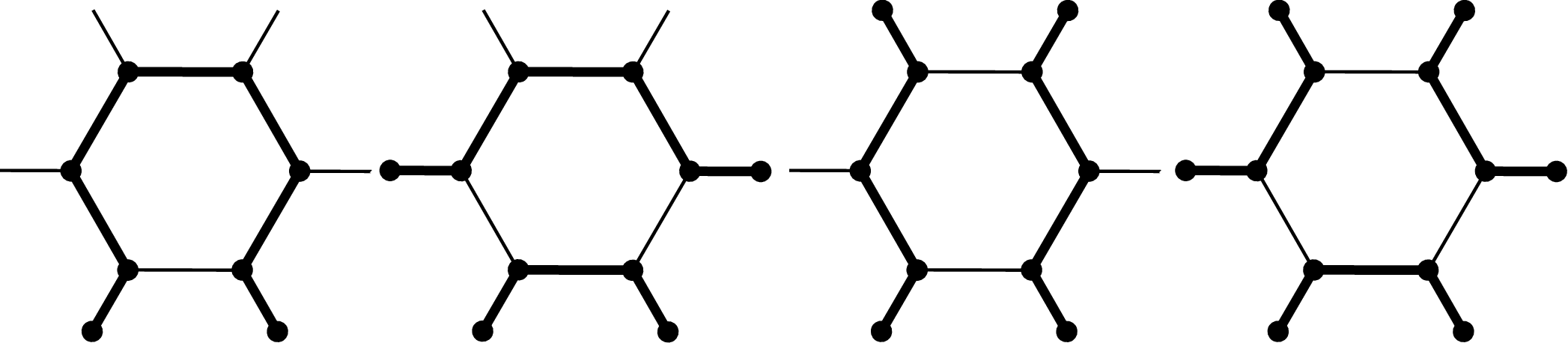}
\end{center}
\caption{Local forms of a Hamiltonian cycle near a $6$-gon}\label{6gonHam}
\end{figure}
it is easy to see that up to combinatorial symmetries the $6$-barrel has only 
four Hamiltonian cycles  shown in Fig.~\ref{6barrelHam}. 
\begin{figure}[h]
\begin{center}
\includegraphics[width=\textwidth]{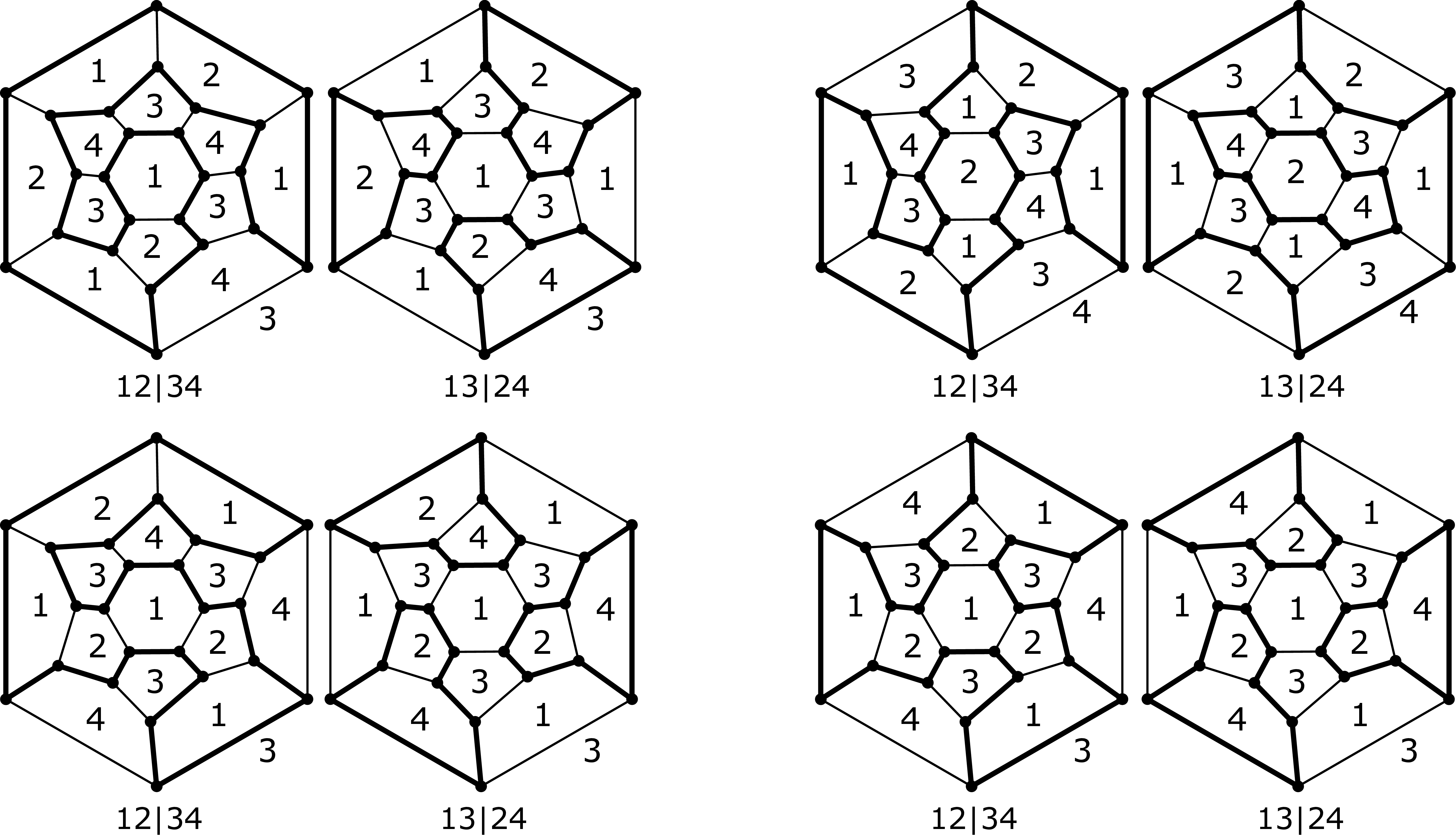}
\end{center}
\caption{Hamiltonian cycles on the $6$-barrel}\label{6barrelHam}
\end{figure}
Each of these cycles can not be included to~the~triple of~consistent Hamiltonian cycles. Thus, the $6$-barrel does not admit 
$3$ consistent Hamiltonian cycles.

\section{Acknowledgements}
The author is grateful to Victor Buchstaber for bringing him to science, for energy and permanent attention.

The author is grateful to Dmitry Gugnin for the introduction to~the~theory of~actions of~finite groups
on manifolds and for fruitful discussions. These discussions lead to the formulation and proof of Theorem \ref{th:ZHM}
and Example \ref{ex:spnG},  and Proposition \ref{prop:spn+r} and Example \ref{ex:spn+k+1G}. The author
is~also grateful to~Vladimir Shastin for the~idea to~consider $3$-manifolds $N(P,\Lambda)$ 
that are~rational homology $3$-spheres, to Alexei Koretskii for building an example of a $4$-dimensional
hyperelliptic small cover, and to Leonardo Ferrari for useful comments on the text.

\end{document}